\def\draft{n}
\documentclass[12pt]{amsart}
\usepackage[headings]{fullpage}
\usepackage{amssymb,epic,eepic,epsfig,amsmath,amsbsy,amsmath,bbm}
\usepackage{graphicx}
\usepackage{url}
\usepackage{amsbsy,amsmath}
\usepackage{latexsym,amsfonts}
\usepackage{amsthm, graphics, graphicx, subfigure, bbm, stmaryrd,stackrel}
\usepackage[bookmarks=true,%
    colorlinks=true,%
    linkcolor=blue,%
    citecolor=blue,%
    filecolor=blue,%
    menucolor=blue,%
    urlcolor=blue,%
    breaklinks=true]{hyperref}

\newtheorem{theorem}{Theorem}[section]
\newtheorem{proposition}{Proposition}[section]
\theoremstyle{definition}
\newtheorem{lemma}[proposition]{Lemma}
\newtheorem{definition}[proposition]{Definition}
\newtheorem{remark}[proposition]{Remark}
\newtheorem{corollary}[proposition]{Corollary}

\newtheorem{problem}[proposition]{Problem}
\newtheorem{example}[proposition]{Example}

\def\printname#1{
        \if\draft y
                \smash{\makebox[0pt]{\hspace{-0.5in}
                        \raisebox{8pt}{\tt\tiny #1}}}
        \fi
}

\newcommand{\psdraw}[2]
         {\begin{array}{c} \hspace{-1.3mm}
        \raisebox{-4pt}{\epsfig{figure=draws/#1.eps,width=#2}}
        \hspace{-1.9mm}\end{array}}

\newlength{\standardunitlength}
\catcode`\@=11
\long\def\@makecaption#1#2{%
     \vskip 10pt

\setbox\@tempboxa\hbox{%\ifvoid\tinybox\else\box\tinybox\fi
       \small\sf{\bfcaptionfont #1. }\ignorespaces #2}%
     \ifdim \wd\@tempboxa >\captionwidth {%
         \rightskip=\@captionmargin\leftskip=\@captionmargin
         \unhbox\@tempboxa\par}%
       \else
         \hbox to\hsize{\hfil\box\@tempboxa\hfil}%
     \fi}
\font\bfcaptionfont=cmssbx10 scaled \magstephalf
\newdimen\@captionmargin\@captionmargin=2\parindent
\newdimen\captionwidth\captionwidth=\hsize
\catcode`\@=12

\def\lbl#1{\label{#1}\printname{#1}}

\def\BN{\mathbb N}

\def\BQ{\mathbb Q}
\def\BR{\mathbb R}
\def\BC{\mathbb C}

\def\calA{\mathcal A}

\def\calF{\mathcal F}

\def\calG{\mathcal G}

\def\calF{\mathcal F}
\def\calV{\mathcal V}
\def\calW{\mathcal W}

\def\calH{\mathcal H}
\def\calI{\mathcal I}

\def\a{\alpha}

\def\La{\Lambda}
\def\l{\lambda}

\def\ga{\gamma}
\def\calS{\mathcal{S}}

\def\b{\beta}
\def\th{\theta}

\def\ti{\widetilde}

\def\longto{\longrightarrow}

\def\ev{\mathrm{ev}}
\def\Tr{\mathrm{Tr}}

\def\deg{\mathrm{deg}}

\def\longto{\longrightarrow}
\def\calA{\mathcal A}

\def\calF{\mathcal F}
\def\calH{\mathcal H}

\def\calV{\mathcal V}
\def\calZ{\mathcal Z}
\def\calW{\mathcal W}
\def\calR{\mathcal R}

\def\SL{\mathrm{SL}}

\def\be{\begin{equation}}
\def\ee{\end{equation}}

\def\ev{\mathrm{ev}}

\newcommand{\field}[1]{\mathbb{#1}}
\newcommand{\R}{\field{R}}
\newcommand{\N}{\field{N}}

\newcommand{\C}{\field{C}}
\newcommand{\Z}{\field{Z}}

\newcommand{\supp}{\mathrm{supp}}

\newcommand{\mat}[2][rrrrrrrrrrrrrrrrrrrrrrrrrrrrrrrrrrrrrrrrrrrrrrrrrrr]{\left[ \begin{array}{#1} #2 \\ \end{array}\right]}
\newcommand{\met}[2][ccccccccccccccccccccccccccccccccc]{\left[ \begin{array}{#1} #2 \\ \end{array}\right]}

\begin{document}

%%%%%%%%%%%%%%%%%%%%%%{page1}

\title[Analyticity of the planar limit of a matrix model]{
Analyticity of the planar limit of a matrix model}
\author{Stavros Garoufalidis}
\address{School of Mathematics \\
         Georgia Institute of Technology \\
         Atlanta, GA 30332-0160, USA \newline
         {\tt \url{http://www.math.gatech.edu/~stavros}}}
\email{stavros@math.gatech.edu}
\author{Ionel Popescu}
\address{School of Mathematics \\
         Georgia Institute of Technology \\
         Atlanta, GA 30332-0160, USA \newline
         {\tt \url{http://www.math.gatech.edu/~ipopescu}} \\
Institute of Mathematics of Romanian Academy \\
21, Calea Grivitei Street\\
010702-Bucharest, Sector 1\\
ROMANIA}         
\email{ipopescu@math.gatech.edu}

\thanks{S.G. was supported in part by by NSF.   
I.P. was supported in part by the Marie Curie Action grant nr. 249200.
\newline
1991 {\em Mathematics Classification.} Primary 57N10. Secondary 57M25.
}

\date{April 9, 2012}%\today }

%\dedicatory{\large{\bf Preliminary version. Private copy.}}

\begin{abstract}
Using Chebyshev polynomials combined with some mild combinatorics,  we provide an alternative approach to the analytical and formal planar limits of a random matrix model with a one-cut potential $V$.  For potentials  $V(x)=x^{2}/2-\sum_{n\ge1}a_{n}x^{n}/n$, as a power series in all $a_{n}$, the formal Taylor expansion of the analytic planar limit is exactly the formal planar limit.  In the case $V$ is analytic in infinitely many variables $\{a_{n}\}_{n\ge1}$ (on the appropriate spaces), the planar limit is also an analytic function in infinitely many variables and we give quantitative versions of where this is defined.

Particularly useful in enumerative combinatorics are the gradings of $V$, $V_{t}(x)=x^{2}/2-\sum_{n\ge1}a_{n}t^{n/2}x^{n}/n$ and $V_{t}(x)=x^{2}/2-\sum_{n\ge3}a_{n}t^{n/2 -1}x^{n}/n$.   The associated planar limits $F(t)$ as functions of $t$ count planar diagram sorted by the number of edges respectively faces.  
We point out a method of computing the asymptotic of the coefficients of $F(t)$ using the combination of the \emph{wzb} method and the resolution of singularities.   This is illustrated in several computations revolving around the important extreme potential  $V_{t}(x)=x^{2}/2+\log(1-\sqrt{t}x)$ and its variants.   This particular example gives a quantitative and sharp answer to a conjecture of t'Hooft's which states that if the potential is analytic, the planar limit is also analytic.

%In this paper we study the analyticity of the planar limit of a potential
%using both the formal matrix models and the analytic matrix models approach.  
%The latter one is used in deducing a new formula for the planar limit using 
%just real analysis combined with manipulations of Chebyshev polynomials and 
%elementary combinatorics.   We prove then that the formal planar limit 
%coincides with the Taylor series expansion of the analytic planar limit.
%Our formula for the planar limit allows us to identify it exactly with
%an algebraic function in the case of various extreme potentials, and
%to give exact recursions and asymptotics of its coefficients. Finally,
%we show that our analyticity results are sharp.
\end{abstract}

\maketitle

\tableofcontents

%%%%%%%%%%%%%%%%%%%%%%%%%%%%%%%%%%%%%%%%%%%%%%%%%%%%%%%%%%%%%%%%%%%%%%%%%%%%%
%%%%%%%%%%%%%%%%%%%%%%%%%%%%%%%%%%%%%%%%%%%%%%%%%%%%%%%%%%%%%%%%%%%%%%%%%%%%%

\section{Introduction}
\lbl{sec.intro}

\subsection{Formal Matrix Models and Their Planar Limit}
\lbl{sub.fmm}

Matrix models are integrals of exponentiated potential functions
over finite dimensional vector spaces
(such as the vector space of Hermitian matrices of size $N$) that were
studied in the seventies as an approximation of Quantum
Field Theory in a 0-dimensional space-time. Matrix models 
at fixed value of $N$ and their behavior when $N$ is large is useful
in a variety
of problems that include enumerative problems of ribbon graphs,
random two-dimensional gravity, triangulations of surfaces, random matrices,
topological string theory, intersection theory on the moduli space of
curves and perturbative gauge theory; 
see \cite{tH,BIZ,DGZ,DV1,DV2,M0,M2,Me}. 

Matrix models come in two flavors: formal and analytic. {\em Formal
matrix models} (FMM in short) are easy to define, using formal Gaussian 
integration. The input of a formal matrix model is a {\em formal potential} 
$\calV$ 
\begin{equation}
\lbl{eq.V}
\calV(x)=\frac{x^2}{2} - \sum_{n=1}^\infty \frac{a_n}{n} x^n \in \calA[[x]],
\end{equation}
which lies in a formal power series ring $\calA[[x]]$, where $\calA$
is the completed ring
\begin{equation}
\lbl{eq.calA}
\calA=\BQ[[a_1,a_2,a_3,\dots]].
\end{equation}
The {\em partition function} $\calZ$ 
and the {\em free energy} $\calF$ of the formal matrix model is given by 
the following formal integral and its logarithm, respectively

\begin{equation}
\lbl{eq.calZ}
\calZ=\frac{\int_{\calH_N} dM \exp(-N \Tr (\calV(M)))}{
\int_{\calH_N} dM \exp(- N \Tr (M^2/2))},
\qquad
\calF=\log \calZ \in N^2\calA[[1/N^2]]
\end{equation}
where 

\begin{itemize}
\item
$\calH_N$ is the vector space of Hermitian matrices of size $N$, 
\item
$\Tr(M)$ denotes the trace of a matrix $M$,  
%\item
%$\calA=\BQ[[a_1,a_2,\dots]]$ is a graded formal power series ring where 
%$\text{deg}(a_k)=k$ and
%\item
%$\calV$ is the {\em potential} given by
%\begin{equation}
%\lbl{eq.V}
%\calV(x)=\sum_{k=1}^\infty \frac{a_k}{k} x^k \in \calA[[x]]
%\end{equation}
\item
The meaning of the formal integration is the following: expand 
$e^{-N \Tr(\calV(M)+M^2/2)}$ as formal power series in 
$\calA[[N,\Tr(M),\Tr(M^2),\dots]]$ and integrate coefficient-wise. 
This operation produces a well-defined element of $N^2\calA[[1/N^2]]$.
\end{itemize}
So, we can write

\begin{equation}
\lbl{eq.calfreeg}
\calF=\sum_{g=0}^\infty N^{2-2g} \calF_g, 
\qquad \calF_g \in \calA.
\end{equation}
$\calF_g \in \calA$ is called the {\em genus $g$-limit} of the formal
matrix model. We can expand $\calF_g$ in terms of monomials 
$$
\calF_g=\sum_{\l} c_{\l,g} a_{\l}
$$
where the sum is over the set of all {\em partitions} $\l=(1^{n_1}2^{n_2}\dots)$,
and $a_{\l}=\prod_j a_j^{\l_j}$ and $c_{\l}$ are rational numbers. 
$\calF_g$ enumerates connected ribbon graphs of arbitrary valency
on a connected, oriented surface of genus $g$;  see \cite{BIZ,BIPZ,LZ,Po,LZ}. 
More precisely, it follows by {\em Wick's theorem} that $c_{\l}$ is
the weighted sum of all connected ribbon graphs 
(weighted by the inverse of the order of the automorphism group) of genus $g$
that have $n_k$ vertices of valency $k$; see \cite{Po,LZ}.
When $g=0$, $\calF_0$ is the {\em planar limit}
of the formal matrix model.  
The planar limit depends on the formal potential $\calV$,
and if we want to stress this dependence, we will use the notation
$\calF_{0,\calV}$. As an example, when
$$
\calV_4=\frac{x^2}{2}-\frac{a_4}{4} x^4
$$
the coefficients of the formal power series 
$\calF_{0,\calV_4} \in \BQ[[a_4]]$ counts the weighted sum 
 of connected planar 4-valent ribbon graphs. From the definition of $\mathcal{F}_{0,\calV_4}$, one can
compute several terms of the power series $\mathcal{F}_{0,\calV_4}$, 
by hand or by machine. The pioneering work of 
\cite{BIZ,BIPZ} gave an exact formula for the power series 
$\mathcal{F}_{0,\calV_4}$ 
using {\em potential theory}: 
\[
\mathcal{F}_{0,\calV_4}=\frac{1-36a_{4}+162a_{4}^{2}+(1-30a_{4})
\sqrt{1-12a_{4}}}{432a_{4}^{2}}
+\frac{1}{2}\log\left(\frac{1-\sqrt{1-12a_{4}}}{6a_{4}}\right) \in \BQ[[a_4]]
\]
The computation of \cite{BIZ,BIPZ} was lacking rigor, and several years later
their method was justified by using potential theory and the Riemann-Hilbert
method; see \cite{EM,DKM}. In the present paper, we give an independent proof,
using mostly techniques from real analysis and elementary potential theory.
In addition, we describe explicitly the analyticity properties of $\calF_0$ with
sharp results, see Theorems \ref{thm.1} and \ref{thm.2} below.

As a notational convention, we will use caligraphic symbols
$\calV,\mathcal{R},\mathcal{S},\calF_{0},\dots$ for formal matrix models and 
straight
symbols as $V,R,S,F_{0}\dots$ for the analytic matrix models.

\subsection{Analytic Matrix Models and Their Planar Limit}
\lbl{sub.analyticamm}

Let us now define the {\em analytic matrix models} (AMM in short).
An {\em admissible potential} $V(x)$ is a function
$V: \BR \longto \BR$ which is lower-semicontinuous, 
and grows sufficiently at infinity, i.e., satisfies
\begin{equation}\lbl{e:V}
\lim_{|x|\to\infty}\frac{V(x)}{2\log|x|}>1.  
\end{equation}
For an analytic matrix model with 
an admissible potential $V$ define
\begin{eqnarray}
I_{V}&=&
-\lim_{N\to\infty}\frac{1}{N^{2}}\log \int_{\calH_{N}}\exp(-N\Tr(V(M)))dM
\\ \notag &=&
\inf_{\mu\in\mathcal{P}(\R)}\left\{\int V(x)\mu(dx)-\iint 
\log|x-y|\mu(dx)\mu(dy)\right\},
\end{eqnarray}
where $\mathcal{P}(\R)$ is the set of all probability measures on $\R$.  
The second equality in the above equation follows for example from
\cite{Deift1}, \cite{Jo}. 

In the case $V(x)=\frac{x^{2}}{2}-\sum_{n\ge1}\frac{a_{n}x^{n}}{n}$ is an 
admissible potential we then define the \emph{analytic planar limit} as 
\begin{equation} \lbl{eq.IV}
F_{0,V}=\frac{3}{4}-I_{V}.  
\end{equation}
We will call $F_{0,V}$ and $I_{V}$ the \emph{the analytic planar limit}.

As we already mentioned, this formula allows one to reduce the problem of 
the planar limit to the 
investigation of what is known in the literature as the logarithmic potential 
with external fields.   

A {\em 1-cut potential} $V$ is an admissible potential whose equilibrium
measure has support in a {\em single interval} $[b-2c,b+2c]$.  Following 
the notation of \cite{BDG}, we will use the change of variables
\begin{equation}
\lbl{eq.bcRS}
(b,c^2)=(S,R).
\end{equation} 

%\begin{figure}[htbp] 
%   \centering
%   \includegraphics[width=3in]{interval} 
%   \caption{A graphical interpretation of $b,c$ and $R,S$ in terms of 
%   the endpoints of the support of the equilibrium measure.}
%   \label{graph}
%\end{figure}

\begin{figure}[htbp] 
$$
\psdraw{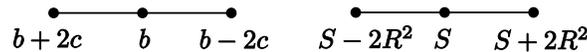}{3in}
$$
\caption{A graphical interpretation of $b,c$ and $R,S$ in terms of the 
endpoints of the support of the equilibrium measure.}\label{graph}
\end{figure}

It turns out that in the case $V$ is a 1-cut potential (plus some nondegeneracy), and $V_{a}$ is an 
analytic perturbation of $V$, then the endpoints and the planar limit 
depend analytically on $a$.  

An admissible potential $V$ is {\em even} if it satisfies $V(x)=V(-x)$
for all $x \in \BR$. For even 1-cut potentials, the equilibrium measure of 
$V$ is centered at $b=0$.

\subsection{Analyticity of the planar limit}
\lbl{sub.analyticity}
Analyticity of functions in infinitely many variables is well defined and 
understood on functions defined on  $\ell^{1}$ spaces (see \cite{Le} and 
\cite{Ry}).   In our case we need to define a weighted version of $\ell^{1}$ 
space.  To this end, let $r>0$ be a positive number, and for a complex-valued
sequence $\mathbf{a}=\{a_{n}\}_{n\ge1}\subset \C^{\BN}$, consider its
$\ell^{1}_{r}$ norm 
\begin{equation}
\lbl{eq.ell1norm}
||\mathbf{a}||_r=\sum_{n=1}^{\infty}|a_{n}|r^{n}.
\end{equation}
Now, consider the following $\ell^{1}$ type spaces 
\begin{align*}
\ell_{r}^{1}(\N)&=\{\mathbf{a}=\{a_{n}\}_{n\ge1}\subset \C^{\BN}: 
||\mathbf{a}||_r <\infty \} \\
\ell_{r}^{1}(2\N)&=\{\mathbf{a}=\{a_{n}\}_{n\ge1} \in \ell_{r}^{1}(\N):
a_{2n}=0,n\geq 1\}.
\end{align*}
Let $B_{r}$ and $B_{r}^{ev}$ denote the open balls of radius $1$ in 
$\ell_{r}^{1}(\N)$ and $\ell_{r}^{1}(2\N)$, respectively. 

Now consider $\mathfrak{S}\subset \R^{\N}$ to be the set of sequences 
$\mathbf{a}=\{ a_{n}\}_{n\ge1}\in \R^{\N}$ such that
\[
V(x)=\frac{x^{2}}{2}-\sum_{n\ge1}\frac{a_{n}x^{n}}{n}
\]
is a 1-cut admissible potential which is analytic near $0$.  
Using   Equation \eqref{eq.IV} we can define a map $F_{0,V}$
\begin{equation}
\lbl{eq.mapIV}
F_0 : \mathfrak{S} \longto \R.
\end{equation}
Likewise,  we have a map
$F_0^{\ev}: \mathfrak{S}^{ev} \longto \R$.

We use here the definition of \cite{Le,Ry} for an analytic function on 
$\ell^{1}_{r}(\N)$ which essentially means that the Taylor series in 
infinitely many variables converges. 

\begin{theorem}
\lbl{thm.1}
The maps $F_0$, $R$ and $S$ (resp. $F_0^{\ev}$, $R^{\ev}$ and $S^{\ev}$) 
uniquely extend to analytic functions on $B_{1/\sqrt{12}}$ 
(resp. $B^{ev}_{1/\sqrt{8}}$). 
\end{theorem}

Our next theorem identifies the planar limit of the formal and analytic
matrix model. Since the map $F_0$ from \eqref{eq.mapIV} is analytic 
at $0 \in \ell^1_{r}(\N)$, its Taylor series regarded as a formal power 
series is given by
\begin{equation}
\lbl{eq.TI}
F_{0}=\sum_{\l} c_\l a_\l \in \calA
\end{equation}
where the sum is over the set of partitions (including the empty one), 
$c_{\l} \in \BQ$ and $\calA$ is given in \eqref{eq.calA}. Consider the formal 
power series  $(\calR,\calS) \in \calA^2$ defined in Section
\ref{sub.2gradings} below.

\begin{theorem}
\lbl{thm.2}
We have
\begin{equation}
\lbl{eq.RSF0}
R=\calR, \qquad S=\calS, \qquad F_{0}=\calF_0.
\end{equation}
\end{theorem}

What this means is that, if the analytical procedures are taken formally, 
one recaptures the planar limit of the formal matrix models.  

Theorem \ref{thm.1} and Theorem \ref{thm.2} confirm a conjecture of 't Hooft 
for the planar limit of matrix models. 't Hooft' s conjecture 
is motivated by perturbative gauge theory ideas whose Feynman diagrams
are ribbon graphs, and 
asserts that $\calF_0(\calV(x))$ should be an analytic function at $x=0$
when $\calV(x)$ is analytic at zero; \cite{tH}.  For a proof of t' Hooft's
conjecture for the case of Chern-Simons gauge theory, see \cite{GL}.

A natural problem is to extend Theorem \ref{thm.1} to all genera $g$. 

\begin{problem}
\lbl{prob.1}
Show that for all $g \geq 0$, $\calF_g$ (resp., $\calF^{\ev}_g$) is the Taylor 
series of an analytic function on $B_{1/\sqrt{12}}$ (resp.  $B_{1/\sqrt{8}}$).
\end{problem}
This may be achieved using \cite{ACKM,E1}.

\subsection{Two gradings for the planar limit}
\lbl{sub.2gradings}

The formal planar limit $\calF_0 \in \calA$ enumerates planar ribbon graphs 
of arbitrary valency, and it is closely related to two other formal 
power series $(\calR,\calS)$ which are uniquely determined by the
system of non-linear equations

\begin{equation}\lbl{eq.RS}
\begin{cases}
\calR=\calH_1(\calR,\calS) \\
\calS=\calH_2(\calR,\calS)
\end{cases}
\end{equation}
where 
\begin{eqnarray}
\lbl{eq.H1}
\calH_{1}(\calR,\calS) &=& 
1+ \sum_{n\ge1}a_{n}\sum_{j\ge1}{n-1\choose j-1}{n-j \choose j}\calR^{j}\calS^{n-2j} 
\\
\lbl{eq.H2}
\calH_{2}(\calR,\calS) &=& 
\sum_{n\ge1}a_{n}\sum_{j\ge0}{n-1\choose 2j}{2j \choose j}\calR^{j}\calS^{n-2j-1} 
\end{eqnarray}
Equation  \eqref{eq.RS} always has a unique solution
in $(\calR,\calS) \in \calA^2 $ that satisfies 
$R \in 1+ \calA^+$ and $S \in \calA^+$,
where $\calA^+$ are the formal power series in the variables $a_n$ with
no constant term. Moreover, it is easy to see that this unique formal solution
has integer coefficients. 

An enumerative interpretation of the coefficients of $(\calR,\calS)$
is given in \cite{BDG}, which in particular implies that they
are natural numbers.
An analytic interpretation of $(\calR,\calS)$ 
is that they determine the position of the interval of a 1-cut analytic
matrix model; see Section \ref{s:13}. 

Enumerative combinatorics dictates two gradings on the set of variables
$a_n$, the {\em edge grading} $\deg_e(a_n)$ and the 
{\em face grading} $\deg_f(a_n)$ defined by

\begin{equation}
\lbl{eq.2gradings}
\deg_e(a_n)=\frac{n}{2}, \qquad
\deg_f(a_n)=\frac{n}{2}-1.
\end{equation}
Given an element $\calH \in \calA$, let $\calH_e \in \calA[[t^{1/2}]]$ 
and $\calH_f$  denote the result of substituting $a_n$ by $a_n t^{n/2}$
and $a_n t^{n/2-1}$ respectively. For example, for the formal potential 
$\calV(x)$ from Equation \eqref{eq.V} we have
\begin{equation}
\calV_e(x)=\frac{x^2}{2} - \sum_{n=1}^\infty \frac{a_n}{n} t^{n/2} x^n \in 
\calA[[t^{1/2},x]],
\qquad
\calV_f(x)=\frac{x^2}{2} - \sum_{n=3}^\infty \frac{a_n}{n} t^{n/2-1} x^n \in 
\calA[[t^{1/2},x]]
\end{equation}
where in the latter we assume that $a_1=a_2=0$. Likewise, for
$\calR_e$, $\calR_e$ and $\calF_{0,e}$. 

Of course, when we set $t=1$ to $\calH_e$ or $\calH_f$, we recover $\calH$.
In particular,

\begin{equation}
\lbl{eq.Fef}
\calF_{0,e}(1)=\calF_{0,f}(1)=\calF_{0} \in \calA
\end{equation}

The next theorem gives a simple formula for $\calF_{0,e}$ in terms of
$\calR_e$ and $\calS_e$.  This appears in \cite{BI} but for polynomial potentials $\mathcal{V}$ and the proof in there uses orthogonal polynomials.

\begin{theorem}
\lbl{thm.Fe}
We have:
\begin{equation}\lbl{eq.Fe}
\calF_{0,e}(t)=
\frac{1}{t}\int_{0}^{t}\frac{(t-s)(2\calR_e(s)\calS^{2}_e(s)
+\calR^{2}_e(s)-1)}{2s}ds.
\end{equation}
%with $R(t)=c^{2}(t)$.  
It follows that
\begin{equation}
\lbl{eq.dFe}
(t^2\calF_{0,e}')'=\frac{2 \calR_e(t)\calS^{2}_e(t) +\calR^{2}_e(t)}{2}.
\end{equation}
where $f'$ indicates  the derivative with respect to $t$.
\end{theorem}

The next theorem gives a simple formula for $\calF_{0,f}$ in terms of
$\calR_f$ alone.

\begin{theorem}
\lbl{thm.Ff}
We have:
\begin{equation}\lbl{eq.Ff}
\calF_{0,f}(t)=
\frac{1}{t^{2}}\int_{0}^{t}(t-s)\log \calR_f(s)ds
\end{equation}
In particular
\begin{equation}
\lbl{eq.dFf}
(t^2\calF_{0,f})''=\log \calR_f(t).
\end{equation}
\end{theorem}

\begin{remark}
\lbl{rem.f1}
Given a potential $\mathcal{V}=\sum_{n\ge1}a_{n}x^{n}/n$ as a formal power 
series, from the potential theoretic approach one obtains that at the 
formal level,  $c$ and $b$ satisfy the system
\begin{equation}
\begin{split}\lbl{sys:1}
 2=\int_{-2}^{2}cx\calV'(cx+b)\frac{dx}{\pi \sqrt{4-x^{2}}}&= 
\sum_{n\ge1}a_{n}\sum_{j\ge1}{n-1\choose 2j-1}{2j \choose j}c^{2j}b^{n-2j} \\
0=\int_{-2}^{2}\calV'(cx+b)\frac{dx}{\pi \sqrt{4-x^{2}}}&= 
\sum_{n\ge1}a_{n}\sum_{j\ge0}{n-1\choose 2j}{2j \choose j}c^{2j}b^{n-2j-1}. 
\end{split}
\end{equation}
In the case $\mathcal{V}=x^{2}/2-\sum_{n\ge1}a_{n}x^{n}/n$, and $R=c^{2}$, $S=b$, 
one easily obtains the system \eqref{eq.RS}.  
\end{remark}

\begin{remark}
\lbl{rem.Vef}  
Some authors prefer to consider the following rescaling $\ti\calV_e$
of $\calV_{e}$
\begin{equation*}
\ti\calV_e = \frac{x^{2}}{2t}-\sum_{n\ge1}\frac{a_{n}x^{n}}{n}
\end{equation*}
Since $\calV_e(x)=\ti\calV_e(t^{1/2}x)$, it is easy to see that 
\begin{eqnarray*}
\tilde{c}(t) &=& \sqrt{t}c(t) \\
\tilde{b}(t) &=& \sqrt{t}b(t),
\end{eqnarray*}
where $\tilde{c}(t)$ and $\tilde{b}(t)$ are defined using the system 
\eqref{sys:1}. 
\end{remark}

\begin{remark}
\lbl{rem.Vef2}
Likewise, for the following rescaling $\ti\calV_f$ of $\calV_{f}$
\begin{equation*}
\ti\calV_f =\frac{x^{2}}{2}-\sum_{n\ge3}\frac{a_{n}x^{n}}{n}
\end{equation*}
we have $\calV_f(x)=\ti\calV_f(t^{1/2}x)/t$
which implies that (here $\tilde{c}(t)$ and $\tilde{b}(t)$ are constructed 
from $\tilde{\calV}(x)/t$)
\begin{eqnarray*}
\tilde{c}(t) &=& \sqrt{t}c(t) \\
\tilde{b}(t)&=&\sqrt{t}b(t). 
\end{eqnarray*}
\end{remark}

\begin{remark} 
\lbl{rem.evenV}
When $\calV$ is even, then $\calS=0$ and $\calR$
satisfies the implicit Equation
\begin{equation}
\lbl{eq.HR}
\calH(\calR)=1
\end{equation}
where
\begin{equation}
\lbl{eq.H}
\calH(x)
=x- \frac{1}{2}\sum_{n=3}^{\infty}a_{2n}x^{n}{2n \choose n}
=x-\sum_{n=3}^{\infty}a_{2n}x^{n}{2n \choose n-1}.
\end{equation}
This is indeed so because 
\[
\int_{0}^{\pi}\frac{x \calV'(xy)dy}{t\pi\sqrt{4-y^{2}}}
=\frac{1}{t}\left(x-\frac{1}{2}\sum_{n=3}^{\infty}a_{2n}x^{2n}{2n \choose n}\right).
\]
\end{remark}

\subsection{Algebricity, holonomicity and asymptotics  of the planar limit}
\lbl{sub.algF0}

In this section we discuss the algebricity of the planar limit. 
Let us recall first some well-known properties of algebraic functions
and the asymptotics of their Taylor coefficients. The reader
may consult \cite{vPS} and also \cite[Chpt.VII]{FS} for further details.
Computer implementations are available at \cite{DvH,Ka,Pt}.

An {\em algebraic function} $y=y(x)$ is one that satisfies a polynomial 
equation
$P(y,x)=0$ for some 2-variable polynomial with rational coefficients. 
Below, we will be interested in algebraic functions $y(x)$ which are
{\em regular} at $x=0$, i.e., they have a Taylor series expansion
\begin{equation}
\lbl{eq.fan}
y(x)=\sum_{n=0}^\infty a_n x^n
\end{equation}
The set of algebraic functions is a field, closed under differentiation
with respect to $x$. Algebraic functions are always {\em holonomic} i.e.,
they satisfy (regular singular)
linear differential equations with coefficients polynomials
in $x$ with rational coefficients. An algebraic function $y(x)$ gives
rise to a ramified $d$-sheeted covering $\BC \longto \BC$ with semisimple
local monodromy (with eigenvalues complex roots of unity) and global
monodromy a {\em finite} subgroup of $\SL(d,\BC)$. In other words,
an algebraic function $y(x)$ regular at $x=0$ can be uniquely analytically
continued as a multivalued analytic function on $\BC\setminus \La$, where
$\La$ is a finite set of algebraic numbers. In practice the analytic
continuation can be obtained via {\em Puiseux series}, and all local expansions
of $y(x)$ around a singularity $x \in \La$ are exactly computed by $y(x)$;
see for example \cite{DvH,Pt}. Since $y(x)$ is holonomic, it follows
that the sequence $(a_n)$ of its Taylor coefficients from \eqref{eq.fan}
is {\em holonomic} i.e., it satisfies a linear difference equation
with coefficients polynomials in $n$ with rational coefficients; see
\cite{Z}. To discuss the asymptotics of $(a_n)$ we need to recall what
is a sequence of Nilsson type, discussed in detail in \cite{Ga2}.

\begin{definition}
\lbl{def.nilsson}
We say that a sequence $(a_n)$ is of Nilsson type if it has
an asymptotic expansion of the form:
\begin{equation}
\lbl{eq.asan}
a_n \sim_{n \to \infty} \sum_{\l,\a,\b}\lambda^{n} n^{\a} (\log n)^{\b} S_{\l,\a,\b} h_{\l,\a,\b}
\left(\frac{1}{n}\right)
\end{equation}
where
\begin{itemize}
\item
the summation is over a finite set, 
\item
the {\em growth rates} $\l$ are algebraic numbers of equal modulus,
\item
the {\em exponents} $\a$ are rational and the {\em nilpotency exponents} 
$\b$ are natural numbers,
\item
the {\em Stokes constants} $S_{\l,\a,\b}$ are complex numbers,
\item
the formal power series $h_{\l,\a,\b}(x) \in K[[x]]$ 
are Gevrey-1 (i.e., the coefficient of $x^n$ is bounded 
by $C^n n!$ for some $C>0$),
\item
$K$ is a  {\em number field} generated by the coefficients of 
$h_{\l,\a,\b}(x)$ for all $\l,\a,\b$. 
\end{itemize}
\end{definition}
For a detailed discussion of the uniqueness, existence and computation 
of the asymptotic expansion of a sequence $(a_n)$ of Nilsson type, see
\cite{Ga2}. The results of \cite{Ga2} and the above discussion implies
the following.

\begin{proposition}
\lbl{prop.asan}
\rm{(a)}
If $y(x)$ is algebraic and regular at $x=0$, then the sequence $(a_n)$
defined by \eqref{eq.fan} is of Nilsson type, where $\b=0$ in 
\eqref{eq.asan}.
\newline
\rm{(b)} Moreover, the asymptotic
expansion \eqref{eq.asan} can be computed exactly and effectively.
\end{proposition}

\noindent
We will apply the above proposition to the planar limit.

\begin{proposition}
\lbl{prop.alg}
\rm{(a)}
If $\calR_{e}(t),\calS_{e}(t)$ (resp. $\calR_{f}(t)$) are algebraic functions, 
then $(t^{2}\calF'_{0,e})'(t)$ (resp. $(t^{2}\calF_{0,f})'''(t)$) 
is also an algebraic function.  
\newline
\rm{(b)}
If $\calV$ is a polynomial, then $\calR_{e}$, $\calS_{e}$ and $\calR_{f}$ are
algebraic functions.  
\newline
\rm{(c)} Let
$$
\calF_{0}(t)=\sum_{n\ge1}f_{n}t^{n}.
$$
Under the assumptions (a) it follows that the sequence $(f_n)$ is holonomic.
\newline
\rm{(d)} In addition, the sequence $(f_n)$ is of Nilsson type.
\end{proposition}
Several illustrations of the above proposition to extreme potentials
are given in Sections \ref{s:14} and \ref{s:15}.

\subsection{The plan of the paper}
\lbl{sub.plan}

In Section~\ref{sec.fmm} we introduce and discuss the formal matrix models 
with the two important gradings, the edge and the face gradings.   
Section~\ref{sec.amm} introduces the potential theoretic part of 
analytic matrix models and the preliminary results needed in 
Section~\ref{s:5} where the main analytic results are presented.   
We  use here real analysis tools combined with Chebyshev polynomials 
and elementary combinatorics to deal with the minimization problem 
\eqref{eq.IV}, which is an alternative to the classical complex analysis 
techniques.     

Section~\ref{s:5} is the bulk of the analysis, the central pieces being 
Theorems~\ref{t:2} and \ref{t:3}.   These are applied to some analytic 
examples in Section~\ref{s.appl}.  

Next, in Section~\ref{s:13} we prove the matching claimed in 
Theorem~\ref{thm.2} and in Section~\ref{sec.F0RS} we give the proofs of 
the main results for the formal matrix models, namely Theorems~\ref{thm.Fe} 
and \ref{thm.Ff}.   

The main calculations with the extreme potentials are in 
Sections~\ref{s:14} and \ref{s:15}, for the edge grading and respective 
the face grading.    These main calculations are complemented with a small 
discussion in Section~\ref{sec:6} about the calculations in the case of 
planar diagrams with vertices of valence 3 or 4.  
 
 In Section \ref{s:16} we give the formal proof of t'Hooft's conjecture, 
materialized first in the general form of  Theorem~\ref{thm.s16} and then 
in Corollary \ref{c:100}, from which Theorem~\ref{thm.1} follows. 

At last, Section~\ref{s:pert} gives a perturbation result which is used in 
the proof of Theorem~\ref{thm.2} in Section~\ref{s:13}, though the results 
in this section do not give sharp results about the radius of convergence 
for the planar limit as in Section \ref{s:16}.    However this is a very 
useful analytic tool and we decided to include here.  

Finally,  the appendix contains some Taylor series of $\calR$, $\calS$ 
and $\calF$.  Some of these terms are used in the proof of the main 
results, Theorems~\ref{thm.Fe} and \ref{thm.Ff}.  

\subsection{Acknowledgment}
The authors wish to thank the anonymous referee for valuable
comments, corrections and references.

%%%%%%%%%%%%%%%%%%%%%%%%%%%%%%%%%%%%%%%%%%%%%%%%%%%%%%%%%%%%%%%%%%%%%%%%%%%%%
%%%%%%%%%%%%%%%%%%%%%%%%%%%%%%%%%%%%%%%%%%%%%%%%%%%%%%%%%%%%%%%%%%%%%%%%%%%%%

\section{Formal matrix models}
\lbl{sec.fmm}

It follows from the definition of the formal matrix model that the
planar limit $\calF_0$ is the generating series of counting of planar graphs,
weighted by the inverse of the size of their automorphism groups.
This is discussed in detail in \cite{BIPZ,E2,M2,Po}. In particular,
$\calF_{0,e}$ counts planar graphs where every $n$-valent edge contributes
a term $t^{n/2}$. Likewise, $\calF_{0,e}$ counts planar graphs where 
every $n$-valent face contributes a term $t^{n/2-1}$.

%%%%%%%%%%%%%%%%%%%%%%%%%%%%%%%%%%%%%%%%%%%%%%%%%%%%%%%%%%%%%%%%%%%%%%%%%%%%%
%%%%%%%%%%%%%%%%%%%%%%%%%%%%%%%%%%%%%%%%%%%%%%%%%%%%%%%%%%%%%%%%%%%%%%%%%%%%%

\section{Analytic matrix models}
\lbl{sec.amm}

\subsection{A summary of analytic matrix models}
\lbl{sub.amm}

One of the main problems one faces with the minimization problem 
\eqref{eq.IV} is the support of the equilibrium measure.  
Without extra assumptions on the potential $V$, the support 
can be an arbitrary compact subset of the reals. However, most of the 
formal computations  on the planar limit as a  counting object are based on 
the formal manipulations as if the support was one interval.   
 
Naturally, what we want to do here in the first place, is a complete 
analytical characterization of  the one interval support for the equilibrium 
measure of \eqref{eq.IV}.  The way we do this here is based on an elementary 
approach to the logarithmic potential due to the following formula for 
$x,y\in[-2,2]$: 
\[
\log|x-y|=-\sum_{n=1}^{\infty}\frac{2}{n}T_{n}\left(\frac{x}{2}\right)T_{n}
\left(\frac{y}{2}\right)
\]
where $T_{n}$ are the the Chebyshev polynomials of first kind.  Based on 
this formula we give a quick incursion into various formulae in logarithmic 
potential theory on $[-2,2]$, especially the formula from Theorem~\ref{t:1} and show that
the general case of one-cut potentials can always be reduced by rescaling and 
translation in the $x$-variable to this case.   
    The reason of doing this is to highlight a way of using 
manipulations of the Chebyshev polynomials in this framework.  This seems 
to be an alternative (in the case of measures with support $[-2,2]$) to the 
powerful complex analysis methods discussed for example in \cite{ST}.  

However, the more interesting fact is that we obtain the following explicit
formula for $I_{V}$.  If $V$ is a $C^{3}$ potential whose equilibrium 
measure has support $[-2c+b,2c+b]$, then 
\begin{equation}
\lbl{eq.Fformula}
I_{V} = -\log c +\int_{-2}^{2}\frac{V(cx+b)dx}{\pi\sqrt{4-x^{2}}} 
- \int_{0}^{c}s\left[\left(\int_{-2}^{2}
\frac{xV'(sx+b)dx}{2\pi \sqrt{4-x^{2}}} \right)^{2}+\left(  
\int_{-2}^{2}\frac{V'(sx+b)dx}{\pi \sqrt{4-x^{2}}} \right)^{2} \right]ds. 
\end{equation}
This is attained via concrete exploitations of the Chebyshev polynomials, 
first for the case of the interval $[-2,2]$ and then simple rescaling.  It 
is worth pointing out that ultimately, this identity reduces to checking a 
combinatorial identity for binomial coefficients.  This we carry out using 
the implementation of the Zeilberger method.

The previous formula, makes the dependence on the potential very transparent. 
Any questions on the analyticity of  $I_{V}$ (or $F_{0,V}$) under 
perturbation follows from 
the analyticity of the endpoints of the support of the equilibrium measure.

Finally, we show that under certain non-degeneracy conditions made explicit 
in Section~\ref{s:pert}, if $V_{\mathbf{t}}$ is an analytic perturbation of $V$ 
depending on the parameter $\mathbf{t}$, then the planar limit 
$I_{V_{\mathbf{t}}}$ depends 
analytically on $\mathbf{t}$ on a domain of the parameter space.  

Here is an outline of what follows.  In Section~\ref{s:2} we introduce the 
main objects, in Section~\ref{s:3} we discuss the formula that connects the 
logarithmic potentials and the Chebyshev polynomials.  Next, in 
Section~\ref{s:4}, we describe the connection with Fourier analysis.  
Section~\ref{s:5} contains the main analytical results.

\subsection{Logarithmic Potentials with External Fields}
\lbl{s:2}

As it was pointed out in the Introduction, we are going to look at the problem 
of minimizing the logarithmic energy with external fields and then 
investigate the planar limit in this framework.  

Assume $V:\R\to\R$ is an admissible potential.
For a closed set $S\subset\R$, according to \cite{ST} for the general case 
or \cite{Deift1} for the case $S=\R$,  the following minimization problem 
has a unique solution (which turns out to be compactly supported)
\begin{equation}\lbl{e:1}
I_{V}(S)=\inf \left\{ I_{V}(\mu): \mu\in\mathcal{P}(S) \right\}
\end{equation}
where $\mathcal{P}(S)$ stands for the set of probability measure on $S$ and 
\begin{equation}\lbl{e:0}
I_{V}(\mu)=\int V(x)\mu(dx)-\iint \log|x-y|\mu(dx)\mu(dy).
\end{equation}
The term $-\iint \log|x-y|\mu(dx)\mu(dy)$ is called the \emph{logarithmic energy} of the measure $\mu$.  
For simplicity, we will denote $I_{V}=I_{V}(\R)$.  Also for a given measure 
$\mu$, we will denote $\supp{\mu}$, the support of the measure $\mu$. 
The equilibrium measure of \eqref{e:1} on the set $S$ (cf. 
\cite[Thm.I.1.3]{ST}) is characterized by 
\begin{equation}\lbl{e:var}
\begin{split}
V(x)&\ge 2\int \log|x-y|\mu(dy)+C \quad\text{quasi-everywhere on }S \\ 
V(x)&= 2\int \log|x-y|\mu(dy)+C 
\quad\text{quasi-everywhere on}\:\:\supp{\mu}. \\ 
\end{split}
\end{equation} 
Here, a property $P$ holds ``quasi everywhere'' on the set $\Omega$ if we can find a set $A$ such that $\mu(A)=0$ for any measure $\mu$ of finite logarithmic energy and the property $P$ holds on  $\Omega\backslash A$.   
This means,  that the equality on $\supp \mu$ is   
almost surely realized with respect to any measure of finite logarithmic energy.  

Notice here that if we change the variable of integration to $x\to c x+b$ 
and $y\to c y+b$, where $c\ne0$, then, with  
\[
\mu_{c,b}=((\cdot-b)/c)_{\#}\mu,
\] 
($\cdot/c$ standing for the multiplication by $1/c$),  and for a given 
function $\phi$, the push forward $\phi_{\#}\mu$ is the measure defined by 
$\phi_{\#}\mu(A)=\mu(\{x:\phi(x)\in A \})$ for any Borel measurable $A$.   
Therefore we have 
\begin{eqnarray}
\lbl{e:5}
I_{V}(\mu)&=&\int V(cx+b)\mu_{b,c}(dx)-\iint\log|cx-cy|\mu_{b,c}(dx)
\mu_{b,c}(dy) \\ \notag &=& I_{V(\cdot c+b)-\log(c)}(\mu_{b,c})
=I_{V(\cdot c+b)}(\mu_{b,c})-\log c
\end{eqnarray}
which in turn results with 
\begin{equation}\lbl{e:2}
\notag I_{V}=I_{V(\cdot c+b)-\log(c)}= I_{V(\cdot c+b)}-\log(c).
\end{equation}

\subsection{Logarithmic Potentials  and Chebyshev Polynomials}
\lbl{s:3}

Recall that the {\em Chebyshev polynomials of the first kind} 
$T_n(x)$ are defined by
\begin{equation}
\lbl{eq.Tn}
T_n(\cos\th)=\cos(n\th)
\end{equation}
see for example, \cite{Ol}. Alternatively, they are given by the recursion
relation
$$
T_{n+1}(x)=2xT_n(x)-T_{n-1}(x), \qquad T_0(x)=1, \,\, T_1(x)=x.
$$
$T_{n}$ are the orthogonal polynomials for the {\em arcsine law} 
$\mathbbm{1}_{[-1,1]}(x)\frac{1}{\pi\sqrt{1-x^{2}}}$.
The following lemma is due to Haagerup \cite{H} and we reproduce the proof 
here for completeness.  

\begin{lemma}[Haagerup]
\lbl{l:1} 
\rm{(a)} For any real $x,y\in [-2,2]$, $x\ne y$, we have
\[
\log|x-y|=-\sum_{n=1}^{\infty}\frac{2}{n}T_{n}\left(\frac{x}{2}\right)
T_{n}\left(\frac{y}{2}\right)
\]
where the series here is convergent on $x\ne y$.  
\newline
\rm{(b)} If $x>2$ and  $y\in[-2,2]$, we have 
\[
\log|x-y|=\log\left|\frac{x+\sqrt{x^{2}-4}}{2} \right|-\sum_{n=1}^{\infty}
\frac{2}{n}\left( \frac{x-\sqrt{x^{2}-4}}{2} \right)^{n}T_{n}
\left(\frac{y}{2}\right)
\]
where the series is absolutely convergent.  
\newline
\rm{(c)} 
The logarithmic potential of a measure on $[-2,2]$ is given by 
\begin{equation}\lbl{e:34}
\int \log|x-y|\mu(dx)=-\sum\frac{2}{n}T_{n}\left(\frac{x}{2}\right)
\int T_{n}\left(\frac{y}{2}\right)\mu(dy)
\end{equation}
where this series makes sense pointwise.
\newline 
\rm{(d)} The logaritmic energy of the 
measure $\mu$ is given by
\begin{equation}\lbl{e:35}
\iint \log|x-y|\mu(dx)\mu(dy)=-\sum_{n=1}^{\infty}\frac{2}{n}\left(
\int T_{n}\left(\frac{x}{2}\right)\mu(dx)\right)^{2}.
\end{equation}
In particular $\iint \log|x-y|\mu(dx)\mu(dy)$ is finite if and only if 
$\sum_{n=1}^{\infty}\frac{2}{n}\left(\int T_{n}\left(\frac{x}{2}\right)
\mu(dx)\right)^{2}$ is finite.

\end{lemma}

\begin{proof} We first point out that for any complex number $z\ne1$, with 
$|z|=1$, one has that 
\begin{equation}\lbl{e:33}
\log(1-z)=-\sum_{n=1}^{\infty}\frac{z^{n}}{n},
\end{equation}
where we take the branch of $\log$ on $\C\backslash(-\infty,0]$.  Now, write 
$x=2\cos u$ and $y=2\cos v$, and observe 
\[
x-y=2(\cos u-\cos v)=4\sin\left(\frac{u+v}{2}\right)\sin\left(\frac{v-u}{2}
\right),
\]
and hence,
\begin{align*}
\log|x-y|& =\log \left|2\sin\left(\frac{u+v}{2} \right) \right| 
+\log \left|2\sin\left(\frac{v-u}{2} \right) \right| \\ 
&=\log|1-e^{i(u+v)}|+\log|1-e^{i(v-u)}|\\ 
& = Re\left(\log(1-e^{i(u+v)})+\log(1-e^{i(v-u)}) \right) \\ 
& = -\sum_{n=1}^{\infty}\frac{1}{n}Re\left( e^{in(u+v)}+e^{in(v-u)} \right) \\ 
& = -\sum_{n=1}^{\infty}\frac{1}{n} \left(\cos(n(u+v))+\cos(n(v-u)) \right) \\ 
& = -\sum_{n=1}^{\infty}\frac{2}{n} \cos (nu)\cos(nv) \\
& = -\sum_{n=1}^{\infty}\frac{2}{n}T_{n}\left(\frac{x}{2}\right)T_{n}
\left(\frac{y}{2}\right).
\end{align*}

For the case $x>2$ and $|y|\le2$, then write $x=2\cosh u=e^{u}+e^{-u}$, 
where $u=\log\frac{x+\sqrt{x^{2}-4}}{2}$ and $y=2\cos v$, thus 
\begin{align*}
\log|x-y|& =\log \left( e^{u}(1-e^{-u+iv})(1-e^{-u-iv}) \right) \\
& = u+ \log(1-e^{-u+iv})+\log(1-e^{-u-iv}) \\ 
&= u-\sum_{n=1}^{\infty}\frac{2}{n}e^{-nu}\cos(nv).
\end{align*}

For the second part, for given $-1<r<1$ we introduce the kernel
\[
L_{r}(x,y):=-\sum_{n\ge1}\frac{2r^{n}}{n}T_{n}\left( \frac{x}{2}\right)
T_{n}\left( \frac{y}{2}\right).
\]
This can be computed for $x=2\cos u$ and $y=2\cos v$ for $u,v\in[0,\pi)$ 
with $u\ne v$ as 
\[
\begin{split}
L_{r}(2\cos u,2\cos v)&=-\sum_{n\ge 1}\frac{2r^{n}}{n}\cos(nu)\cos(nv) \\ 
 & =-\sum_{n=1}^{\infty}\frac{r^{n}}{n} \left(\cos(n(u+v))+\cos(n(u-v)) 
\right) \\
 & = -\sum_{n=1}^{\infty}\frac{r^{n}}{n}Re\left( e^{in(u+v)}+e^{in(u-v)} 
\right) \\ 
& =_{\eqref{e:33}} \log|1-e^{i(u+v)}|+\log|1-e^{i(u-v)}|\\ 
& = \frac{1}{2}\left(\log (1+r^{2}-2r\cos(u+v))+\log (1+r^{2}-2r\cos(u-v)) 
\right)
\end{split}
\]
Next, for any $\theta$, 
\[
4\ge 1+r^{2}-2r\cos \theta\ge \left(\frac{1+r}{2}\right)^{2}(2-2\cos\theta)
\] 
which results with 
\[
\log 4\ge L_{r}(2\cos u,2\cos v)\ge 2\log\frac{1+r}{2}+\log|2\cos u-2\cos v|,
\]
or for $x,y\in[-2,2]$, $x\ne y$, 
\[
\log 4\ge L_{r}(x,y)\ge 2\log\frac{1+r}{2}+\log|x-y|.  
\]
This combined with Fatou's lemma yields that 
\[
\lim_{r\to 1^{-}}\int L_{r}(x,y)\mu(dy)=\int \log|x-y|\mu(dy).
\]
The rest follows.\qedhere
\end{proof}
The first consequence of the above proposition is the computation of the 
well-known arcsine law of an interval; \cite{ST}.

\begin{corollary}\lbl{c:5}
 If $\omega(dx)=\mathbbm{1}_{[-2,2]}(x)\frac{dx}{\pi\sqrt{4-x^{2}}}$ is the 
arcsine law of the interval $[-2,2]$, then 
\begin{equation}\lbl{e:43}
\int \log|x-y|\omega(dy)=\begin{cases} 
0,& |x|\le2\\
\log\frac{|x|+\sqrt{x^{2}-4}}{2} ,& |x|>2.
\end{cases}
\end{equation}
If $\mu$ is a signed measure on $[-2,2]$ with finite total variation and 
finite logarithmic energy, then 
\begin{equation}\lbl{e:32}
\int \log|x-y|\mu(dy)=c
\text{ almost everywhere for all } x\in[-2,2] 
\end{equation}
 if and only if $\mu(dx)=\mathbbm{1}_{[-2,2]}(x)\frac{\mu([-2,2])dx}{
\pi\sqrt{4-x^{2}}}$. Here,  almost everywhere is understood with respect 
to the Lebesgue measure.  Additionally,  the constant $c$ must be $0$.  
\end{corollary}

\begin{proof}
Because the density of $\omega$ is even, it suffices to prove \eqref{e:43} 
for $x>2$ or $x\in[-2,2]$.   Equation \eqref{e:43} follows from the lemma 
and the fact that the series in \eqref{e:34} is convergent and is convergent 
also in $L^{2}(\mathbbm{1}_{[-2,2]}(x)\frac{1}{\pi\sqrt{4-x^{2}}})$.   

For the second part, integrating \eqref{e:32} with respect to the arcsine 
law and exchanging the integration one obtains that 
$c=0$.  Now, using equality \eqref{e:34} we obtain that $\int T_{n}
\left(\frac{x}{2} \right)\mu(dy)=0$ for all $n\ge1$ and thus $\mu$ and 
$\mathbbm{1}_{[-2,2]}(x)\frac{\mu([-2,2])dx}{\pi\sqrt{4-x^{2}}}$ have the 
same moments and consequently must be equal.    
\end{proof}

\subsection{A Connection with Fourier Analysis}
\lbl{s:4}

Take the map 
\[
\Theta :[0,\pi]\ni\theta \to  2\cos\theta\in[-2,2].
\]
For a given measure $\mu$ on $[-2,2]$ we define 
$\tilde{\mu}=\mu\circ \Theta$, the measure on $[0,\pi]$ such that 
\[
\tilde{\mu}(A)=\mu(\Theta(A))
\]
for every measurable set $A$ in $[0,\pi]$.  The map $\mu\to\tilde{\mu}$ from 
measures on $[-2,2]$ into the set of measures on $[0,\pi]$ is a 
one to one and onto.    The advantage of using this comes from 
\begin{equation}\lbl{e:3}
\alpha_{n}:=\int T_{n}\left(\frac{x}{2}\right)\mu(dx)=\int\cos(n\theta)
\tilde{\mu}(d\theta),\quad n\ge0,
\end{equation}
which tells us that the ``moments'' of $\mu$ with respect to Chebyshev's 
polynomials are seen as the Fourier coefficients of a measure on 
$[0,\pi]$. 

The following result is standard and we state it without proof.  

\begin{proposition}\lbl{p:2}
Given a sequence $\{\alpha_{n}\}_{n\ge0}$, with $\alpha_{0}=1$,  the 
following are equivalent:
\begin{enumerate}
\item There exists a measure on $[-2,2]$ such that $\alpha_{n}=
\int T\left(\frac{x}{2}\right)\mu(dx)$; 

\item $\{\alpha_{n}\}_{n\ge0}$ is a bounded sequence and 
\[
\langle \mu,\phi\rangle:=\int_{-2}^{2}\frac{\phi(x)}{\pi \sqrt{4-x^{2}}}dx
+2\sum_{n=1}^{\infty}\alpha_{n}\int_{-2}^{2}T_{n}\left(\frac{x}{2}\right)
\frac{\phi(x)}{\pi \sqrt{4-x^{2}}}dx
\]
defines a nonnegative distribution, i.e. for any smooth nonnegative function 
$\phi:[-2,2]\to\R_{+}$,
\[
\langle \mu,\phi \rangle\ge0.
\]
\end{enumerate}
In particular,  if $\sum_{n=1}^{\infty}|\alpha_{n}|$ is convergent, then 
there is a measure $\mu$ on $[-2,2]$ such that $\alpha_{n}=\int T
\left(\frac{x}{2}\right)\mu(dx)$, if and only if 
\begin{equation}\lbl{e:7}
u(x):=1+2\sum_{n=1}^{\infty}\alpha_{n}T_{n}\left(\frac{x}{2}\right)\ge 0
\quad\text{for all} \quad x\in[-2,2].
\end{equation}
In this case the measure $\mu$ is given by $\mu(dx)=\frac{u(x)dx}{\pi 
\sqrt{4-x^{2}}}$.

\end{proposition}

%%%%%%%%%%%%%%%%%%%%%%%%%%%%%%%%%%%%%%%%%%%%%%%%%%%%%%%%%%%%%%%%%%%%%%%%%%%%%
%%%%%%%%%%%%%%%%%%%%%%%%%%%%%%%%%%%%%%%%%%%%%%%%%%%%%%%%%%%%%%%%%%%%%%%%%%%%%

\section{The planar limit of analytic matrix models, The Main Results}
\lbl{s:5}

Given a continuous function $f$ on $[-2,2]$, we define 
\begin{equation}\lbl{e:36}
\beta_{n}(f)=\int_{-2}^{2}f(x)T_{n}\left(\frac{x}{2}\right)\frac{dx}{
\pi\sqrt{4-x^{2}}},\quad n\ge0.
\end{equation}
Notice that if $f$ is bounded by a $C^{k-1}$ function and piecewise $C^{k}$ for some $k\ge0$, using the Fourier 
interpretation and repeated integrations by parts we learn that 
$\beta_{n}(f)=o(n^{-k})$.  
Next define the orthogonal polynomials 

\begin{equation}
\lbl{e:37}
\tilde{T}_{n}(x)=\sqrt{2}T_{n}(x/2)
\end{equation} 
for $n\ge1$ and $\tilde{T}_{0}=T_{0}=1$.   These  provide a Hilbert basis of 
$L^{2}\left(\mathbbm{1}_{[-2,2]}(x)\frac{1}{\pi\sqrt{4-x^{2}}}dx\right)$.

\begin{theorem}\lbl{t:1}
Assume that $V$  is a $C^{2}$ and piecewise $C^{3}$ function on $[-2,2]$ and $A\in\R$ a constant.  
Then,  there is a unique signed measure $\mu$ on $[-2,2]$ of finite total 
variation which solves 
\[
\begin{cases}
2\int \log|x-y|\mu(dx)=V(x)+C \text{ almost everywhere for } x\in[-2,2],\\
\mu([-2,2])=A.
\end{cases}
\]
where almost everywhere is with respect to the Lebesgue measure on $[-2,2]$. 
The solution $\mu$ is given by $\mu(dx)=\frac{u(x)dx}{\pi \sqrt{4-x^{2}}}$ 
where
\begin{equation}\lbl{e:sol}
u(x)=A-\frac{1}{2}\int_{-2}^{2}\frac{yV'(y)dy}{\pi\sqrt{4-y^{2}}} 
-\frac{x}{2}\int_{-2}^{2}\frac{V'(y)dy}{\pi\sqrt{4-y^{2}}}
+ \frac{4-x^{2}}{2}\int_{-2}^{2} \frac{V'(x)-V'(y)}{x-y}
\frac{dy}{\pi\sqrt{4-y^{2}}}.
\end{equation}
In addition, the constant $C$ must be given by  $C=-\int_{-2}^{2}
\frac{V(x)dx}{\pi\sqrt{4-x^{2}}}$.
\end{theorem}

\begin{proof}
In the first place, the uniqueness is clear because of Corollary~\ref{c:5}.  

To prove the rest we first write
the function $V$ 
\[
V(x)=\sum_{n=0}^{\infty}\langle \tilde{T}_{n},V \rangle \tilde{T}_{n}(x)
=\beta_{0}(V)+2\sum_{n=1}^{\infty}\beta_{n}(V)T_{n}\left(\frac{x}{2}\right)
\]
where the inner product is taken in 
$L^{2}\left(\mathbbm{1}_{[-2,2]}(x)\frac{dx}{\pi\sqrt{4-x^{2}}}\right)$ and 
point out that the regularity of $V$ implies that $\beta_{n}(V)=O(1/n^{2})$.  
Invoking representation \eqref{e:34}, results with
\[
-2\sum_{n\ge1}\frac{2}{n}\left(\int T_{n}
\left(\frac{y}{2}\right)\mu(dy)\right)T_{n}\left(\frac{x}{2}\right)
=C+\beta_{0}(V)+2\sum_{n=1}^{\infty}
\beta_{n}(V)T_{n}\left(\frac{x}{2}\right).
\]
Thus, equating the coefficients, we must have now $C=-\beta_{0}(V)$ and 
\[
\int T_{n}\left(\frac{x}{2}\right)\mu(dx) = -\frac{n}{2}\beta_{n}(V)
\]
which,  means that $\mu(dx)=\frac{u(x)dx}{\pi\sqrt{4-x^{2}}}$, for
\[
u(x)=A-\sum_{n=1}^{\infty}n\beta_{n}(V)T_{n}\left(\frac{x}{2}\right).
\]
Here is the point where we need the $C^{3}$ assumption to make sure this series converges absolutely since in this case $n\beta_{n}(V)=o(1/n^{2})$.

To prove equality \eqref{e:sol}, our task now is to show that 
{\small
\[
-\sum_{n=1}^{\infty}n\beta_{n}(V)T_{n}\left(\frac{x}{2}\right)
=-\frac{1}{2}\int_{-2}^{2}\frac{yV'(y)dy}{\pi\sqrt{4-y^{2}}} 
-\frac{x}{2}\int_{-2}^{2}\frac{V'(y)dy}{\pi\sqrt{4-y^{2}}}
+ \frac{4-x^{2}}{2\pi^{2}}\int_{-2}^{2} 
\frac{V'(y)-V'(x)}{x-y}\frac{dy}{\sqrt{4-y^{2}}}. 
\]
}
Notice that both sides of this equation are linear functions of $V$ and 
thus by a simple approximation argument it suffices to check it for the 
case of $V(x)=T_{m}\left( \frac{x}{2}\right)$ for some $m\ge1$.  After 
making the change of variables $x=2\cos u$, $y=2\cos v$, this identity 
reduces to checking that 
{\small
\[
-\cos(n u)=-\frac{1}{\pi}\int_{0}^{\pi}\frac{\cos v\sin n v}{\sin v}dv
-\frac{\cos u}{\pi}\int_{0}^{\pi}\frac{\sin n v}{\sin v}dv
+\frac{\sin u}{\pi}\int_{0}^{\pi} 
\frac{\sin(nu)\sin v-\sin(nv)\sin u}{(\cos u-\cos v)\sin v}dv.
\]
}
Now, instead of checking this we look at the generating functions of the 
right and left hand sides.  Specifically, using the fact that 
for $-1<r<1$ and $t\in[0,\pi]$, 
\begin{equation}\lbl{e:40}
\sum_{n=1}^{\infty}r^{n-1}\cos(nt)=\frac{\cos t-r}{1-2r\cos t+r^{2}}
\quad \text{and} \quad\sum_{n=1}^{\infty}r^{n-1}\sin(nt)
=\frac{\sin t}{1-2r\cos t+r^{2}},
\end{equation}
the rest follows from straightforward computations.\qedhere
\end{proof}

\begin{proposition}  
\lbl{prop.pa}
If $\mu\in\mathcal{P}([-2,2])$ and $V$  is a $C^{2}$ and piecewise $C^{3}$ function on $[-2,2]$, then 
\begin{equation}\lbl{e:4}
I_{V}(\mu)=\beta_{0}(V)+2\sum_{n=1}^{\infty}\left(\beta_{n}(V)\alpha_{n}
+\frac{\alpha_{n}^{2}}{n}\right)
\end{equation}
where 
\begin{equation}\lbl{e:6}
\alpha_{n} =\int T_{n}\left(\frac{x}{2}\right)\mu(dx).
\end{equation}
Furthermore, we also have that
\begin{equation}\lbl{e:9}
I_{V}(\mu)\ge\beta_{0}(V)-\frac{1}{2}\sum_{n=1}^{\infty}n\beta_{n}(V)^{2}
\end{equation}
with equality if and only if $\alpha_{n}=-\beta_{n}(V)/2$ and
\begin{equation}\lbl{e:10}
1-\sum_{n=1}^{\infty}n\beta_{n}(V)T_{n}\left(\frac{x}{2}\right)
\ge0\quad\text{for any}\quad x\in [-2,2].
\end{equation}
In this case, 
\begin{equation}\lbl{e:11}
\mu(dx)=\left(1-\sum_{n=1}^{\infty}n\beta_{n}(V)T_{n}\left(\frac{x}{2}
\right)\right)\frac{dx}{\pi\sqrt{4-x^{2}}}.
\end{equation}
\end{proposition}

\begin{proof}
One can write
\[
\int Vd\mu=\beta_{0}(V)+2\sum_{n=1}^{\infty}\beta_{n}(V) \int T_{n}
\left(\frac{x}{2}\right)\mu(dx)=\beta_{0}+2\sum_{n=1}^{\infty}\beta_{n}(V)
\alpha_{n}. 
\]
To prove \eqref{e:9}, one needs to complete the square in \eqref{e:4} to 
get that 
\[
I_{V}(\mu)=\beta_{0}(V)-\frac{1}{2}\sum_{n=1}^{\infty}n\beta_{n}(V)^{2}
+\sum_{n=1}^{\infty}\frac{2}{n}\left(\alpha_{n}+\frac{n\beta_{n}(V)}{2}
\right)^{2}.
\]
This implies inequality \eqref{e:9}.  The equality is attained only for 
the case $\alpha_{n}=-\frac{n\beta_{n}(V)}{2}$ which, cf. \eqref{e:7} 
determines a measure on $[-2,2]$ if and only if \eqref{e:10} is satisfied. 
 The rest follows easily.  \qedhere
\end{proof}

We arrive at a necessary and sufficient condition for
deciding that an equilibrium measure on $[-2,2]$ has full support.  

\begin{corollary}
\lbl{c:1}
Assume that $V$  is a $C^{2}$ and piecewise $C^{3}$ function on $[-2,2]$.  Then, the equilibrium 
measure on $[-2,2]$ has full support if and only if 
\begin{equation}
\lbl{e:27}
1-\sum_{n=1}^{\infty}n\beta_{n}(V)T_{n}\left(\frac{x}{2}\right)>0
\quad\text{for $x$ on a dense subset of}\:\: [-2,2].
\end{equation}
In addition, in this case, we also have that
\[
\inf_{\mu\in\mathcal{P}([-2,2])}I_{V}(\mu)=\beta_{0}(V)-\frac{1}{2}
\sum_{n=1}^{\infty}n\beta_{n}(V)^{2}.
\]
\end{corollary}

\begin{proof}  
Condition \eqref{e:27} and the previous Proposition 
guarantee that there is a measure $\mu$ with full support such that 
$\int_{-2}^{2} T_{n}\left(\frac{x}{2} \right)\mu(dx)=-\frac{n\beta_{n}}{2}$.  

The other way around works as follows.  Assume that $\mu_{V}$ is the 
equilibrium measure on $[-2,2]$ and has full support.  What we need to 
show is that \eqref{e:27} is satisfied. 

Let $\alpha_{n}=\int_{-2}^{2} T_{n}\left(\frac{x}{2} \right)\mu_{V}(dx)$.  
Then, for any other measure $\nu$ with $I_{V}(\nu)<\infty$ on $[-2,2]$,  
from \eqref{e:4} we obtain that
\[
I_{V}\left((1-\epsilon)\mu_{V}+\epsilon \nu\right)=\beta_{0}(V)
-\frac{1}{2}\sum_{n=1}^{\infty}n\beta_{n}(V)^{2}
+\sum_{n=1}^{\infty}\frac{2}{n}\left((1-\epsilon)\alpha_{n}
+\epsilon\alpha'_{n}+\frac{n\beta_{n}(V)}{2}\right)^{2}
\] 
where $\alpha'_{n}=\int T_{n}\left(\frac{x}{2} \right)\nu(dx)$.  
Since $I_{V}(\mu_{V})$ is the minimum of $I_{V}(\nu)$ over all probability 
measures on $[-2,2]$, differentiation with respect to 
$\epsilon>0$ at $0$ yields
\begin{equation}\lbl{e:28}
\sum_{n\ge1}\frac{1}{n}\left(\alpha_{n}
+\frac{n\beta_{n}(V)}{2}\right)(\alpha'_{n}-\alpha_{n})\ge0. 
\end{equation}
Now we consider measures of the form 
\[
\nu(dx)=(1+\delta \phi(x))\mu_{V}(dx)
\]
where $\phi$ is a polynomial such that $\int \phi d\mu=0$ and $\delta$ is 
small in absolute value.   Applying this for $\pm\delta$, in \eqref{e:28}, 
where $\delta$ is small enough, we obtain that 
\begin{equation}\lbl{e:29}
\sum_{n\ge1}\frac{1}{n}\left(\alpha_{n}+\frac{n\beta_{n}(V)}{2}\right)
\int T_{n}\left(\frac{x}{2}  \right)\phi(x)\mu_{V}(dx)=0. 
\end{equation}
A word of caution is in order here.  We need to justify that the measure 
$\nu$ has finite logarithmic energy, namely that 
\[
\sum_{n\ge1}\frac{1}{n}\left(\int T_{n}\left(\frac{x}{2} \right)(1
+\delta\phi(x))\mu_{V}(dx) \right)^{2}<\infty.
\]  
This actually follows easily for each polynomial $\phi=T_{k}$ for some 
$k\ge0$ from the fact that $2T_{k}T_{l}=T_{|k-l|}+T_{k+l}$ for any $k,l\ge0$.  

Because of \eqref{e:var} we have that $2\int \log|x-y|\mu_{V}(dy)=V(x)+C$, almost surely (with respect to the Lebesgue measure) 
on $[-2,2]$, and from Theorem~\ref{t:1}, the density $g(x)$ of $\mu_{V}$ 
is given by 
\[
g(x)=A_{1}\sqrt{4-x^{2}}\int_{-2}^{2} 
\frac{V'(y)-V'(x)}{\sqrt{4-y^{2}}(y-x)}dy+\frac{A_{2}+A_{3}x}{\sqrt{4-x^{2}}}
\]
for some constants $A_{1}$ and $A_{2}$.   In particular, since $V$ is 
$C^{3}$ it implies that $g(x)=\frac{h(x)}{\pi\sqrt{4-x^{2}}}$  for some 
continuous function $h$.  Since the measure $\mu_{V}$ has full support, 
$h(x)>0$ on a dense set.   

Next, we observe that because $\beta_{n}(V)$ is square summable in 
$L^{2}\left(\mathbbm{1}_{[-2,2]}(x)\frac{dx}{\pi \sqrt{4-x^{2}}}\right)$, 
\[
R(x):=\sum_{n\ge1}\frac{1}{n}\left(\alpha_{n}+\frac{n\beta_{n}(V)}{2}\right) 
T_{n}\left(\frac{x}{2}  \right)
\]
is convergent in  $L^{2}(\mathbbm{1}_{[-2,2]}(x)\frac{dx}{\pi 
\sqrt{4-x^{2}}})$, therefore we deduce from \eqref{e:29} that 
\[
\int_{-2}^{2} R(x)\phi(x)h(x)\frac{dx}{\pi \sqrt{4-x^{2}}}=0,
\]
for any polynomial $\phi$.  This easily implies that $R(x)=0$ almost 
everywhere and this in turn results with $\alpha_{n}=-n\beta_{n}(V)/2$.  
The rest follows.  \qedhere
\end{proof}

\begin{theorem}
\lbl{t:2}  
If the equilibrium measure of a $V$ which is  $C^{2}$ and piecewise $C^{3}$ on $[-2,2]$ has full support then, 
\begin{align*}
I_{V}&=\inf_{\mu\in\mathcal{P}([-2,2])} I_{V}(\mu) 
= \int_{-2}^{2}\frac{V(x)dx}{\pi\sqrt{4-x^{2}}}
+\int_{0}^{1}t\left[\left(\int_{-2}^{2}\frac{xV'(tx)dx}{2\pi 
\sqrt{4-x^{2}}} \right)^{2}+\left(  \int_{-2}^{2}
\frac{V'(tx)dx}{\pi \sqrt{4-x^{2}}} \right)^{2} \right]dt  \\
&=V(0)-\int_{0}^{1}
\frac{1}{t}\left[-1+\left(1-\int_{-2}^{2}\frac{txV'(tx)dx}{2\pi 
\sqrt{4-x^{2}}} \right)^{2}+\left(  \int_{-2}^{2}
\frac{tV'(tx)dx}{\pi \sqrt{4-x^{2}}} \right)^{2} \right]dt.
\end{align*}
\end{theorem}

\begin{proof}   
According to Corollary~\ref{c:1}, we have 
\[
I_{V}=\beta_{0}(V)-\frac{1}{2}\sum_{n\ge1}n\beta_{n}(V)^{2}.
\]
Thus our task is to prove that
\[
\frac{1}{2}\sum_{n\ge1}n\beta_{n}(V)^{2}=\int_{0}^{1}t
\left[\left(\int_{-2}^{2}\frac{xV'(tx)dx}{2\pi \sqrt{4-x^{2}}} 
\right)^{2}+\left(  \int_{-2}^{2}\frac{V'(tx)dx}{\pi \sqrt{4-x^{2}}} 
\right)^{2} \right]dt.
\]
A polarization argument shows that this is equivalent to proving that 
for any $C^{3}$ potentials $V_{1}$ and $V_{2}$
\begin{equation}\lbl{e:30}
\begin{split}
\frac{1}{2}\sum_{n\ge1}n\beta_{n}(V_{1})\beta_{n}(V_{2})=&\int_{0}^{1}t
\left(\int_{-2}^{2}\frac{xV_{1}'(tx)dx}{2\pi \sqrt{4-x^{2}}} \right)
\left(\int_{-2}^{2}\frac{xV_{2}'(tx)dx}{2\pi \sqrt{4-x^{2}}} \right)dt \\  &
+\int_{0}^{1}t\left(  \int_{-2}^{2}\frac{V_{1}'(tx)dx}{\pi \sqrt{4-x^{2}}} 
\right)\left(  \int_{-2}^{2}\frac{V_{2}'(tx)dx}{\pi \sqrt{4-x^{2}}} \right)dt.
\end{split}
\end{equation}
To do this, because of the linearity in $V_{1}$ and $V_{2}$ and the fact 
that polynomials are dense (with respect to $C^{3}$ topology) in the set 
of smooth functions on $[-2,2]$, it suffices to check this for 
$V_{1}(x)=x^{k}$ and $V_{2}(x)=x^{m}$.   If $k$ or $m$ is zero, both sides 
of \eqref{e:30} are zero, therefore we need to check this for $k,m\ge1$.  

Now, we use
\[
x^{2n}=\frac{1}{2^{2n-1}}\sum_{k=0}^{n-1}{2n \choose k }T_{2n-2k}(x)
+\frac{1}{2^{2n}}{2n \choose n}, \quad\text{and}\quad x^{2n+1}
=\frac{1}{2^{2n}}\sum_{k=0}^{n}{2n+1\choose k}T_{2n+1-2k}(x)
\]
from which a direct calculation yields
\begin{equation}\lbl{e:14-1}
\int_{-2}^{2} x^{i}T_{n}\left( \frac{x}{2}\right)\frac{dx}{\pi 
\sqrt{4-x^{2}}}={i\choose \frac{i-n}{2}},
\end{equation}
with the convention that ${i\choose p+1/2}=0$ for $p\in\Z$ and 
${i\choose p}=0$ for $p<0$.   Therefore, \eqref{e:30} becomes in this case 
\[
\frac{1}{2}\sum_{n\ge1} n {m \choose \frac{m-n}{2}}{k \choose 
\frac{k-n}{2}}=\frac{mk}{4(m+k)}{m \choose \frac{m}{2}}{k \choose 
\frac{k}{2}}+\frac{mk}{(m+k)}{m-1 \choose \frac{m-1}{2}}{k-1 \choose 
\frac{k-1}{2}}.
\]
In the case $m,k$ have different parities, then the above expression is 
$0$.  If they have the same parities, the equality follows from the next 
Lemma.\qedhere
\end{proof}

\begin{lemma} 
\lbl{lem.la}
The following identities hold
\begin{align*}
\sum_{p}p{2l_{1}\choose l_{1}-p}{2l_{2}\choose l_{2}-p}&
=\frac{l_{1}l_{2}}{2(l_{1}+l_{2})}{2l_{1}\choose l_{1}}{2l_{2}\choose l_{2}}\\ 
\sum_{p}(2p+1){2l_{1}+1\choose l_{1}-p}{2l_{2}+1\choose l_{2}-p}&
=\frac{(2l_{1}+1)(2l_{2}+1)}{l_{1}+l_{2}+1}{2l_{1}\choose l_{1}}{2l_{2}
\choose l_{2}} 
\end{align*}
with the convention that  ${j\choose q}=0$ for $q<0$ or $q>j$.
\end{lemma}

\begin{proof}  
These identities can be checked with the $zb$ package written for Mathematica.
For details on this we refer the reader to the wonderful book \cite{PWZ}.  
For completeness we give here the main calculation.  

The first identity is equivalent to 
\[
h(l_{1},l_{2}):=\sum_{p}\frac{2p(l_{1}+l_{2}){2l_{1}\choose l_{1}-p}{
2l_{2}\choose l_{2}-p}}{l_{1}l_{2}{2l_{1}\choose l_{1}}{2l_{2}\choose l_{2}}}=1.
\]
Let us denote
\[
f(l_{1},l_{2},p)=\frac{2p(l_{1}+l_{2}){2l_{1}\choose l_{1}-p}{
2l_{2}\choose l_{2}-p}}{l_{1}l_{2}{2l_{1}\choose l_{1}}{2l_{2}\choose l_{2}}}.
\]
The idea of the $zb$ method for our case is to write 
\begin{equation}\lbl{e:zb}
f(l_{1}+1,l_{2},p)-f(l_{1},l_{2},p)=g(l_{1},l_{2},p+1)-g(l_{1},l_{2},p).
\end{equation}
and this proves that for all $l_{1}\ge1$ one has $h(l_{1},l_{2})=h(1,l_{2})$.  
Since it is is immediate to show that $h(1,l_{2})=1$, the rest follows as 
soon as we know that $g(l_{1},l_{2},p)=0$ for $p=1$ and for large $p$.      

The whole point of the $zb$ method is to actually compute the function 
$g(l_{1},l_{2},p)$.   We will refer the reader for the details to 
\cite{PWZ} and will give here just the results obtained with Mathematica.  
\[
g(l_{1},l_{2},p)=-\frac{2p(p-1){2l_{1}+1 \choose l_{1}+p}{2l_{2}
-1\choose l_{2}-p}}{l_{1}(2l_{1}+1){2l_{1}\choose l_{1}}{2l_{2}\choose \l_{2}}}.
\]
Notice that for $p\ge \min\{l_{1}+2,l_{2}+1\}$, $g(l_{1},l_{2},p)=0$. One can 
directly check \eqref{e:zb}, by dividing both sides by $f(l_1,l_2,p)$
which reduces it to an identity in the field $\BQ(l_1,l_2,p)$.

For the second identity, as in the preceding argument, we want to show that 
\[
h(l_{1},l_{2}):=\sum_{p}\frac{(2p+1)(l_{1}+l_{2}+1){2l_{1}+1\choose l_{1}-p}{
2l_{2}+1\choose l_{2}-p}}{(2l_{1}+1)(2l_{2}+1){2l_{1}\choose l_{1}}{
2l_{2}\choose l_{2}}}=1.
\]
Defining 
\[
f(l_{1},l_{2},p)=\frac{(2p+1)(l_{1}+l_{2}+1){2l_{1}+1\choose l_{1}-p}{2l_{2}
+1\choose l_{2}-p}}{(2l_{1}+1)(2l_{2}+1){2l_{1}\choose l_{1}}{2l_{2}\choose l_{2}}},
\]
the corresponding companion in this case is 
\[
g(l_{1},l_{2},p)=-\frac{p^{2}(l_{2}+1)^{2}{2l_{1}+2\choose l_{1}-p+1}{2l_{2}
\choose l_{2}-p}}{{2l_{1}+2\choose l_{1}+1 }{2l_{2}\choose l_{2}}}.
\]
Equation \eqref{e:zb} is satisfied and $g(l_{1},l_{2},0)=0$ and 
$g(l_{1},l_{2},p)=0$ for $p\ge \min\{l_{1}+1,l_{2}+1 \}$.   This proves 
that $h(l_{1},l_{2})=h(0,l_{2})$.  Now, $h(0,l_{2})=1$ which ends the proof.  
\qedhere
\end{proof}

Before we state the next result, for a $C^{3}$ potential $V$, we define 
\begin{equation}\lbl{e:50}
\psi_{b,c}(x):=\int_{-2}^{2}\frac{V'(cx+b)-V'(cy+b)}{x-y}
\frac{dy}{\pi\sqrt{4-y^{2}}}.
\end{equation}

\begin{theorem}
\lbl{t:3}
Assume $V$ is an admissible potential on $\R$.  Then the equilibrium measure on $\R$ 
associated to $V$ has support the interval $[-2c+b,2c+b]$ if and only if 
$(c,b)$ is the unique absolute maximizer in $\R^{*}_{+}\times \R$ of  
\begin{equation}
\lbl{e:31}
H(c,b):=\log c-\frac{1}{2}\int_{-2}^{2}V(cx+b)\frac{dx}{\pi \sqrt{4-x^{2}}}
\end{equation}
and
\begin{equation}
\lbl{e:44}
\psi_{b,c}> 0\:\: \text{on a dense subset of}\:\: [-2,2]. 
\end{equation}
If in addition $V$ is a $C^{2}$ and piecewise $C^{3}$ potential on a neighborhood of the support $[-2c+b,2c+b]$, then  $(b,c)$ is a solution of 
\begin{equation}
\lbl{eq:bc}
\begin{cases}
\int_{-2}^{2}cxV'(cx+b)\frac{dx}{\pi\sqrt{4-x^{2}}}=2 \\ 
\int_{-2}^{2}V'(cx+b)\frac{dx}{\pi\sqrt{4-x^{2}}}=0.
\end{cases}
\end{equation}

In this case the equilibrium measure $\mu_{V}$ is given by 
\[
\mu_{V}(dx)=\mathbbm{1}_{[-2c+b,2c+b]}(x)\frac{\psi_{b,c}((x-b)/c)
\sqrt{4c^{2}-(x-b)^{2}}}{2c\pi}dx
\]
and 
\begin{equation}
\lbl{e:28-2}
\begin{split}
I_{V}& = -\log c +\int_{-2}^{2}\frac{V(cx+b)dx}{\pi\sqrt{4-x^{2}}} 
- \int_{0}^{c}s\left[\left(\int_{-2}^{2}
\frac{xV'(sx+b)dx}{2\pi \sqrt{4-x^{2}}} \right)^{2}+\left(  
\int_{-2}^{2}\frac{V'(sx+b)dx}{\pi \sqrt{4-x^{2}}} \right)^{2} \right]ds \\
& = V(b)-\log c-\int_{0}^{c}\frac{1}{s}\left[-1+\left(1-\int_{-2}^{2}
\frac{sxV'(sx+b)dx}{2\pi \sqrt{4-x^{2}}} \right)^{2}+\left(  
\int_{-2}^{2}\frac{sV'(sx+b)dx}{\pi \sqrt{4-x^{2}}} \right)^{2} \right]ds.
\end{split}
\end{equation}
\end{theorem}

\begin{proof}  
If the support of $\mu_{V}$ is the interval $[-2c+b,2c+b]$, 
we have to prove first that $(c,b)$ is the unique absolute maximizer of $H$. 
 The function $H$ appears in the literature as the $F$-functional of Mhaskar 
and Saff (see for instance \cite[page 194]{ST}) and for the sake of 
completeness we adapt the proof of this first part from there.    

Define the arcsine law of the interval $[-2c+b,2c+b]$ to be  
\[
\omega_{c,b}(dx)=\mathbbm{1}_{[-2c+b,2c+b]}(x)
\frac{dx}{\pi\sqrt{4c^{2}-(x-b)^{2}}}.
\]
A simple rescaling of equation \eqref{e:43}, shows that 
$\int\log|x-y|\omega_{c,b}(dy)\ge\log(c)$ for all $x$, with equality only for 
$x\in[-2c+b,2c+b]$. 

Integrating equation \eqref{e:var} against the measure $\omega_{c',b'}$ 
yields that 
\[
\int V(x)\omega_{c',b'}(dx)-2\int 
\int\log|x-y|\omega_{c',b'}(dx)\mu_{V}(dy)\ge C
\]
and thus, after interchanging the integrations, 
\[
\int V(x)\omega_{c',b'}(dx)-2\log(c')\ge C.
\]
Because \eqref{e:var} is equality quasi-everywhere on $[-2c+b,2c+b]$, this 
implies that we have equality in the above inequality for $c'=c$ and $b'=b$.
   In fact, this is the only case of equality as otherwise 
\[
C = \int V(x)\omega_{c',b'}(dx)-2\log(c')\ge \int V(x)\omega_{c',b'}(dx)
-2\int \int\log|x-y|\omega_{c',b'}(dx)\mu_{V}(dy) \ge C,
\]
hence  we must have that $\mu_{V}$ almost surely, $\log(c')=\int 
\log|x-y|\omega_{c',b'}(dy)$, which according to \eqref{e:43} is possible if
 and only if $\omega_{c',b'}$ is actually equal to $\omega_{c,b}$, or $c'=c$ 
and $b'=b$.  

From \eqref{e:31}, upon differentiation with respect to $c$ and $b$, we 
deduce that
\begin{equation}\lbl{e:45}
\int_{-2}^{2} cxV'(cx+b)\frac{dx}{\pi\sqrt{4-x^{2}}}=2 \quad\text{and}
\quad\int_{-2}^{2} V'(cx+b)\frac{dx}{\pi\sqrt{4-x^{2}}}=0
\end{equation}
which combined with \eqref{e:sol} proves \eqref{e:44}.  

To prove the converse, notice that because $(c,b)$ is a maximizer of $H$, 
we have \eqref{e:45}.   It is then clear that the $\mu_{V}$ solves equation 
\eqref{e:sol}.  What we have to prove is that this measure satisfies 
condition \eqref{e:var}.  To this end, it is sufficient to prove that for 
any $b'\in\R$ and $c'>0$ one has
\[
\int \left(V(x)-2\int\log|x-y|\mu_{V}(dy)\right)\omega_{c',b'}(dx)\ge C.
\]
Switching the integration in the double integral,  and performing some 
elementary calculations, this inequality becomes equivalent to 
\[
\int \frac{V(c'x+b')dx}{\pi\sqrt{4-x^{2}}}-2\int\int\log|x-y|
\omega_{c',b'}(dx)\mu_{V}(dy)\ge C.
\]
This inequality is equality for $c'=c$ and $b'=b$, and thus $C=-H(c,b)$.  
If $c'$ and $b'$ are arbitrary, the inequality is a consequence of the fact 
that the left hand side of this inequality is greater than or equal to 
$-H(c',b')$ which in turn is by the hypothesis $\ge -H(c,b)$. 

Identity \eqref{e:28-2} follows from Theorem~\ref{t:2} applied to 
$\tilde{V}(x)=V(cx+b)$.  
\qedhere
\end{proof}

In the case of even potentials, we know that the equilibrium measure is 
symmetric and thus in the preceding result we can always assume that $b=0$ 
and this deserves a special statement because of its simplicity.  

\begin{corollary}
\lbl{c:3}
If $V$ is a $C^{2}$, piecewise $C^{3}$ and even satisfying \eqref{e:V}, its equilibrium 
measure is supported on the interval $[-2c,2c]$ if and only if $c>0$ is 
the unique maximizer of 
\[
H(c)=\log c-\int_{0}^{2}\frac{V(cx)dx}{\pi\sqrt{4-x^{2}}}
\]
and
\[
\psi_{c}(x):=\int_{-2}^{2}\frac{V'(cx)-V'(cy)}{x-y}\frac{dy}{\pi\sqrt{4-y^{2}}}
\]
is positive on a dense set of $[-2,2]$. In particular $c$ solves
\begin{equation}\lbl{eq:calone}
\int_{-2}^{2}cxV'(cx)\frac{dx}{\sqrt{4-x^{2}}}=2.
\end{equation}
In this case the planar limit is 
\begin{equation}\lbl{eq:pe}
I_{V}=V(0)-\log c -\int_{0}^{c}\frac{1}{s}
\left[-1+\left(1-\int_{0}^{2}sxV'(sx)\frac{dx}{\pi
\sqrt{4-x^{2}}} \right)^{2}\right]ds.
\end{equation}
\end{corollary}

We point out here an interesting property, namely, that the solutions 
$(c,b)$ of the system \eqref{eq:bc} are critical points of the functional 
$I_{V}$ from \eqref{e:28-2}.  

\begin{proposition}
\lbl{prop.critp}
Let $V$ be a $C^{1}$ potential on $\R$ and consider
{\small
\[
\begin{split}
I_{V}(u,v)& = -\log u +\int_{-2}^{2}\frac{V(ux+v)dx}{\pi\sqrt{4-x^{2}}} 
- \int_{0}^{u}s\left[\left(\int_{-2}^{2}
\frac{xV'(sx+v)dx}{2\pi \sqrt{4-x^{2}}} \right)^{2}+\left(  
\int_{-2}^{2}\frac{V'(sx+v)dx}{\pi \sqrt{4-x^{2}}} \right)^{2} \right]ds. \\
\end{split}
\]
}
If $(c,b)$ satisfy 
\[
\int_{-2}^{2}\frac{V'(c x+b)dx}{\pi\sqrt{4-x^{2}}}=0,
\]
then 
\begin{equation}\lbl{e:cr1}
\frac{\partial I_{V}}{\partial v}\bigg|_{(c,b)}=0.
\end{equation}
If $(c,b)$ satisfy \eqref{eq:bc}, then 
\begin{equation}\lbl{e:cr2}
\frac{\partial I_{V}}{\partial u}\bigg|_{(c,b)}=0.
\end{equation}
In particular the critical points of $H$ from \eqref{e:31} are also 
critical points of $I_{V}$.
\end{proposition}

\begin{proof}
To see \eqref{e:cr1}, after  differentiating with respect to $v$, we need 
to show that 
\[
\frac{\partial I_{V}}{\partial v}\bigg|_{(c,b)}=-\int_{-2}^{2}
\frac{V'(c x+b)dx}{\pi\sqrt{4-x^{2}}}+\int_{0}^{c}\int_{-2}^{2}\int_{-2}^{2}
\frac{s(xy+4)V'(sx+b)V''(sy+b)}{4\pi^{2}\sqrt{(4-x^{2})(4-y^{2})}}dx\,dy\,ds=0. 
\]

Now we present the following result. 

\begin{lemma}
\lbl{l:2}  
If $U\in C^{1}([-2,2])$ or is a formal power series, then the 
following holds
\begin{equation}\lbl{eq:eq2}
\int_{0}^{1}\int_{-2}^{2}\int_{-2}^{2}\frac{s(xy+4)U(sx)U'(sy)}{
4\pi^{2}\sqrt{(4-x^{2})(4-y^{2})}}dx\,dy\,ds=\int_{-2}^{2}\frac{U(x)}{
\pi\sqrt{4-x^{2}}}dx\int_{0}^{1}\int_{-2}^{2}\frac{sU'(sy)}{\pi\sqrt{4-y^{2}}}
dy\,ds.
\end{equation}
In particular, if $U$ satisfies,  
\[
\int_{-2}^{2}\frac{U(x)}{\pi\sqrt{4-x^{2}}}dx=0,
\]
then 
\[
\int_{0}^{1}\int_{-2}^{2}\int_{-2}^{2}\frac{s(xy+4)U(sx)U'(sy)}{4\pi^{2}
\sqrt{(4-x^{2})(4-y^{2})}}dx\,dy\,ds=0.
\]
\end{lemma}

\begin{proof}  By polarization, it suffices to show that for any two 
$C^{1}$ potentials or formal power series, $U_{1}$ and $U_{2}$ on $[-2,2]$, 
we have that 
\begin{multline*}
\int_{0}^{1}\int_{-2}^{2}\int_{-2}^{2}\frac{s(xy+4)(U_{1}(sx)U_{2}'(sy)
+U_{2}(sx)U_{1}'(sy)}{4\pi^{2}\sqrt{(4-x^{2})(4-y^{2})}}dx\,dy\,ds \\ 
=\int_{-2}^{2}\frac{U_{1}(x)}{\pi\sqrt{4-x^{2}}}dx
\int_{0}^{1}\int_{-2}^{2}\frac{sU_{2}'(sy)}{\pi\sqrt{4-y^{2}}}dy\,ds  
+\int_{-2}^{2}\frac{U_{2}(x)}{\pi\sqrt{4-x^{2}}}dx\int_{0}^{1}\int_{-2}^{2}
\frac{sU_{1}'(sy)}{\pi\sqrt{4-y^{2}}}dy\,ds.
\end{multline*}
It is clear now that it suffices to check this for $U_{1}(x)=x^{n}$ and 
$U_{2}(x)=x^{m}$, which, with the help of \eqref{e:14-1}, becomes 
\begin{align*}
\frac{1}{n+m+1}&\left[ \frac{n}{4}{n \choose \frac{n}{2}}{
m+1 \choose \frac{m+1}{2}}+\frac{m}{4}{m \choose \frac{m}{2}}{
n+1 \choose \frac{n+1}{2}}+m{n \choose \frac{n}{2}}{
m-1 \choose \frac{m-1}{2}}+n{m \choose \frac{m}{2}}{
n-1 \choose \frac{n-1}{2}}\right] \\ &=\frac{m}{m+1}{
n \choose \frac{n}{2}}{m-1\choose \frac{m-1}{2}}+\frac{n}{n+1}{
m \choose \frac{m}{2}}{n-1\choose \frac{n-1}{2}}.
\end{align*}
Here we use the convention that ${a \choose b}=0$ if $b$ is not a 
nonnegative integer.   As long as $n$ and $m$ have the same parity, 
both sides of the above expression are $0$.  Also due to the symmetry 
in $n$ and $m$, it suffices to check this for $n=2k$ and $m=2l+1$.  
In this case, it is easy to prove that both sides are equal to 
\[
\frac{2l+1}{2l+2}{2k\choose k}{2l\choose l}.  \qedhere
\]
\end{proof}

Taking $U(x)=V'(cx+b)$ in the Lemma, after a simple change of variables,  
the rest follows. 

Equation \eqref{e:cr2} is clear from the fact that 
{\small
\[
I_{V}(u,v)=V(v)-\log u-\int_{0}^{u}\frac{1}{s}\left[-1+\left(1-\int_{-2}^{2}
\frac{sxV'(sx+v)dx}{2\pi \sqrt{4-x^{2}}} \right)^{2}+\left(  
\int_{-2}^{2}\frac{sV'(sx+v)dx}{\pi \sqrt{4-x^{2}}} \right)^{2} \right]ds.
\]
}
and thus the $u$-derivative vanishes under the condition of 
\eqref{eq:bc}. \qedhere
\end{proof}

%%%%%%%%%%%%%%%%%%%%%%%%%%%%%%%%%%%%%%%%%%%%%%%%%%%%%%%%%%%%%%%%%%%%%%%%%%%%%
%%%%%%%%%%%%%%%%%%%%%%%%%%%%%%%%%%%%%%%%%%%%%%%%%%%%%%%%%%%%%%%%%%%%%%%%%%%%%

\section{Examples and Computations with Analytic Matrix Models}
\lbl{s.appl}

\subsection{Cases of One-Cut Potentials}
\lbl{ss.general} 

With the result from Theorem~\ref{t:3}, it is instructive to recover the 
classical results (see \cite{ST} where weaker regularity conditions are 
required) which guarantee that there is a one interval support of the 
equilibrium measure.  

\begin{corollary}
\lbl{c:6}
Assume that a $C^{3}$ potential  $V$ satisfying \eqref{e:V} is either 
convex or even with $xV'(x)$ increasing on $[0,\infty)$.  Then the 
equilibrium measure has one interval support and the maximizer is 
non-degenerate (i.e. the Hessian of $H$ is invertible at the maximizer).    
In addition, the function $\psi_{c,b}$ is positive on $[-2,2]$.  
\end{corollary}

\begin{proof}  
First, we need to check that the function $H(b,c)$ has a 
unique maxima.  

In the case $V$ is convex, we show that $H$ is concave.  Indeed, the 
hessian of $H$ at $(c,b)$ is 
\[
(\mathrm{Hess} H)(c,b)=\met{\displaystyle-\frac{1}{c^{2}}-\int_{-2}^{2}
\frac{x^{2}V''(cx+b)dx}{2\pi\sqrt{4-x^{2}}} & \displaystyle -\int_{-2}^{2}
\frac{x V''(cx+b)dx}{2\pi\sqrt{4-x^{2}}} \\ \displaystyle-\int_{-2}^{2}
\frac{x V''(cx+b)dx}{2\pi\sqrt{4-x^{2}}} & \displaystyle-\int_{-2}^{2}
\frac{V''(cx+b)dx}{2\pi\sqrt{4-x^{2}}}}
\]
and strict concavity is equivalent to 
\[
\begin{split}
\frac{1}{c^{2}}+\int_{-2}^{2}&\frac{x^{2}V''(cx+b)dx}{2\pi\sqrt{4-x^{2}}}>0
\quad \text{and} \\
\frac{1}{c^{2}}\int_{-2}^{2}\frac{V''(cx+b)dx}{2\pi\sqrt{4-x^{2}}}+ & 
\int_{-2}^{2}\frac{V''(cx+b)dx}{2\pi\sqrt{4-x^{2}}}\int_{-2}^{2}
\frac{x^{2}V''(cx+b)dx}{2\pi\sqrt{4-x^{2}}}-\left(\int_{-2}^{2}
\frac{xV''(cx+b)dx}{2\pi\sqrt{4-x^{2}}}\right)^{2}>0. 
\end{split}
\]
The only way either of these fail is if $V''(cx+b)=0$ for all $x\in[-2,2]$, 
which  implies $V(x)=Ax+B$ for some constants $A, B$ and all 
$x\in[-2c+b,2c+b]$.  This in turn results with $F(c',b')=\log c'-B$ for all 
$c'<c$ which contradicts the assumption that $(c,b)$ is a maximizer of $F$.  

On the other hand, one can easily check that $H$ is concave on $(0,\infty)
\times \R$.  This combined with strict concavity near $(c,b)$ implies that 
the maximizer is unique.  

In the case $V$ is even and $xV'(x)$ is increasing, we may assume that $b=0$ 
and thus the function $H$ becomes a function of one variable with 
\[
H'(c)=\frac{1}{c}-\int_{-2}^{2}\frac{xV'(cx)dx}{2\pi\sqrt{4-x^{2}}}
=\frac{1}{c}\left(1-\int_{0}^{2}\frac{cxV'(cx)dx}{\pi\sqrt{4-x^{2}}} \right)
\]
Now, since the function $xV'(x)$, is increasing, one can see that $cH'(c)$ 
is decreasing and thus there is only at most one critical point of $H$.  
On the other hand, one can check that there is a maximizer of $H(c)$, hence 
we deduce that there is a unique such maximizer.   

In addition to this, the Hessian of $H(c,b)$ at the maximizer $(c,0)$ is
\[
(\mathrm{Hess} H)(c,b)=\met{\displaystyle-\frac{1}{c^{2}}-\int_{-2}^{2}
\frac{x^{2}V''(cx)dx}{2\pi\sqrt{4-x^{2}}} & 0 \\ 0& 
\displaystyle-\int_{-2}^{2}\frac{V''(cx)dx}{2\pi\sqrt{4-x^{2}}}}
\]
 Now,  using the fact that $H'(c)=0$ and a simple integration by parts 
reveals that
\[
1=\int_{-2}^{2}\frac{cxV'(cx)dx}{2\pi\sqrt{4-x^{2}}}=\int_{-2}^{2}
\frac{c^{2}V''(cx)\sqrt{4-x^{2}}dx}{2\pi}=\int_{-2}^{2}
\frac{4c^{2}V''(cx)dx}{2\pi\sqrt{4-x^{2}}}-\int_{-2}^{2}\frac{c^{2}x^{2}
V''(cx)dx}{2\pi\sqrt{4-x^{2}}}
\]
which implies
\[\tag{*}
4\int_{-2}^{2}\frac{V''(cx)dx}{2\pi\sqrt{4-x^{2}}}=\frac{1}{c^{2}}
+\int_{-2}^{2}\frac{x^{2}V''(cx)dx}{2\pi\sqrt{4-x^{2}}}.
\]
Now, if we denote $g(x)=xV'(x)$, then $g$ is an increasing function on 
$[0,\infty]$ and therefore $x^{2}V''(x)=xg'(x)-g(x)>-g'(x)$ for all $x>0$.  
In particular we obtain that $c^{2}x^{2}V''(cx)>-g(cx)$ for all $x\in[0,2]$. 
 Furthermore, from the equation determining $c$, we get
\[
1=\int_{0}^{2} \frac{g(cx)dx}{\pi\sqrt{4-x^{2}}}
\]
which in turn implies 
\[
\frac{1}{c^{2}}+\int_{0}^{2}\frac{c^{2}x^{2}V''(cx)dx}{\pi\sqrt{4-x^{2}}}> 
\frac{1}{c^{2}}-\frac{1}{c^{2}}\int_{0}^{2} \frac{g(cx)dx}{\pi\sqrt{4-x^{2}}}=0
\]
and this means that quantities in  (*) are positive, thus the Hessian 
of $F$ at $(c,0)$ is non-degenerate.  

Having checked the uniqueness of the maximizer, we need to check the other 
condition.  In the case of convex potentials, the non-negativity of 
$\psi_{c,b}$ follows from the fact that $\frac{V'(cx+b)-V'(cy+b)}{x-y}\ge0$ 
for all $x,y\in[-2,2]$.   Furthermore, $\psi_{c,b}(x)=0$, enforces 
$V'(cx+b)=V'(cy+b)$ for all $y\in[-2,2]$, which in turn yields $V(c\cdot+b)$ 
is constant on $[-2,2]$, something which is contradicted by the assumption 
that $(c,b)$ is a maximizer of $H(c,b)$.  Hence we actually obtain the 
stronger conclusion, namely $\psi_{c,b}(x)>0$ on $[-2,2]$.   

In the case $V$ is even and $xV'(x)$ increases on $[0,\infty]$, one can show 
that $\psi_{c}$ is an even function and with simple manipulations of 
integrals that 
\[
\psi_{c}(x)=\int_{0}^{2}\frac{xV'(cx)-yV'(cy)}{x^{2}-y^{2}}
\frac{dy}{\pi\sqrt{4-y^{2}}}
\]
which makes clear that $\psi_{c}(x)>0$ for all $x\in[-2,2]$.
\qedhere
\end{proof}

\subsection{Analytic planar limits of various even potentials}
\lbl{sub.examples1}

In this section we explicitly compute the planar limit of some
1-cut potentials, illustrating the formulas of Section \ref{s:5}.
A typical example is the case where $V$ is a smooth potential 
which is analytic near the support of the 
equilibrium measure.  

The easiest to deal with is the case of even potentials because in this 
case we can invoke Corollary~\ref{c:3} and reduce the problem of determining 
the support of the equilibrium measure to the maximization of a 
function of a single variable.    
In this case the planar limit is actually a one variable 
function of the right endpoint $2c$ of the equilibrium measure.  

Assume that $V$ is an even potential such that it has a power series 
expansion valid on a neighborhood of the support:
\begin{equation}
\lbl{e:60}
V(x)=\sum_{n=1}^{\infty}a_{2n}\frac{x^{2n}}{2n}.
\end{equation}
In this case, from Corollary~\ref{c:3} we learn that
\[
H(c)=\log c-\frac{1}{2}\sum_{n=1}^{\infty}\frac{a_{2n}c^{2n}}{2n}
\int_{-2}^{2}\frac{x^{2n}dx}{\pi\sqrt{4-x^{2}}}=\log c-\frac{1}{2}
\sum_{n=1}^{\infty}\frac{a_{2n}c^{2n}}{2n}{2n \choose n}
\]
where in the last equality we used equation \eqref{e:14-1}.   
The critical points of this function satisfy \eqref{eq:calone} which becomes
\begin{equation}\lbl{e:61}
\sum_{n=1}^{\infty}a_{2n}c^{2n}{2n \choose n}=2.
\end{equation}
If $c$ is the maximizer of $F$,  then, again from Corollary~\ref{c:3} and 
\eqref{e:14-1}, the planar limit is given by
\[
I_{V}=-\log c+\int_{0}^{c}\frac{1}{t}\left[-1+\left(1-\frac{1}{2}
\sum_{n=1}^{\infty}a_{2n}t^{2n}{2n \choose n} \right)^{2}\right]dt.
\]

\begin{example}
\lbl{ex:1}  
For $V(x)=a_{2n}\frac{x^{2n}}{2n}$, with 
$a_{2n}>0$, and $n\ge1$, the support of the equilibrium measure is 
$[-2c,2c]$, where 
\[
c=\left( \frac{a_{2n}}{2}{2n \choose n} \right)^{-\frac{1}{2n}}.
\]
In this case,  the equilibrium measure is 
\[
\mu_{V}(dx)=\mathbbm{1}_{[-2c,2c]}(x)\frac{1}{{2\pi c}}\psi_{c}(x/c)
\sqrt{4c^{2}-x^{2}}dx, \qquad \psi_{c}(x)=a_{2n}c^{2n-1}
\sum_{l=0}^{n-1}{2l\choose l}x^{2(n-l-1)}
\]
and the planar limit is 
\[
I_{V}=\frac{\log a_{2n}}{2n}+\frac{\log \left({2n\choose n}/2\right)}{2n}
+\frac{3}{4n}.
\]
\end{example} 

To see this, one has to realize that \eqref{e:61} becomes in this case 
\[
a_{2n}c^{2n}{2n\choose n}=2
\]
which has only one positive solution, and this is the maximizer of 
$H(c)=\log c-\frac{a_{2n}c^{2n}}{4n}{2n\choose n}$.  The rest of the 
equalities are straightforward calculations.  

It is worth pointing out that in this example the potential is convex and thus, 
the equilibrium measure must be supported on a single interval.  

For $n=1$, we recover the semicircular law.

\begin{example}
\lbl{ex:2}  
Assume $V(x)=a_{2n}\frac{x^{2n}}{2n}
+a_{2m}\frac{x^{2m}}{2m}$ with $a_{2m}>0$ and $1\le n\le m$.  In this 
case the equilibrium measure has a single interval support if and only if 
\begin{equation} \lbl{e:63}
a_{2n}\ge-C_{nm}a_{2m}^{m/n}
\end{equation}
where 
\begin{equation}\lbl{e:64}
C_{nm}=K_{nm}\left(\frac{2}{{2m\choose m}-{2n\choose n}K_{nm}} 
\right)^{\frac{m-n}{n}}\text{  with  } K_{nm}=\min_{t\in[0,4]}
\frac{\sum_{l=0}^{m-1}{2l\choose l}t^{m-l-1}}{\sum_{l=0}^{n-1}{2l
\choose l}t^{n-l-1}}.
\end{equation}
In this case, the support of $\mu_{V}$ is $[-2c,2c]$ where $c$ is the 
unique positive solution to
\begin{equation}\lbl{e:62}
a_{2n}c^{2n}{2n\choose n}+a_{2m}c^{2m}{2m\choose m}=2,
\end{equation}
the equilibrium measure is
\begin{align*}
\mu_{V}(dx)& =\frac{1}{2\pi c}\mathbbm{1}_{[-2c,2c]}(x)\psi_{c}(x/c)
\sqrt{4c^{2}-x^{2}}dx\\
\psi_{c}(x)&=a_{2n}c^{2n-1}\sum_{l=0}^{n-1}{2l\choose l}x^{2(n-l-1)}
+a_{2m}c^{2m-1}\sum_{l=0}^{m-1}{2l\choose l}x^{2(m-l-1)}
\end{align*}
and the planar limit is
\begin{multline*}
I_{V} = -\log c+
\frac{c^{2 n} a_{2n}}{2n}{2n\choose n}+\frac{c^{2 m}  
a_{2m} }{2 m }{2m\choose m} \\ -\frac{c^{2 (n+ m)} 
a_{2n} a_{2m} }{4 (n+m)}{2n\choose n} {2m\choose m} %\\
-\frac{c^{4 n} a_{2n}^2 }{16 n}{2n\choose n}^2
-\frac{c^{4 m} a_{2m}^2 }{16 m}{2m\choose m}^2.
\end{multline*}
\end{example}

To prove these, we need to look at the critical equation \eqref{e:61} 
and notice that for
\[
f(c)=a_{2n}c^{2n}{2n\choose n}+a_{2m}c^{2m}{2m\choose m}-2,
\]
one has
\[ 
f'(c)=2c^{2n-1}\left(na_{2n}{2n\choose n}+ma_{2m}c^{2(m-n)}{2m\choose m}
\right).
 \]
It is clear that $f'$ has at most one positive root.  If  $c_{0}>0$ is the 
positive root of $f'$, then $f'(c)<0$ for $0<c<c_{0}$ and $f'(c)>0$ for 
$c >c_{0}$.  If $f'$ does not have any positive root, then $f'>0$.   
Since $f(0)=-2$ and $f(\infty)=\infty$, it follows that $f$ must have a 
unique zero which in turn is the unique maxima of $H(c)$.  

Now having proved that there is a unique maxima, we need to check the 
second condition from \eqref{c:3}.  That boils down to 
\[
\psi_{c}(x)\ge0
\]  
on $[-2,2]$ with strict inequality on a dense set.  This is equivalent to 
\[
a_{2n}\ge -a_{2m}c^{2m-2n}\frac{\sum_{l=0}^{m-1}{2l\choose l}
x^{2(m-l-1)}}{\sum_{l=0}^{n-1}{2l\choose l}x^{2(n-l-1)}}
\]
for all $x\in[-2,2]$ which in turn is satisfied if and only if 
\[
a_{2n}\ge -a_{2m}c^{2m-2n}K_{nm}
\]
where $K_{nm}$ is defined by \eqref{e:64}.  On the other hand from the 
critical equation \eqref{e:62} replacing $a_{2n}$, we arrive at
\[
\frac{2-a_{2m}c^{2m}{2m\choose m}}{{2n\choose n}}\ge -a_{2m}c^{2m}K_{nm}
\]
and thus, after noting that $K_{nm}\le {2m \choose m}/{2n\choose n}$, is 
the same as 
\[
c\le \left( \frac{2}{a_{2m}({2m\choose m}-K_{nm}{2n\choose n})} 
\right)^{1/(2m)}. 
\]
For the function $f$, we know that $f(x)\le 0$ if and 
only if $x\le c$.  Thus, we have the second condition in Corollary 
\ref{c:3} satisfied  if and only if 
\[
f\left(\left( \frac{2}{a_{2m}({2m\choose m}-K_{nm}{2n\choose n})} 
\right)^{1/(2m)}\right)\ge0
\]
which is equivalent to equation \eqref{e:63}.  

The constant $K_{nm}$ from  \eqref{e:64} depends only on $n$ and $m$.  
It can be explicitly computed in the case $n=1$ and any $m\ge 2$ as the 
minimizer is $t=0$ and thus $K_{1m}={2m-2\choose m-2}/{2m-2\choose m-1}$ 
and then a simple rearrangement reveals that
\[
C_{1m}=\frac{{2m-2\choose m-1}}{{2m-2\choose m}^{m-1}}.
\]
In general, it does not seem that one can find an explicit algebraic expression 
of the minimizer in \eqref{e:64}.   For the case of $n=2$ and $m=3$, we have 
an exact solution as the minimizer in the expression there is $t=-2+\sqrt{6}$ 
and then in this case $K_{23}=2\sqrt{6}-2$ which produces
\[
C_{23}=\sqrt{4+\sqrt{6}}.
\]  

The root $c$ from equation \eqref{e:62} does not have a simple 
representation in general.  However, in some cases it can be solved explicitly.  
For example if $m=2n$, one has
\[
c=\left(\frac{-a_{2n}{2n\choose n}+\sqrt{a_{2n}^{2}{2n\choose n}^{2}
+8a_{4n}{4n\choose 2n}}}{2a_{4n}{4n\choose 2n}} \right)^{\frac{1}{2n}}
\]
and similarly there are algebraic expressions in the case $m=3n$ or 
$m=3k, n=2k$ and also $m=4n$ or $m=4k,n=3k$, but we omit the lengthy 
formulae here. 

\begin{corollary}
\lbl{cor.quartic}
For the quartic potential
$$
V(x)=a_{2}\frac{x^{2}}{2}+a_{4}\frac{x^{4}}{4}
$$ 
the equilibrium measure has a single interval 
support if and only if $a_{2}\ge-2\sqrt{a_{4}}$ in which case 
\begin{eqnarray*}
h(x)&=&\frac{1}{2\pi}\mathbbm{1}_{[-2c,2c]}(x)
\sqrt{4c^{2}-x^{2}}(b_{2}+a_{4}x^{2})
\\
c&=&
\sqrt{\frac{-a_{2}+\sqrt{a_{2}^{2}+12a_{4}}}{6a_{4}}}
\\
b_{2}&=&
\frac{2a_{2}+\sqrt{a_{2}^{2}+12a_{4}}}{3} 
\\
I_{V}&=&
\frac{3}{8}+\frac{1}{2}\log \left(\frac{a_{2}+\sqrt{a_{2}^{2}
+12a_{4}}}{2}\right)+\frac{-a_{2}^4-36 a_{2}^2 a_{4}+162 a_{4}^2
+(a_{2}^3 +30 a_{2} a_{4}) \sqrt{a_{2}^2+12 a_{4}}}{432 a_{4}^2}
\end{eqnarray*}
\end{corollary}

We should point out that this example appears for instance in 
\cite{Jo}.

%%%%%%%%%%%%%%%%%%%%%%%%%%%%%%%%%%%%%%%%%%%%%%%%%%%%%%%%%%%%%%%%%%%%%%%%%%%%%
%%%%%%%%%%%%%%%%%%%%%%%%%%%%%%%%%%%%%%%%%%%%%%%%%%%%%%%%%%%%%%%%%%%%%%%%%%%%%

\section{Matching Formal and Analytic Matrix Models}
\lbl{s:13}

In this section we will prove Theorem \ref{thm.2}. Our first task is
to match the analytic equations Equations \eqref{eq:bc}
of $(b,c)$ of a 1-cut potential with the equations 
\eqref{eq.RS} for $(\calR,\calS)$.
Consider
a 1-cut potential $V$ and its Taylor series expansion at $x=0$:
\[
V(x)=\sum_{n=1}^{\infty}a_{n}\frac{x^{n}}{n}.
\]
Using the {\em key identity}
%%%%%%%% see Mathematica file: ChebyshevCheck.nb
\begin{equation}
\lbl{eq.key}
\int_{-2}^{2} \frac{x^n dx}{\pi \sqrt{4-x^{2}}}=\begin{cases}
\binom{n}{n/2} & \text{if $n$ is even} \\
0 & \text{if $n$ is odd}
\end{cases}
\end{equation}
and interchanging summation and integration, \eqref{e:14-1} gives 
\begin{align*}
\int_{-2}^{2}cxV'(cx+b)\frac{dx}{\pi \sqrt{4-x^{2}}}&
= \sum_{n\ge1}a_{n}\sum_{j\ge1}{n-1\choose 2j-1}{2j \choose j}c^{2j}b^{n-2j} \\
\int_{-2}^{2}V'(cx+b)\frac{dx}{\pi \sqrt{4-x^{2}}}&
= \sum_{n\ge1}a_{n}\sum_{j\ge0}{n-1\choose 2j}{2j \choose j}c^{2j}b^{n-2j-1}
\end{align*}
Then, Equation \eqref{eq:bc} gives the 
system of non-linear equations for $(b,c)$
\begin{equation}
\lbl{eq:bc2}\displaystyle
\begin{cases}
 \sum_{n\ge1}a_{n}\sum_{j\ge1}{n-1\choose 2j-1}{2j \choose j}c^{2j}b^{n-2j}=2  \\
 \sum_{n\ge1}a_{n}\sum_{j\ge0}{n-1\choose 2j}{2j \choose j}c^{2j}b^{n-2j-1}=0
\end{cases}
\end{equation} 
Following the notation of \cite{BDG}, let us use the change of variables 
$(b,c^2)=(S,R)$ as in Equation \eqref{eq.bcRS}.
Then, $(R,S)$ satisfy the system of equations
\begin{equation}
\lbl{eq:bc22}\displaystyle
\begin{cases}
 \sum_{n\ge1}a_{n}\sum_{j\ge1}{n-1\choose 2j-1}{2j \choose j}R^{j}S^{n-2j}=2  \\
 \sum_{n\ge1}a_{n}\sum_{j\ge0}{n-1\choose 2j}{2j \choose j}R^{j}S^{n-2j-1}=0
\end{cases}
\end{equation} 
Consider now the 1-cut potential 
\[
V(x)=\frac{x^{2}}{2}-\sum_{n=1}^{\infty}a_{n}\frac{x^{n}}{n}
\]
Then, Equation \eqref{eq:bc22} gives the system of non-linear 
equations for $(R,S)$
\begin{equation}\lbl{eq:200}
\begin{cases}
2R=2+ \sum_{n\ge1}a_{n}\sum_{j\ge1}{n-1\choose 2j-1}{2j \choose j}R^{j}S^{n-2j}  \\
S= \sum_{n\ge1}a_{n}\sum_{j\ge0}{n-1\choose 2j}{2j \choose j}R^{j}S^{n-2j-1}
\end{cases}
\end{equation}
Using 
\[
{n-1\choose 2j-1}{2j \choose j}=2{n-1\choose j-1}{n-j\choose j}
\]
it follows that $(R,S)$ satisfy  the system of non-linear 
equations \eqref{eq.RS}. 

Observe that for a fixed admissible potential $V$, Equation \eqref{eq.RS}
may have none or more than one real solutions for $(R,S)$ but for small 
parameters $\mathbf{a}=(a_{1},a_{2},\dots )$ in some $\ell^{1}_{r}$ for small 
enough $r$, $R$ and $S$ become analytic functions of $\mathbf{a}$ (see 
Theorem \ref{t:4}).   {\em However},
it always has a unique formal solution $(\calR,\calS) \in 
(1+\calA^+,\calA^+)$. 

This proves that $\calR=R$ and $\calS=S$
in Theorem \ref{thm.2}.

To finish the proof of Theorem \ref{thm.2}, 
we need to prove that the coefficient of any 
monomial $a_{1}^{n_{1}}\dots a_{k}^{n_{k}}$ from the power series $\calF_{0}$ and 
$F_{0}$ are equal.   The important point here is the fact that 
the each such monomial involves finitely many $a_{1},a_{2},\dots, a_{k}$ and 
thus we may assume that all the 1-cut potential is actually a polynomial.

Now, assume that $a_{n}$ are all $0$ for $n\ge k$ and consider potentials 
of the form 
\[
V(x)=\frac{x^{2}}{2}-\left(\sum_{n=1}^{k} a_{n}\frac{x^{n}}{n} \right) 
+\frac{x^{2k+2}}{2k+2}.  
\]
For small real parameters $a_{1},a_{2},\dots a_{k}$ the functions
\[
g_{N}(a_{1},a_{2},\dots,a_{k})=\frac{1}{N^{2}}
\log\frac{\int_{\mathcal{H}_{N}}\exp(-N\Tr(V(M)))dM}{
\int_{\mathcal{H}_{N}}\exp(-N\Tr(M^{2}/2)))dM}
\]
are analytic in $a_{1},\dots, a_{k}$ on a neighborhood of $0\in \R^{k}$.  
Since the limit $g_{\infty}$ exists, the limit is going to be also an analytic 
function in these variables.  This means that at the level of power series 
the coefficients must converge to the coefficients of the limit.   

On one hand expanding the $g_{N}$ in power series, the limiting coefficient 
of $a_{1}^{n_{1}}\dots a_{k}^{n_{k}}$ is exactly the corresponding coefficient 
from the formal model.  On the other hand,   the limiting function $g_{\infty}$ 
is obtained via the potential theory and using the perturbation theory from 
Section~\ref{s:pert}, it is easy to see that $c,b$, the solution of the 
system \eqref{eq:200} are actually analytic functions of $a_{1},\dots, a_{k}$.  
In particular it means that the planar limit $F_{0}$ is equal to $g_{\infty}$ 
from \eqref{eq.Fformula} 
and is analytic,  thus concluding the proof.  
\qed

\begin{remark}
From now on, whenever we have a formal potential $\mathcal{V}
=\frac{x^{2}}{2}-\sum_{n=1}^{\infty}a_{n}\frac{x^{n}}{n}$, we will use $b,c$ 
as the solution to \eqref{eq:bc2} which has a unique solution in 
$(c,b)\in (1+\calA^{+},\calA^{+})$.  
\end{remark}

%%%%%%%%%%%%%%%%%%%%%%%%%%%%%%%%%%%%%%%%%%%%%%%%%%%%%%%%%%%%%%%%%%%%%%%%%%%%%
%%%%%%%%%%%%%%%%%%%%%%%%%%%%%%%%%%%%%%%%%%%%%%%%%%%%%%%%%%%%%%%%%%%%%%%%%%%%%

\section{The planar limit $\calF_0(t)$ in terms of $\calR(t)$ and 
$\calS(t)$}
\lbl{sec.F0RS}

This section is devoted to the proofs of Theorems \ref{thm.Fe} and 
\ref{thm.Ff}.  After we discuss the proofs we give a main consequences 
of these formulae, namely the fact that the planar limit enjoys 
algebricity in some cases which allows complete description of the 
asymptotics of the coefficients of $\calF_{0}$.    

\subsection{Proof of Theorem \ref{thm.Fe}}
\lbl{sub.thmFe}

%Once we solve the system \eqref{eq:200}, 
%we define the planar limit via \eqref{e:28-2} as
%\begin{equation}\lbl{e:2020}
%F_{0}=3/4+ \log c 
%-\int_{-2}^{2}\frac{V(cx+b)dx}{\pi\sqrt{4-x^{2}}} 
%+ \int_{0}^{c}s\left[\left(\int_{-2}^{2}
%\frac{xV'(sx+b)dx}{2\pi \sqrt{4-x^{2}}} \right)^{2}+\left(  
%\int_{-2}^{2}\frac{V'(sx+b)dx}{\pi \sqrt{4-x^{2}}} \right)^{2} \right]ds. 
%\end{equation}
%This is going to be an element of $\mathcal{A}$ and notice that is 
%entirely defined via a potential theory problem.  

In this Section we will prove Theorem \ref{thm.Fe} and the first part of 
Remark \ref{rem.Vef}. In this section, it will be convenient
to use the 1-cut potentials  $\ti \calV_{e}$ and $\calV_{e}$ given by
\begin{eqnarray*}
\tilde{\calV}_{e}(t,x) &=& \frac{x^{2}}{2t}-\sum_{n\ge1}\frac{a_{n}x^{n}}{n}  
\\
\calV_{e}(t,x)&=&\frac{x^{2}}{2}-\sum_{n\ge1}\frac{a_{n}t^{n/2}x^{n}}{n}.
\end{eqnarray*}

For simplicity of notation in this section we will drop the dependence on $e$ from 
the writing of  $\calV_{e}$ and $\tilde{\calV}_{e}$.  

We start by setting $c(t), b(t)$, and $\tilde{c}(t),\tilde{b}(t)$ to be 
the power series solutions to \eqref{eq:bc2} corresponding to potentials 
$\calV$, respectively $\tilde{\calV}$.  From the fact that $\calV(t,x)
=\tilde{\calV}(t,\sqrt{t}x)$, we easily get that
\begin{equation}\lbl{e:2000b}
\tilde{c}(t)=\sqrt{t}c(t)\text{ and } \tilde{b}(t)=\sqrt{t}b(t).
\end{equation}

Then $\tilde{\calV}(t,x)=\frac{x^{2}}{2t}-\calW(x)$ and the system 
satisfied by $\tilde{c}(t),\tilde{b}(t)$ is given by 
\begin{equation}\lbl{eq:bc4}
\begin{cases}
\int_{-2}^{2}\tilde{c}(t)x\tilde{\calV}'(t,\tilde{c}(t)x+\tilde{b}(t))
\frac{dx}{\pi\sqrt{4-x^{2}}}=2  & \\ 
\int_{-2}^{2}\tilde{\calV}'(t,\tilde{c}(t)x+\tilde{b}(t))\frac{dx}{
\pi\sqrt{4-x^{2}}}=0. 
\end{cases}
\end{equation}
where the derivative $\tilde{\calV}'(t,x)$ is taken with respect to $x$.  
Set now, 
\begin{multline*}
\calI(t)=-\log \tilde{c}(t) +\int_{-2}^{2}\frac{\tilde{\calV}(t,
\tilde{c}(t)x+\tilde{b}(t))dx}{\pi\sqrt{4-x^{2}}} 
\\ - \int_{0}^{\tilde{c}(t)}s\left[\left(\int_{-2}^{2}
\frac{x\tilde{\calV}'(t,sx+\tilde{b}(t))dx}{2\pi \sqrt{4-x^{2}}} 
\right)^{2}+\left(  
\int_{-2}^{2}\frac{\tilde{\calV}'(t,sx+\tilde{b}(t))dx}{\pi \sqrt{4-x^{2}}} 
\right)^{2} \right]ds.
\end{multline*}
Taking the derivative with respect to $t$,
\[
\begin{split}
\calI'(t)& =-\frac{\tilde{c}'(t)}{\tilde{c}(t)}+\int_{-2}^{2}
\frac{(\tilde{c}'(t)x+\tilde{b}(t)))\tilde{\calV}'(t,\tilde{c}(t)x
+\tilde{b}(t))dx}{\pi\sqrt{4-x^{2}}}+\int_{-2}^{2}
\frac{\dot{\tilde{\calV}}(t,\tilde{c}(t)x
+\tilde{b}(t))dx}{\pi\sqrt{4-x^{2}}} \\ 
& -\tilde{c}'(t)\tilde{c}(t)\left[\left(\int_{-2}^{2}
\frac{x\tilde{\calV}'(t,\tilde{c}(t)x
+\tilde{b}(t))dx}{2\pi \sqrt{4-x^{2}}} \right)^{2}+\left(  
\int_{-2}^{2}\frac{\tilde{\calV}'(t,\tilde{c}(t)x
+\tilde{b}(t))dx}{\pi \sqrt{4-x^{2}}} \right)^{2} \right] \\
&-2b'(t)\int_{0}^{\tilde{c}(t)}\int_{-2}^{2}\int_{-2}^{2}
\frac{s(xy+4)\tilde{\calV}'(t,sx+\tilde{b}(t))
\tilde{\calV}''(t,sy+\tilde{b}(t))}{4\pi^{2}
\sqrt{(4-x^{2})(4-y^{2})}}dx\,dy\,ds \\ 
&-2\int_{0}^{\tilde{c}(t)}\int_{-2}^{2}\int_{-2}^{2}
\frac{s(xy+4)\tilde{\calV}'(t,sx+\tilde{b}(t))
\dot{\tilde{\calV}}'(t,sy+\tilde{b}(t))}{4\pi^{2}
\sqrt{(4-x^{2})(4-y^{2})}}dx\,dy\,ds,
\end{split}
\]
where $\dot{\tilde{\calV}}(t,x)$ is the derivative with respect to $t$.  
Since $\dot{\tilde{\calV}}(t,x)=-\frac{x^{2}}{2t^{2}}$, the system 
\eqref{eq:bc4} and Lemma~\ref{l:2}, we can simplify this to 
\begin{eqnarray*}
\calI'(t)&=&-\frac{1}{2t^{2}}\int_{-2}^{2}\frac{(\tilde{c}(t)x
+\tilde{b}(t))^{2}dx}{\pi\sqrt{4-x^{2}}} \\
& & +\frac{2}{t^{2}}\int_{0}^{\tilde{c}(t)}\int_{-2}^{2}
\int_{-2}^{2}\frac{s(xy+4)\tilde{\calV}'(t,sx+\tilde{b}(t))(
sy+\tilde{b}(t))}{4\pi^{2}\sqrt{(4-x^{2})(4-y^{2})}}dx\,dy\,ds\\
&=& -\frac{2\tilde{c}(t)^{2}+\tilde{b}(t)^{2}}{2t^{2}}
+\frac{1}{t^{2}}\int_{0}^{\tilde{c}(t)}\int_{-2}^{2}
\frac{s(sx+2\tilde{b}(t))\tilde{\calV}'(
t,sx+\tilde{b}(t))}{\pi\sqrt{4-x^{2}}}dx\,ds.
\end{eqnarray*}
Next, observe that for any continuous function $f:[-2c,2c]\to\R$ with 
$c>0$, one has 
\begin{eqnarray*}
\int_{0}^{c}\int_{-2}^{2}\frac{s(sx+2b)f(sx)}{\pi\sqrt{4-x^{2}}}dx\,ds & 
\underset{x=cy/s}{=} & c\int_{0}^{c}\int_{-2s/c}^{2s/c}\frac{s(cy+2b)f(cy)}{
\pi\sqrt{4s^{2}-c^{2}y^{2}}}dy\,ds \\ & \underset{\text{Fubini}}{=} &  
c\int_{-2}^{2}\int_{c|y|/2}^{c}\frac{s(cy+2b)f(cy)}{\pi
\sqrt{4s^{2}-c^{2}y^{2}}}ds\,dy \\ &=&  c^{2}\int_{-2}^{2}
\frac{(cy+2b)f(cy)\sqrt{4-y^{2}}}{4\pi}dy\\ & 
\underset{y=(z-b)/c}{=} & \frac{1}{4\pi }
\int_{-2c+b}^{2c+b}(z+b)f(z-b)\sqrt{4c^{2}-(z-b)^{2}}dz.
\end{eqnarray*}
Going back to the previous equation we now have
\begin{multline*}
\int_{0}^{\tilde{c}(t)}\int_{-2}^{2}\frac{s(sx+2\tilde{b}(t))
\tilde{\calV}'(t,sx+\tilde{b}(t))}{\pi\sqrt{4-x^{2}}}dx\,ds 
\\ = \frac{1}{4\pi}\int_{-2\tilde{c}(t)+\tilde{b}(t)}^{2\tilde{c}(t)+\tilde{b}(t)}(
z+\tilde{b}(t))\tilde{\calV}'(t,z)\sqrt{4\tilde{c}(t)^{2}-(
z-\tilde{b}(t))^{2}}dz.
\end{multline*}
Take the derivative with respect to $t$ and observe 
\begin{multline*}
\frac{d}{dt} \int_{-2\tilde{c}(t)+\tilde{b}(t)}^{2\tilde{c}(t)+\tilde{b}(t)}(
z+\tilde{b}(t))\tilde{\calV}'(t,z)\sqrt{4\tilde{c}(t)^{2}-(z
-\tilde{b}(t))^{2}}dz  = \\ \int_{-2\tilde{c}(t)+\tilde{b}(t)}^{2\tilde{c}(t)
+\tilde{b}(t)}\tilde{b}'(t)\tilde{\calV}'(t,z)\sqrt{4\tilde{c}(t)^{2}
-(z-\tilde{b}(t))^{2}}dz \\ 
+\int_{-2\tilde{c}(t)+\tilde{b}(t)}^{2\tilde{c}(t)+\tilde{b}(t)}(z+\tilde{b}(t))
\tilde{\calV}'(t,z)\frac{4\tilde{c}'(t)-\tilde{b}'(t)(
\tilde{b}(t)-z)}{\sqrt{4\tilde{c}(t)^{2}-(z-\tilde{b}(t))^{2}}}dz \\ 
+ \int_{-2\tilde{c}(t)+\tilde{b}(t)}^{2\tilde{c}(t)+\tilde{b}(t)}(z+\tilde{b}(t))
\dot{\tilde{\calV}}'(t,z)\sqrt{4\tilde{c}(t)^{2}-(z-\tilde{b}(t))^{2}}dz =\\ 
\int_{-2\tilde{c}(t)+\tilde{b}(t)}^{2\tilde{c}(t)+\tilde{b}(t)}\tilde{\calV}'(t,z)
\frac{-2 \tilde{b}(t)^2 \tilde{b}'(t)+4 \tilde{c}(t)^2 \tilde{b}'(t)
+4 \tilde{c}'(t)\tilde{c}(t)\tilde{b}(t)+z \left(2  \tilde{b}'(t)
\tilde{b}(t)+4  \tilde{c}'(t)c(t) \right)}{\sqrt{4\tilde{c}(t)^{2}
-(z-\tilde{b}(t))^{2}}}dz \\
- \int_{-2\tilde{c}(t)+\tilde{b}(t)}^{2\tilde{c}(t)+\tilde{b}(t)}(z+\tilde{b}(t))
\frac{z}{t^{2}}\sqrt{4\tilde{c}(t)^{2}-(z-\tilde{b}(t))^{2}}dz. 
\end{multline*}
Changing the variable $z=\tilde{c}(t)x+\tilde{b}(t)$ and using the 
system \eqref{eq:bc4}, we obtain 
\begin{multline*}
\frac{d}{dt} \int_{-2\tilde{c}(t)+\tilde{b}(t)}^{2\tilde{c}(t)+\tilde{b}(t)}(z
+\tilde{b}(t))\tilde{\calV}'(t,z)\sqrt{4\tilde{c}(t)^{2}-(z
-\tilde{b}(t))^{2}}dz= \\ 4\pi(\tilde{b}(t)\tilde{b}'(t)
+2\tilde{c}(t)\tilde{c}'(t))-2\pi \tilde{c}(t)^{2}(2\tilde{b}(t)^{2}
+\tilde{c}(t)^{2})/t^{2}. 
\end{multline*}
Therefore we arrive at the equation 
\begin{eqnarray*}
(t^{2} \calI'(t))'&=&\frac{d}{dt}\left( -\tilde{c}(t)^{2}-\frac{
\tilde{b}(t)^{2}}{2}\right)+(\tilde{b}(t)\tilde{b}'(t)+2\tilde{c}(t)
\tilde{c}'(t)) -\frac{\tilde{c}(t)^{2}(2\tilde{b}(t)^{2}
+\tilde{c}(t)^{2})}{2t^{2}} \\ &=& -\frac{\tilde{c}(t)^{2}(2\tilde{b}(t)^{2}
+\tilde{c}(t)^{2})}{2t^{2}}.
\end{eqnarray*}
Since $\calF_{0,e}(t)=\frac{3}{4}-\calI(t)$, it implies
\[
(t^{2} \calF_{0,e}'(t))'=\frac{2\calR_{e}(t)\calS_{e}^{2}(t)+\calR_{e}^{2}(t)}{2}.
\]
This is exactly the statement from \eqref{eq.dFe}.  To prove also the 
statement from \eqref{eq.Fe}, namely that
\[
\calF_{0,e}(t)=\frac{1}{t}\int_{0}^{t}\frac{(t-s)(2\calR_{e}(s)\calS_{e}^{2}(s)
+\calR_{e}^{2}(s)-1)}{2s}ds, 
\]
denote the right hand side by $\mathcal{G}(t)$ and notice that both sides 
satisfy the same differential equation, namely 
\[
(t^{2}\calG'(t))'=\frac{2\calR_{e}(t)\calS_{e}^{2}(t)+\calR_{e}^{2}(t)}{2}.
\]
In addition, a direct check reveals that 
\[
\calF_{0,e}(0)=\mathcal{G}(0)=0, \calF_{0,e}'(0)=\mathcal{G}'(0)=a^{2}_{1}/2+a_{2}/2
\]
which actually follows from the fact that $\mathcal{R}(t)=1+a_{1}t+O(t^{2})$ 
and $\mathcal{S}(t)=\sqrt{t}a_{1}+O(t)$ (see for example the formulae in 
Appendix \ref{sec.fewterms}).   \qedhere

\subsection{Proof of Theorem \ref{thm.Ff}}
\lbl{sub.thmPf}

In this Section we will prove Theorem \ref{thm.Ff} and the last part of 
Remark \ref{rem.Vef}. It will be convenient
to use the 1-cut potentials  $\ti \calV_{f}$ and $\calV_{f}$ given by
\begin{eqnarray*}
\tilde{\calV}_{f}(x) &=& \frac{x^{2}}{2}-\sum_{n\ge3}\frac{a_{n}x^{n}}{n}
\\
\calV_{F}(x) &=&
\frac{\tilde{\calV}_{f}(\sqrt{t}x)}{t}
=\frac{x^{2}}{2}-\sum_{k\ge3}\frac{t^{k/2-1}a_{k}x^{k}}{k}.
\end{eqnarray*}

As we did in the previous section, for the sake of simplicity we will drop the dependence on $f$ 
from the notation $\calV_{f}$ and $\tilde{\calV}_{f}$.

Define $c(t),b(t)$ and $\tilde{c}(t),\tilde{b}(t)$ the power series 
solutions to \eqref{eq:bc2} corresponding to $\calV$ and $\tilde{\calV}/t$.  
Then, one can easily check that  
\begin{equation}\lbl{e:2000}
\tilde{c}(t)=\sqrt{t}c(t),\quad \tilde{b}(t)=\sqrt{t}b(t).  
\end{equation} 
 
The corresponding system of equations for $\tilde{c}(t)$ and $\tilde{b}(t)$ is 
\begin{equation}\lbl{eq:eq3}
\begin{cases}
\int_{-2}^{2}\tilde{c}(t)x\tilde{\calV}'(\tilde{c}(t)x+\tilde{b}(t))
\frac{dx}{\pi\sqrt{4-x^{2}}}=2t  & \\ 
\int_{-2}^{2}\tilde{\calV}'(\tilde{c}(t)x+\tilde{b}(t))\frac{dx}{
\pi\sqrt{4-x^{2}}}=0. 
\end{cases}
\end{equation}
Now set  $\calG_{0}(t)=-t^{2}\calI_{0,\tilde{\calV}/t}$.  Thus
\begin{eqnarray*}
\calG_{0}(t)&=& t^{2}\log \tilde{c}(t) -t\int_{-2}^{2}\frac{\tilde{\calV}(
\tilde{c}(t)x+\tilde{b}(t))dx}{\pi\sqrt{4-x^{2}}} + \int_{0}^{\tilde{c}(t)}s
\left[\left(\int_{-2}^{2}
\frac{x\tilde{\calV}'(sx+\tilde{b}(t))dx}{2\pi \sqrt{4-x^{2}}} \right)^{2}
\right. \\ & & \left. +\left(  
\int_{-2}^{2}\frac{\tilde{\calV}'(sx+\tilde{b}(t))dx}{\pi \sqrt{4-x^{2}}} 
\right)^{2} \right]ds.
\end{eqnarray*}
Differentiating  this with respect to $t$ and keeping in mind the system 
\eqref{eq:eq3}, we get 
\begin{eqnarray*}
\lbl{eq:eq4}
\calG_{0}'(t)&=&2t\log \tilde{c}(t)-\int_{-2}^{2}\frac{\tilde{\calV}(
\tilde{c}(t)x+\tilde{b}(t))dx}{\pi\sqrt{4-x^{2}}}\\ \notag
& & +2b'(t)\int_{0}^{\tilde{c}(t)}
\int_{-2}^{2}\int_{-2}^{2}\frac{s(xy+4)\tilde{\calV}'(sx+\tilde{b}(t))
\tilde{\calV}''(sy+\tilde{b}(t))}{4\pi^{2}\sqrt{(4-x^{2})(4-y^{2})}}dx\,dy\,ds.
\end{eqnarray*}

Now, taking  $\mathcal{U}(x)=\tilde{\calV}'(\tilde{c}(t)x+\tilde{b}(t))$ 
in Lemma~\ref{l:2} and a simple change of variables proves that the last 
term of \eqref{eq:eq4} becomes $0$.    Thus, we can continue \eqref{eq:eq4} 
with 
\[
\calG_{0}'(t)=2t\log \tilde{c}(t)-\int_{-2}^{2}\frac{\tilde{\calV}(
\tilde{c}(t)x+\tilde{b}(t))dx}{\pi\sqrt{4-x^{2}}}. 
\]
Differentiating this with respect to $t$, and using again the equations 
from \eqref{eq:eq3}, we obtain
\[
\calG_{0}''(t)=2\log \tilde{c}(t).  
\]
In other words, integrating this twice and keeping in mind that 
$\calG_{0}(0)=\calG_{0}'(0)=0$, we get
\[
\calG_{0}(t)=2\int_{0}^{t}(t-u)\log \tilde{c}(u)du.
\]
Now, one has to notice that an easy calculation yields, 
\[
t^{2} \calF_{0}(t)=\frac{3t^{2}}{4}-\frac{t^{2}}{2}\log t+\calG_{0}(t)
=2\int_{0}^{t}(t-u)\log c(u)du=\int_{0}^{t}\log \calR_{f}(u)du,
\]
from which Theorem \ref{thm.Ff} follows. 
\qedhere

\subsection{Proof of Proposition \ref{prop.alg}}
\lbl{sub.prop.alg}

Part (a) and (b) of Proposition \ref{prop.alg} follows from 
Theorems \ref{thm.Fe} and \ref{thm.Fe}.

Part (c) and (d) follow from Proposition \ref{prop.asan}.
\qedhere

%%%%%%%%%%%%%%%%%%%%%%%%%%%%%%%%%%%%%%%%%%%%%%%%%%%%%%%%%%%%%%%%%%%%%%%%%%%%%
%%%%%%%%%%%%%%%%%%%%%%%%%%%%%%%%%%%%%%%%%%%%%%%%%%%%%%%%%%%%%%%%%%%%%%%%%%%%%

\section{The planar limit for extreme edge potentials}
\lbl{s:14}  

In this section we will compute the planar limit for five extreme
formal potentials.

\subsection{Exact Formulae}
\lbl{sub.algVeev}

Consider the extreme formal potentials $V_e(x), V_e^{\ev}(x)
 \in \BQ[[[t^{1/2}]][[x]]$ given by

\begin{eqnarray}
\lbl{eq.Veev}
\calV_e^{\ev}(x) &=& \frac{x^{2}}{2}+\frac{1}{2}\log(1-t x^{2})
=\frac{x^{2}}{2}-\sum_{n=1}^{\infty}\frac{t^{n}x^{2n}}{2n}
\\
\lbl{eq.Ve}
\calV_e(x) &=& \frac{x^{2}}{2}+\log(1-\sqrt{t}x)=
\frac{x^{2}}{2}-\sum_{n\ge1}\frac{t^{n/2}x^{n}}{n}
\end{eqnarray}
These potentials correspond to counting planar diagrams with even 
respectively arbitrary valency of the vertices and a fixed number of edges.  
Their corresponding invariants $b=\calS_{e}=b(t)=\calS_{e}(t)$ is an 
element of $\BQ[[t^{1/2}]]$, while $c=c(t)$ while $\calS_{e}=\calS_{e}(t)$, 
$\calR_{e}=\calR_{e}(t)$ and $\calF_{0,e}=\calF_{0,e}(t)$  are elements of 
$\BQ[[t]]$. Our next proposition summarizes the 
algebraic properties of these elements.

\begin{remark}
\lbl{rem.simple1}
For simplicity of writing, in this section we will omit the subscript $e$ 
in writing $\calR_{e}$, $\calS_{e}$, $\calF_{0,e}$.
\end{remark}

\begin{proposition}
\lbl{p:14}
\begin{enumerate}
\item  
For the potential $\calV_e^{\ev}$, we have 
\begin{equation}
\begin{split}
 \calS(t)&= 0 \\
\calR(t)&= \frac{1+4t-\sqrt{1-8t}}{8t} \\
\lbl{eq:pelog}
\calF_{0}(t)&=
\frac{1-24t+72t^{2}-(1-20t)\sqrt{1-8 t}}{128t^{2}} 
-\frac{3}{8}\log\frac{1-4t+\sqrt{1-8t}}{2} \\
&= \frac{t}{2}+\frac{3 t^2}{4}+2 t^3+7 t^4+\frac{144 t^{5}}{5}
+132 t^{6}+\frac{4576 t^{7}}{7}+3432 t^{8}
%+\frac{56576 t^{9}}{3}
%+\frac{537472 t^{10}}{5}
+O(t^{10})
\end{split}
\end{equation}
\item  
For the potential $\calV_e$, we have 
\begin{equation}
\begin{split}
\calS(t)&=
\frac{1-\sqrt{1-12 t}}{6 \sqrt{t}} \\
\calR(t) &= \frac{1+12 t-\sqrt{1-12 t}}{18 t}  \\
\lbl{eq:plog}
\calF_{0}(t)&=\frac{1-36 t+162t^{2}-(1-30t) \sqrt{1-12 t}}{216 t^{2}} 
-\frac{1}{2}\log\frac{1-6t+\sqrt{1-12t}}{2} \\
&=
t+\frac{9 t^2}{4}+9 t^3+\frac{189 t^4}{4}+\frac{1458 t^{5}}{5}
+\frac{8019 t^{6}}{4}+\frac{104247 t^{7}}{7}
%+\frac{938223 t^{8}}{8}%+966654 t^{9} 
+O(t^{9})
\end{split}
\end{equation}
\end{enumerate}
\end{proposition}

\begin{proof} 
Solving the nonlinear system of Equations \eqref{eq.RS} for our
formal potentials $\calV_{e}(x)$ and $\calV^{\ev}_{e}(x)$ seems at first an impossible
task. Instead, we will use the analytic ideas from Section \ref{s:5}
to translate this system into a more tractable one.
 
In Section \ref{s:13}
it was shown that Equations \eqref{eq.RS} for $(b,c^2)=(\calS,\calR)$ are
exactly Equations \eqref{eq:bc} for $(b,c)$ in case of admissible
analytic potentials. The proof also works for formal potentials, too,
such as our potentials $\calV_e(x)$ and $\calV^{\ev}(x)$. Proposition 
\ref{prop.critp}  implies that Equations 
\eqref{eq:bc} are the critical point equations for the function 
$\calH(b,c)$ from Equation \eqref{e:31}. The last function can be computed
explicitly for the two formal potentials $\calV_e(x)$ and $\calV^{\ev}(x)$.

To prove part (1) of the proposition,  $\calV_e^{\ev}$ is even so $b(t)=0$. 
Computing the function $\calH$ gives 
\begin{align*}
\calH(c)&=\log c- \frac{c^{2}}{2}-\frac{1}{4}\int_{-2}^{2}
\frac{\log(1-t c^{2}x^{2})dx}{\pi\sqrt{4-x^{2}}}\\ &
=\log c-\frac{c^{2}}{2}-\frac{1}{2}\log \frac{1+\sqrt{1-4tc^{2}}}{2}
\end{align*}
and thus 
\[
\calH'(c)=\frac{1}{2c}\left(1-2c^{2} 
+\frac{1}{\sqrt{1-4tc^{2}}}\right). 
\]
The solution $c$ to $\calH'(c)=0$ such that $c(0)=1$ satisfies a 
quartic equation
%%% see Mathematica file: AlgebraicSolutions.nb 
$$
4 c^4 t + c^2 (-1 - 4 t) + t + 1=0
$$ 
and it is given by
\[
c(t)=\sqrt{\frac{1+4t-\sqrt{1-8t}}{8t}}
\]
Given $b(t)$ and $c(t)$ together with \eqref{eq.Fe}, we get 
\eqref{eq:pelog}. 

For part (2) of the proposition, observe that
By Equation \eqref{e:43} we have:
\begin{eqnarray*}
\calH(b,c)&=& \log c-\frac{c^{2}}{2}-\frac{b^{2}}{4}
-\frac{1}{2}\int_{-2}^{2}\frac{\log(1-\sqrt{t}cx-\sqrt{t}b)dx}{
\pi\sqrt{4-x^{2}}} \\ 
&=& \log(c)-\frac{c^{2}}{2}-\frac{b^{2}}{4}
-\frac{1}{2}\log\frac{1-tb+\sqrt{(1-\sqrt{t}b)^{2}-4tc^{2}}}{2}.
\end{eqnarray*}
The critical point $(b,c)$ satisfies the system
\begin{equation}\lbl{eq:201}
\begin{cases}
1-c^{2}+\frac{2tc^{2}}{(1-\sqrt{t}b+\sqrt{(1-\sqrt{t}b)^{2}-4tc^{2}})
\sqrt{(1-\sqrt{t}b)^{2}-4tc^{2}}}=0\\
-b+\frac{\sqrt{t}}{\sqrt{(1-\sqrt{t}b)^{2}-4tc^{2}}}=0. 
\end{cases}
\end{equation}
Solve for $\sqrt{(1-\sqrt{t}b)^{2}-4tc^{2}}=\sqrt{t}/b$ and plug it in the 
first equation which becomes 
\[
1-c^{2}+\frac{2 c^{2} t}{(\sqrt{t}/b)(1-b \sqrt{t}+\sqrt{t}/b)}=0.
\]
In turn, this implies
\[
c^{2}=\frac{b+\sqrt{t}-b^2 \sqrt{t}}{b+\sqrt{t}-3 b^2 \sqrt{t}}.
\]
We need to pick the solution $c$ which for $t=0$ is $1$ and thus 
\[
c=\frac{\sqrt{b+\sqrt{t}-b^2 \sqrt{t}}}{\sqrt{b+\sqrt{t}-3 b^2 \sqrt{t}}}.
\]
We go back to the second equation of \eqref{eq:201} and solve for $c$ 
as a function of $b$ to get 
\[
c^{2}=\frac{(1-b \sqrt{t})^{2}b^{2}-t}{4tb^{2}}.
\]
 Equating now the two expressions of $c^{2}$ in terms of $b$ shows that 
\[ 
\frac{b+\sqrt{t}-b^2 \sqrt{t}}{b+\sqrt{t}-3 b^2 \sqrt{t}}=\frac{(1-b 
\sqrt{t})^{2}b^{2}-t}{4tb^{2}}. 
\]
This implies that $b$ satisfies 
\[
(-b - \sqrt{t}+ b^2 \sqrt{t})^2 (-b + \sqrt{t} + 3 b^2 \sqrt{t})=0.
\]
There are four solutions to this equation, 
\[
\frac{1-\sqrt{1-12 t}}{6 \sqrt{t}},\quad  \frac{1+\sqrt{1-12 t}}{6 
\sqrt{t}},\quad \frac{1-\sqrt{1+4 t}}{2 \sqrt{t}},\quad  \frac{1+
\sqrt{1+4 t}}{2 \sqrt{t}}.
\]
Since $b(0)=0$, this eliminates the second and the fourth solutions.  
To decide which one is the right one, we notice that $c(0)=1$ and this 
implies 
\[
b=\frac{1-\sqrt{1-12 t}}{6 \sqrt{t}} \text{ and } 
c=\sqrt{\frac{1+12 t-\sqrt{1-12 t}}{18t}}.
\]
Now using \eqref{eq.Fe} one concludes \eqref{eq:plog}.
\qedhere
\end{proof}

\subsection{A review of holonomic functions and their asymptotics}
\lbl{sub.reviewhol}

In this section we briefly review soem standard facts about holonomic
functions and their asymptotics from \cite{PWZ}. Recall that a formal
power series
\begin{equation}
\lbl{eq.fx}
f(x)=\sum_{n=0}^\infty a_n x^n
\end{equation}
is {\em holonomic} if it satisfies a linear differential equation 
$$
\sum_{j=0}^d c_j(x) f^{(j)}(x)=0
$$
where $c_j(x) \in \BQ[x]$ for $j=0,\dots,d$ with $c_d(x) \neq 0$. 
A sequence $(a_n)$ is {\em holonomic} if it satisfies a linear recursion
$$
\sum_{j=0}^r \ga_j(n) a_{n+j}=0
$$
for all $n \in \BN$ where $\ga_j(n) \in \BQ[n]$ with $\ga_r(n) \neq 0$. It is
easy to see that a sequence $(a_n)$ is holonomic if and only if the 
generating series \eqref{eq.fx} is holonomic. Of importance to us are
{\em algebraic functions} $y=f(x)$ i.e., solutions to polynomial equations
\begin{equation}
\lbl{eq.pxy}
\sum_{j=0}^d c_j(x) y^j=0
\end{equation}
where $c_j(x) \in \BQ[x]$ for $j=0,\dots,d$ and $c_d(x) \neq 0$.
Algebraic functions regular at $x=0$ are always holonomic; see for 
example \cite{Chu}. Moreover, algebraic functions regular at $x=0$ 
have holomorphic extensions to a finite branched cover of the complex
plane branched along a finite set of algebraic points, given by the roots
of the discriminant of the polynomial \eqref{eq.pxy} with respect to $y$.
Locally, at a point $x=x_0 \in \overline\BQ$ of the field of algebraic
numbers, an algebraic function $y(x)$ has a convergent power series expansion
of the form
$$
y(x)=\sum_{n=0}^\infty c_{n/d} (x-x_0)^{n/d}
$$
for some natural number $d$ and for algebraic numbers $c_{n/d}$. This
is the content of {\em Puiseux's theorem} \cite{Basu,Walker}. If $y(x)$
is regular at $x=0$ with Taylor series
$$
y(x)=\sum_{n=0}^\infty a_n x^n
$$
the asymptotics of the sequence $(a_n)$ can be computed explicitly 
by the singularities of $y(x)$ which are nearest to $x=0$. The computation
also includes the Stokes constants. A computer
implementation of the rigorous computation is available from \cite{Ka}.
In fact, $(a_n)$ is a sequence of {\em Nilsson type} discussed in detail 
in \cite{Ga2}.

\subsection{Holonomicity and asymptotics}
\lbl{sub.holVeev}

In this section we illustrate Proposition \ref{prop.alg} with the concrete 
examples of the extreme potentials and study the coefficients of the Taylor 
series $(f_n)$ of the planar limit written as
$$
\calF_0(t)=\sum_{n=1}^\infty f_n t^n.
$$

%%% See Mathematica file:
%%% AsymptoticsOfTheCoefficientsOfPlanarLimits.stavros.nb

\begin{proposition}
\lbl{c:22}
\rm{(1)}
For the potential $\calV_e^{\ev}$,  
$\calR$,  $\calG_{0}=\calG_{0}(t)=\calF_{0}'(t)$ and $\calF_{0}$ satisfy  
\begin{equation}\lbl{eq.alg.ev}
\begin{split}
4 \calR^{2} t - \calR (1 +4 t) + t + 1 &=  0 \\
64t^{3}\calG_{0}^{2}+(48t^{2}-24t+2)\calG_{0}+9t-1&=  0\\ 
3 +6 ( 4 t-1) \calF_{0}' +2t (8 t-1) \calF_{0}'' &=0,\quad \calF_{0}(0)=0,
\calF_{0}'(0)=1/2 \\ 
(n+3)(n+1)f_{n+1}-4 n (1 + 2 n)f_n&=0, \quad n\ge1,\quad f_1=1/2.  
\end{split}
\end{equation}
In addition, 
\begin{equation}
\lbl{eee:1}
\calF_{0}(t)=\sum_{n \ge1}\frac{3(2n-1)!2^{n-1}}{n!(n+2)!}t^{n},
\end{equation}
and for large $n$, the asymptotics of $f_{n}$ is 
\begin{equation}\lbl{eee:11}
f_{n}= \frac{3}{4\sqrt{\pi}}\frac{8^{n}}{n^{7/2}}\left(1-\frac{25}{8 n}
+\frac{945}{128 n^2}-\frac{16275}{1024 n^3}+O\left(\frac{1}{n^{4}} \right) 
\right).
\end{equation}
\rm{(2)}
For the potential $\calV_e$, we have that $\calS$, $\calR$, $\calG_{0}
=\calF_{0}'$ and $\calF_{0}$ satisfy
\begin{equation}
\lbl{eq.alg.all}
\begin{split}
3\sqrt{t}\calS^{2}-\calS+\sqrt{t}&=0\\
9t\calR^{2}-(12t+1)\calR+1&=0\\ 
108 t^3 \calG_{0}^2+(108 t^{2} - 36 t + 2) \calG_{0} + 27 t-2 & =0\\
3+3 (6 t-1) \calF_{0}'+t\left(12 t-1\right)\calF_{0}''&=0,\calF_{0}(0)=1,
\calF_{0}'(0)=1 \\
(n+3)(n+1)f_{n+1}-6 n (1 + 2 n)f_n&=0, \qquad f_1=1.
\end{split}
\end{equation}
In addition, 
\begin{equation}
\lbl{eee:2}
\calF_{0}(t)=\sum_{n \ge1}\frac{2(2n-1)!3^{n}}{n!(n+2)!}t^{n},
\end{equation}
and for large $n$,
\begin{equation}
\lbl{eee:22}
f_{n}=\frac{2}{\sqrt{\pi}} \frac{12^{n}}{n^{7/2}}\left(1-\frac{25}{16 n}
+\frac{945}{256 n^2}-\frac{16275}{2048 n^3}+O\left(\frac{1}{n^{4}} \right) 
\right).
\end{equation}
\end{proposition}

\begin{proof} 
It is straightforward to see that \eqref{eq.alg.ev} follows 
from \eqref{eq:pelog} while \eqref{eq.alg.all} from \eqref{eq:plog}.

For (1), a direct check proves that $\calF_{0}(t)$ solves the third equation 
of  \eqref{eq.alg.ev}.  This immediately implies the recurrence on $f_{n}$ 
and then the closed formula in \eqref{eee:1} which in turn combined with 
Stirling's formula leads to \eqref{eee:2}.  
 
Another way of checking the closed formula \eqref{eee:1}  is the following.  
Observe by a direct calculation that 
\[ (
8t^{2}-t)\calF_{0}'''(t)+(28t-4)\calF_{0}''(t)+12\calF_{0}'(t)=0, \quad
\calF_{0}(0)=0, \calF_{0}'(0)=1/2 ,\calF_{0}''(0)=3/2.
\]
This is a hypergeometric equation, and its solution is given by
\[
\calF_{0}(t)=\frac{1}{2}t\, _{3}F_{2}(1,1,3/2;2,4;8t)
=\sum_{n \ge1}\frac{3(2n-1)!2^{n-1}}{n!(n+2)!}t^{n},
\]
where $_{3}F_{2}(a_{1},a_{2},a_{3};b_{1},b_{2};x)$ stands for the 
hypergeometric function with parameters $a_{1},a_{2},a_{3}$ and $b_{1},b_{2}$.   
This is exactly \eqref{eee:1}. 
Using {\em Stirling's formula} (see \cite{Ol}), 
one can easily deduce \eqref{eee:1}.  

For (2), use the same proof as for (1).\qedhere
\end{proof}

\subsection{Three more flavors of the extreme edge potentials}
\lbl{sub.3more}

In this section we will investigate the following three flavors of the 
extreme edge potentials \eqref{eq.Ve} and \eqref{eq.Veev}, given by

\begin{eqnarray}
\calV_1(x) &=& \frac{(1+t)x^{2}}{2}+\frac{1}{2}\log(1-tx^{2})
=\frac{x^{2}}{2}-\sum_{n\ge2}\frac{t^{n}x^{2n}}{2n}
\\
\calV_2(x) &=& \frac{x^{2}}{2}+\sqrt{t}x+\log(1-\sqrt{t}x)=
\frac{x^{2}}{2}-\sum_{n\ge2}\frac{t^{n/2}x^{n}}{n} 
\\
\calV_3(x) &=& \frac{(1+t)x^{2}}{2}+\sqrt{t} x+\log(1-\sqrt{t}x)=
\frac{x^{2}}{2}-\sum_{n\ge3}\frac{t^{n/2}x^{n}}{n}.  
\end{eqnarray}
These correspond to the counting of planar diagrams with a fixed number of 
edges and  vertices of even valency greater or equal to 4,  or arbitrary 
valency greater or equal to 2, respectively arbitrary valency greater or 
equal to 3.

\begin{proposition}
\lbl{p:14b}
\rm{(1)} For the potential $\calV_1$, we have
\begin{eqnarray}
 \calS(t) &=& 0 \\
\calR(t) &=& \frac{1+5t-\sqrt{(1+t)(1-7t)}}{8t(1+t)}
\\
\lbl{eq:pelog21}
\calF_{0}(t) &=&\frac{1-22t+49t^{2}-(1-19t)\sqrt{(1+t)(1-7t)}}{128t^{2}} 
- \frac{1}{8}\log(1+t) \\ & & 
-\frac{3}{8}\log\frac{1-3t+\sqrt{(1+t)(1-7t)}}{2} \\
&=& \frac{t^{2}}{2}+\frac{5t^{3}}{6}+\frac{23t^{4}}{8}+\frac{51t^{5}}{5}
+\frac{124 t^{6}}{3}+\frac{2515t^{7}}{14}+\frac{13245t^{8}}{16}
%+\frac{35902t^{9}}{9}+
%\frac{199319t^{10}}{10}
+O(t^{9}).
\end{eqnarray}
\rm{(2)} For the potential $\calV_2$, we have
\begin{eqnarray}
 \calS(t) &=& \frac{1-5t-\sqrt{1-10t+t^{2}}}{6\sqrt{t}}  \\
\calR(t) &=& \frac{1+14t+t^{2}-(1+t)\sqrt{1-10t+t^{2}}}{18t} \\
\lbl{eq:plog22}
\calF_{0}(t) &=& 
\frac{1-32t+96t^{2}+76t^{3}+t^{4}-(1-27t-27t^{2}+t^{3})
\sqrt{1-10t+t^{2}}}{216t^{2}}  \\ \notag & &
- \log \frac{1+t+\sqrt{1-10t+t^{2}}}{2} \\
&=& \frac{t}{2}+\frac{3 t^2}{4}+\frac{8 t^3}{3}+12 t^4+\frac{312 t^{5}}{5}
+\frac{1076 t^{6}}{3}+\frac{15528 t^{7}}{7}+14508 t^{8}
%+\frac{892376 t^{9}}{9}+\frac{3510828 t^{10}}{10}
+O(t^{9}).
\end{eqnarray}
\rm{(3)}
For the potential $\calV_3$, we have
\begin{eqnarray}
 \calS(t) &=&
\frac{1-4t-\sqrt{1-8t-8t^{2}}}{6(1+t)\sqrt{t}}
\\ \calR(t) &=&
\frac{1+16t+16t^{2}-(1+16t)\sqrt{1-8t-8t^{2}}}{18t(1+t)^{2}}
\\
\lbl{eq:plog23}
\calF_{0}(t)&=&
\frac{1 - 28 t + 6 t^2 + 176 t^3 + 142 t^4-(1 - 24 t - 78 t^2 - 52 t^3)
\sqrt{1-8t-8t^{2}}}{216t^{2}(1+t)^{2}}   \\ 
& & + \frac{1}{2}\log(1+t)-\log\frac{1+2t+\sqrt{1-8t-8t^{2}}}{2} 
\\ & =&
\frac{t^2}{2}+\frac{3 t^3}{2}+\frac{47 t^4}{8}+\frac{139 t^{5}}{5}
+\frac{430 t^{6}}{3}+\frac{11175 t^{7}}{14}+\frac{75149 t^{8}}{16}
%+\frac{86642 t^{9}}{3}+\frac{367979 t^{10}}{2}
+O(t^{9}).
\end{eqnarray}
\end{proposition}

\begin{proof}  
We follow the same approach as in Proposition \ref{p:14}. 

For (1), the function $\calH$ becomes
\begin{align*}
\calH(c)&=\log c- \frac{(1+t)c^{2}}{2}-\frac{1}{4}\int_{-2}^{2}
\frac{\log(1-c^{2}tx^{2})dx}{\pi\sqrt{4-x^{2}}}\\ &
=\log c-\frac{(1+t)c^{2}}{2}-\frac{1}{2}\log \frac{1+\sqrt{1-4c^{2}t}}{2}
\end{align*}
and thus 
\[
\calH'(c)=\frac{1}{2c}\left(1-2(1+t)c^{2} 
+\frac{1}{\sqrt{1-4c^{2}t}}\right). 
\]
The solution to $\calH'(c)=0$ with $c(0)=1$ is 
\[
c(t)=\frac{\sqrt{1+5t-\sqrt{(1+t)(1-7t)}}}{2\sqrt{2t(1+t)}}.
\]
From this, using \eqref{eq.Fe}, gives \eqref{eq:pelog21}.

For (2), we have
\begin{align*}
\calH(b,c)&=\log c-\frac{c^{2}}{2}-\frac{b^{2}}{4}-\frac{b \sqrt{t}}{2}
-\frac{1}{2}\int_{-2}^{2}\frac{\log(1-\sqrt{t}cx-\sqrt{t}b)dx}{
\pi\sqrt{4-x^{2}}}\\ 
&=\log(c)-\frac{c^{2}}{2}-\frac{b^{2}}{4}-\frac{b\sqrt{t}}{2}
-\frac{1}{2}\log\frac{1-\sqrt{t}b+\sqrt{(1-\sqrt{t}b)^{2}-4tc^{2}}}{2}.
\end{align*}
The critical point $(b,c)$ satisfies the system
\[
\begin{cases}
1-c^{2}+\frac{2tc^{2}}{(1-\sqrt{t}b+\sqrt{(1-\sqrt{t}b)^{2}-4tc^{2}})
\sqrt{(1-\sqrt{t}b)^{2}-4tc^{2}}}=0\\
-b-\sqrt{t}+\frac{\sqrt{t}}{\sqrt{(1-\sqrt{t}b)^{2}-4tc^{2}}}=0. 
\end{cases}
\]
The solution to this system such that $c(0)=1$ is given by 
\[
c(t)=\frac{\sqrt{1+14t+t^{2}-(1+t)\sqrt{1-10t+t^{2}}}}{3\sqrt{2t}}\text{ and } 
b(t)=\frac{1-5t-\sqrt{1-10t+t^{2}}}{6\sqrt{t}}
\]
Then \eqref{eq.Fe} together with some simplifications give \eqref{eq:plog22}.

For (3) we have
\begin{align*}
\calH(b,c)&=\log c-\frac{(1+t)c^{2}}{2}-\frac{(1+t)b^{2}}{4}
-\frac{b \sqrt{t}}{2}-\frac{1}{2}\int_{-2}^{2}\frac{\log(1-\sqrt{t}cx
-\sqrt{t}b)dx}{\pi\sqrt{4-x^{2}}} 
\\ &=\log(c)-\frac{(1+t)c^{2}}{2}-\frac{(1+t)b^{2}}{4}-\frac{b\sqrt{t}}{2}
-\frac{1}{2}\log\frac{1-\sqrt{t}b+\sqrt{(1-\sqrt{t}b)^{2}-4tc^{2}}}{2}.
\end{align*}
The critical point $(b,c)$ satisfies the system
\[
\begin{cases}
1-c^{2}+\frac{2tc^{2}}{(1-\sqrt{t}b+\sqrt{(1-\sqrt{t}b)^{2}-4tc^{2}})\sqrt{(1
-\sqrt{t}b)^{2}-4tc^{2}}}=0\\
-(1+t)b-\sqrt{t}+\frac{\sqrt{t}}{\sqrt{(1-\sqrt{t}b)^{2}-4tc^{2}}}=0,
\end{cases}
\]
with the solution satisfying $c(0)=1$, being 
\[
c(t)=\frac{\sqrt{1+16t+16t^{2}-(1+16t)\sqrt{1-8t-8t^{2}}}}{3(1+t)\sqrt{2t}}
\text{ and } b(t)=\frac{1-4t-\sqrt{1-8t-8t^{2}}}{6(1+t)\sqrt{t}}.
\]
Finally, using \eqref{eq.Fe} one obtains \eqref{eq:plog23}.\qedhere
\end{proof}

Next we present algebraic and differential equations satisfied by 
$\calG_{0}=\calF_0'(t)$ and the recursion relation
for the Taylor coefficients $(f_n)$ of $\calF_0(t)$ and their exact
asymptotic expansions.

\begin{proposition}
\lbl{c:223}  
\rm{(1)}
For the potential $\calV_1$, $\calG_{0}$  %$\calF_0(t)$ and $(f_{n})$ 
satisfies
$$
 -2 t+13 t^2+16 t^3+2 \left(1-9 t+3 t^2+45 t^3+32 t^4\right) 
\calG_0+64 t^3 (1+t)^2 \calG_0^2 =0\,,
$$
and $\calF_0(t)$ satisfies
\begin{equation}
\lbl{eq.ODE3} 
t(8 +7 t)-2 \left(3-4 t-21 t^2-14 t^3\right) 
\calF_{0}'-2 t\left(1-5 t-13 t^2-7 t^3\right) \calF_{0}'' =0\,,
%\quad\calF_{0}(0)=0,\,\calF_{0}'(0)=0 \,,
\end{equation}
with $\calF_{0}(0)=0,\,\calF_{0}'(0)=0$ and $(f_{n})$ satisfies  
\begin{multline*}
49 n^2 (1+n) f_{n}+7 (1+n)^2 (32+21 n) f_{n+1}+(2+n) 
\left(544+543 n+139 n^2\right) f_{n+2}\\ 
+(3+n) \left(224+157 n+33 n^2\right) f_{n+3}
-8 (2+n) (4+n) (6+n) f_{n+4}=0,
\end{multline*}
with $f_{1}=0,f_{2}=1/2,f_{3}=5/6$.
For large $n$, the asymptotics of $(f_{n})$ is given by
\begin{equation}
\lbl{eee:3}
f_{n} = \frac{147}{512}\sqrt{\frac{7}{2\pi}} \frac{7^{n}}{n^{7/2}}
\left( 1-\frac{105}{32 n}+\frac{16065}{2048 n^{2}}
-\frac{1109115}{65536 n^{3}}+O\left(\frac{1}{ n^{4}}\right)\right).
\end{equation}
\rm{(2)}
For the potential $\calV_2$, $\calG_{0}$ %$\calF_0(t)$ and $(f_{n})$ 
satisfies
\begin{equation}
\lbl{eq.ODE4}
1 - 13 t + 22 t^2 -  9 t^3 - t^4 + (-2 + 32 t - 108 t^2 + 76 t^3 + 2 t^4) 
\calG_0 - 108 t^3 \calG_0^2=0
\end{equation}
and $\calF_0(t)$ satisfies
$$ 
3-5 t+t^2+t^3-2 \left(3-17 t+5 t^2+t^3\right) \calF_{0}'- 2t 
\left(1-11 t+11 t^2-t^3\right) \calF_{0}'' =0,
$$
with $\calF_{0}(0)=0, \calF_{0}'(0)=1/2$ and $(f_{n})$ satisfies
\begin{multline*}
(-2+n) (-1+n) n f_{n}-2 (-1+n) (1+n) (2+5 n) f_{n+1}-(2+n) 
(44+25 n+5 n^2) f_{n+2}  \\ 
+4 (3+n) \left(116+91 n+17 n^2\right) f_{n+3}  -(4+n) 
(1326+701 n+89 n^2 ) f_{n+4}   \\ 
+2 (5+n) (6+n) (85+19 n) f_{n+5}-3 (5+n) (6+n) (8+n) f_{n+6}=0,
\end{multline*}
with $f_{1}=1/2, f_{2}=3/4, f_{3}=8/3, f_{4}=12, f_{5}=312/5$.
Moreover,  for large $n$, the asymptotics of $(f_{n})$ is given by
\begin{equation}
\lbl{eee:4}
f_{n} = \frac{2}{3\sqrt{\pi}}\sqrt[4]{\frac{2}{3}} 
\frac{(5+2\sqrt{6})^{n}}{n^{7/2}}\left( 1-\frac{45\sqrt{6}}{32 n}
+\frac{8435}{1024 n^{2}}-\frac{238805\sqrt{6}}{32768 n^{3}}
+O\left(\frac{1}{ n^{4}}\right)\right).
\end{equation}
\rm{(3)}
For the potential $\calV_3$, $\calG_{0}$ %, $\calF_0(t)$ and $(f_{n})$ 
satisfies
\begin{multline}
\lbl{eq.ODE5}
-2 t + 11 t^2 + 65 t^3 + 107 t^4 + 81 t^5 + 
 27 t^6 + (2 - 20 t - 22 t^2 + 184 t^3 + 560 t^4 + 700 t^5 + 
    432 t^6 + 108 t^7) \calG_0  \\ 
    + (108 t^3 + 540 t^4 + 1080 t^5 + 1080 t^6 + 
    540 t^7 + 108 t^8) \calG_0^2=0 
\end{multline}
and  $\calF_0(t)$ satisfies
\begin{multline*}
-8 t-35 t^2-44 t^3-16 t^4-6 \left(-1+15 t^2+34 t^3+28 t^4+8 t^5\right) 
\calF_{0}'  \\ 
- 2 \left(-t+5 t^2+29 t^3+47 t^4+32 t^5+8 t^6\right) \calF_{0}''=0,
\end{multline*}
with $\calF_{0}(0)=0,\calF_{0}'(0)=0$ and  $(f_{n})$ satisfies
\begin{multline*}
128 n^2 (2+n) f_{n}+32 (1+n) \left(66+97 n+27 n^2\right) f_{n+1}+8 (2+n) 
\left(1836+1568 n+305 n^2\right) f_{n+2} \\ 
  +4 (3+n) \left(9972+6193 n+929 n^2\right) f_{n+3}+(4+n) 
\left(54276+26661 n+3257 n^2\right) f_{n+4} \\ 
 +(5+n) \left(38220+15479 n+1595 n^2\right) f_{n+5}+3 (6+n) 
\left(3972+1349 n+121 n^2\right) f_{n+6} \\ 
 +5 (7+n) \left(36+5 n+n^2\right) f_{n+7}-8 (6+n) (8+n) (10+n) 
f_{n+8}=0 
\end{multline*}
with $f_{1}=0,f_{2}=1/2,f_{3}=3/2,f_{4}=47/8,f_{5}=138/5,
f_{6}=430/3, f_{7}=11175/14$.
Moreover, for large $n$, the asymptotics of $(f_{n})$ is given by
\begin{eqnarray}
\lbl{eee:5}
f_{n}&=& c \frac{\left(4+2\sqrt{6}\right)^{n}}{n^{7/2}}\left( 1-\frac{5 
\left(62-23 \sqrt{6}\right)}{8 n}+\frac{35 \left(4567-1858 \sqrt{6}
\right)}{64 n^2} \right. \\ \notag & & \left.
-\frac{35 \left(2608410-1064767 \sqrt{6}\right)}{512 n^3}
+O\left(\frac{1}{ n^{4}}\right)\right), 
\end{eqnarray}
with 
\[
c=\frac{32}{3 \sqrt{\left(267+109 \sqrt{6}\right) \pi }}.
\]
\end{proposition}

\begin{remark}
Notice here that the exponential rates change depending on the counting 
problem at hand.  Excluding just a few types of vertices leads to different 
exponential behavior. It is worth pointing out that in
$$
f_n \sim C \frac{{t_0}^{-n}}{n^{3-\ga}}
$$
although the exponential growth rate $t_0$ depends on the details of the model,
the exponent $\ga$ is universal, as was observed in \cite{LeGall} and also
in \cite[Sec.2]{Borot}.

The recurrence relations for the coefficients $(f_{n})$ are not in general of 
the lowest degree.  However we did not attempt to simplify them even further 
because they are easily deduced from the differential equations satisfied by 
$\calF_{0}$.  
\end{remark}

\begin{remark} 
\lbl{rem.stokes}
The linear recursions for $(f_n)$ or the linear differential
equation for $\calF_0$ cannot compute the Stokes constants, i.e.,
the leading terms in the asymptotic expansion \eqref{eee:3}. 

It is the algebricity of $\calF_0'(t)$ which uniquely determines the Stokes 
constants.  In the case at hand the Stokes constants come from the explicit 
expressions of $\calF_{0}$. 
\end{remark}

\begin{proof}
The results from equations \eqref{eq.ODE3}, \eqref{eq.ODE4} and 
\eqref{eq.ODE5} follow straightforwardly from equations \eqref{eq:pelog21}, 
\eqref{eq:plog22} and \eqref{eq:plog23}.

The idea of proving these asymptotics is based on the analysis of 
singularities as explained in \cite{FS}.  This is particularly tractable 
as we have explicit expressions for the planar limits.   

For $\calV_1$,
from \eqref{eq:pelog21} one can see that the singularities of 
$\calF_{0}$ are at $t_{0}=1/7$ and $t_{0}=-1$.  The smallest singularity 
gives the leading terms.  In our case the expansion of $\calF_{0}$ near 
$t_{0}=1/7$ is 
{\small
\begin{align*}
\frac{1}{16} (-7&-12 \log(2)+8 \log(7))-\frac{3 (1-7t)}{64}+\frac{137 
\left(1-7t \right)^2}{1024}-\frac{49}{320} \sqrt{\frac{7}{2}} 
\left(1-7t \right)^{5/2}+\frac{5569 \left(1-7t \right)^3}{12288} \\ &
-\frac{343 \sqrt{\frac{7}{2}} \left(1-7t \right)^{7/2}}{1024}
+\frac{105473 \left(1-7t \right)^4}{131072}-\frac{51793 
\sqrt{\frac{7}{2}} \left(1-7t \right)^{9/2}}{98304}+O((1-7t)^{5}).
\end{align*}
}
Since the main contribution to the asymptotics of the coefficients is 
given by the half powers, combined with 
\begin{equation}\lbl{e:2010}
(1-x)^{\alpha}=\sum_{k=0}^{\infty}(-1)^{k}{\alpha \choose k}x^{k}
\end{equation}
applied for $\alpha=5/2,7/2,9/2$, after a simple asymptotics expansion 
give the result of \eqref{eee:3}.

For $\calV_2$,
the same argument works in this case for the other examples.  
Using \eqref{eq:plog22}, we can see that the singularities of $\calF_{0}$ 
are $t_{0}=5-2\sqrt{6}$ and $5+2\sqrt{6}$.  For the asymptotics of the 
coefficients the leading one is the smallest, namely $t_{0}=5-2\sqrt{6}$.  
The expansion near $t_{0}$ is in this case 
{\small
\begin{align*}
&\left(\frac{23}{12}-\sqrt{6}-\log(3-\sqrt{6})\right)
+\left(-2+\frac{7 \sqrt{\frac{2}{3}}}{3}\right) 
\left(1-\frac{t}{t_{0}}\right)+\frac{1}{36} \left(15-4 \sqrt{6}\right) 
\left(1-\frac{t}{t_{0}}\right)^2 \\ &\qquad
-\frac{16}{45} \left(\frac{2}{3}\right)^{1/4} 
\left(1-\frac{t}{t_{0}}\right)^{5/2}+\frac{1}{27} 
\left(9+2 \sqrt{6}\right) \left(1-\frac{t}{t_{0}}\right)^3
-\frac{1}{63} \sqrt{1008+1270 \sqrt{\frac{2}{3}}} 
\left(1-\frac{t}{t_{0}}\right)^{7/2}\\ &\qquad
+\left(\frac{14}{27}+\frac{1}{\sqrt{6}}\right) 
\left(1-\frac{t}{t_{0}}\right)^4-\frac{\sqrt{1089936+1337137 
\sqrt{\frac{2}{3}}} \left(1-\frac{t}{t_{0}}\right)^{9/2}}{1296}
+O\left(\left(1-\frac{t}{t_{0}}\right)^{5}\right).
\end{align*}
}
Considering the coefficients of the half powers, and noting that 
$1/t_{0}=(5+2\sqrt{6})$, one gets \eqref{eee:4}.

For $\calV_3$, we proceed similarly.  From \eqref{eq:plog23}, 
observe that the singularities  of $\calF_{0}$ are $t_{0}=\frac{1}{4} 
\left(-2+\sqrt{6}\right)$ and $\frac{1}{4} \left(-2-\sqrt{6}\right)$.  
The one with smallest absolute value is $t_{0}=\frac{1}{4} 
\left(-2+\sqrt{6}\right)$.  The expansion near this point is given by 
{\small
\begin{align*}
&\left(\frac{23}{12}-\sqrt{6}+\frac{1}{2} \log\left(\frac{2}{3} 
\left(2+\sqrt{6}\right)\right)\right)+\frac{1}{18} \left(357-146 
\sqrt{6}\right) \left(1-\frac{t}{t_{0}}\right)+\left(\frac{2539}{12}
-259 \sqrt{\frac{2}{3}}\right) \left(1-\frac{t}{t_{0}}\right)^2 \\ 
&\quad -\frac{256 \left(1-\frac{t}{t_{0}}\right)^{5/2}}{45 \sqrt{267+109 
\sqrt{6}}} +\frac{1}{54} \left(122589-50038 \sqrt{6}\right) 
\left(1-\frac{t}{t_{0}}\right)^3-\frac{1472 \left(1-\frac{t}{t_{0}}
\right)^{7/2}}{21 \sqrt{7929+3237 \sqrt{6}}} \\ 
&\quad+\frac{1}{216} \left(5233579-2136536 \sqrt{6}\right) 
\left(1-\frac{t}{t_{0}}\right)^4 -\frac{16520 
\left(1-\frac{t}{t_{0}}\right)^{9/2}}{81 \sqrt{26163+10681 \sqrt{6}}}
+O\left(\left(1-\frac{t}{t_{0}}\right)^{5}\right)
\end{align*}
}
from which, noticing that $1/t_{0}=4+2\sqrt{6}$ one can deduce 
\eqref{eee:5}.  \qedhere
\end{proof}

%%%%%%%%%%%%%%%%%%%%%%%%%%%%%%%%%%%%%%%%%%%%%%%%%%%%%%%%%%%%%%%%%%%%%%%%%%%%%
%%%%%%%%%%%%%%%%%%%%%%%%%%%%%%%%%%%%%%%%%%%%%%%%%%%%%%%%%%%%%%%%%%%%%%%%%%%%%

\section{The planar limit for extreme face potentials}
\lbl{s:15}

%%%%%%%%%%% Amthematica files: Planar-Algebraic-Diff-Recurence.nb
Consider the extreme formal potentials

\begin{eqnarray}
\lbl{eq.Vfev}
\calV_f^{\ev}(x) &=& \frac{x^{2}}{2}-\sum_{n\ge4}\frac{t^{n-1}x^{2n}}{2n} \\
\calV_f(x) &=& \frac{x^{2}}{2}-\sum_{n\ge3}\frac{t^{n/2-1}x^{n}}{n}.
\end{eqnarray}

We consider now the case of extremal potentials and compute the corresponding 
planar limit.  

\begin{remark}
For simplicity, in this section we will drop the subscript $f$ from 
the writings of $\calR_{f}$, $\calF_{0,f}$, $\calV_{f}$.
\end{remark}

\begin{proposition}
\lbl{p:15}
\rm{(1)} For the potential $\calV^{\ev}$, 
the planar limit $\calF_{0}(t)$ has the following 10 terms in the Taylor 
expansion 
\begin{align*}
\calF_{0}(t)=\frac{t}{2}+\frac{47 t^2}{24} &+\frac{49 t^3}{4}
+\frac{11839 t^4}{120}+\frac{9283 t^5}{10}+\frac{3260543 t^6}{336}
+\frac{18387797 t^7}{168}+\frac{941448191 t^8}{720}\\ &
+\frac{490223647 t^9}{30}+\frac{93171535189 t^{10}}{440}+O(t^{11})
\end{align*}
and its radius of convergence is $t_{0}=\frac{4-3\sqrt[3]{2}}{4}$.   
In addition, if $f_{n}$ is the coefficient of $t^{n}$ in the Taylor expansion 
of $F_{0}$, then the asymptotic is
\begin{equation}
\lbl{e:2006}
f_{n}= \frac{1}{3}\sqrt{\frac{2\sqrt[3]{2}-1}{\pi}} 
\frac{\left(\frac{4}{4 - 3  \sqrt[3]{2}}\right)^n}{n^{
 7/2}} \left(1 - \frac{243 - 8  \sqrt[3]{2}}{72 n} + \frac{
   91881 - 2640 \sqrt[3]{2} - 5696  \sqrt[3]{4}}{
   10368 n^2} +O\left(\frac{1}{n^{3}}\right)\right).
\end{equation}
\rm{(2)} For the potential $\calV$,
the planar limit $\calF_{0}$ has the Taylor expansion 
\begin{eqnarray*}
\calF_{0}(t)&=&
\frac{7 t}{6} +\frac{109 t^2}{8}+\frac{15631 t^3}{60} %\\ & &
+\frac{256629 t^4}{40}+\frac{38720767 t^5}{210}+\frac{658811733 t^6}{112}
+O(t^{7}). %+\frac{14569839433 t^7}{72} 
%+\frac{1768959661909 t^8}{240}
%+\frac{277954408632319 t^9}{990}+\frac{4878489849423189 t^{10}}{440}
\end{eqnarray*}
The planar limit has radius of convergence given by $t_{0}$, the only 
positive root of the polynomial equation
$$
-11 - 128 t + 41088 t^2 - 20480 t^3 + 4096 t^4=0.
$$   
The coefficient $f_{n}$ of $t^{n}$ from the expansion of $\calF_{0}$ has 
the asymptotic expansion
\begin{equation}
\lbl{e:2007}
f_{n}\sim c \frac{(1/t_{0})^{n}}{n^{7/2}}\left( 1+\frac{d_{1}}{n}
+\frac{d_{2}}{n^{2}}+O\left(\frac{1}{n^{3}}\right) \right),
\end{equation}
with 
\begin{align*}
c=&\sqrt{ \frac{34133-914556 t_{0}+449856 t_{0}^2-89344 t_{0}^3}{ 176868\pi}}
\\ 
d_{1}=&-\frac{36145645+79913928 t_{0}-39094848 t_{0}^2+7808512 t_{0}^3}{11319552} 
\\
d_{2}=&\frac{7806311269+20984001752 t_{0}-10129539392 t_{0}^{2}
+2006727168 t_{0}^{3}}{1026306048}
\end{align*}
Numerically,
\begin{align*}
t_{0}&=0.0180827901833\dots \\
1/t_{0}&=55.3012001942\dots \\ 
c&=0.1786898225\dots\\ 
d_{1}&=-3.3197404318\dots\\ 
d_{2}&=7.9727292073\dots.
\end{align*} 
\end{proposition}

\begin{proof} 
For part (1), we will find the singularities of $\calR(t)=c^{2}(t)$ and then 
expand it around its singularities nearest to the origin, 
as explained in \cite{Ga2}.

First,  notice that 
\[
\calH(c)=\log(c)-c^{2}-\frac{1}{2t}\log\frac{1+\sqrt{1-4t c^{2}}}{2}
\]
and $\calH'(c)=0$ implies that the equation satisfied by $\calR=c^{2}$ is 
\begin{equation}\lbl{e:2009}
1 - \calR - t^2 + 4 \calR^2 t^2 + 4 \calR t^4 - 16 \calR^2 t^4 + 16 \calR^3 t^4 
= 0.
\end{equation}
Our condition is that $\calR(0)=1$ and this determines the branch near $0$.  
One can actually find the solution explicitly, however that is not very 
useful.   The singularities of $\calR$ are at the points  where the 
discriminant of \eqref{e:2009} vanishes.  That means that the singularities 
are solutions to the polynomial equation 
\[
-16 (-5 t^2 + 96 t^3 - 96 t^4 + 32 t^5)=0.
\]
The latter are given by 
\[ \left\{
0,\frac{1}{2} \left(2-\frac{3}{2^{2/3}}\right),
1+\frac{3 \left(1-i \sqrt{3}\right)}{4 2^{2/3}},1+\frac{3 
\left(1+i \sqrt{3}\right)}{4 2^{2/3}} \right\}. 
\]
The singularity of $\calR$ is thus $t_{0}=\frac{1}{2} 
\left(2-\frac{3}{2^{2/3}}\right)\approx 0.0550592$.   
The series expansion of $\calR$ near $t_{0}$ is
\begin{eqnarray*}
\calR(t)&=& 
\frac{1}{10} \left(7+4\, \sqrt[3]{2}+3\, \sqrt[3]{4}\right)
-\frac{104976}{1574640} \sqrt{15 \left(9+8 \sqrt[3]{2}
+6 \sqrt[3]{4}\right)} \left(1-\frac{t}{t_{0}}\right)^{1/2}
\\ & &
+\frac{17496}{1574640} 
\left(38+36 \sqrt[3]{2}+27 \sqrt[3]{4}\right) \left(1-\frac{t}{t_{0}}\right) 
\\ & &
-\frac{1944}{1574640} \sqrt{15 \left(28569+23328 \sqrt[3]{2}
+17746 \sqrt[3]{4}\right)}  \left(1-\frac{t}{t_{0}}\right)^{3/2}
\\ & &
+\frac{648}{1574640} \left(996+972 \sqrt[3]{2}+749 \sqrt[3]{4}\right) 
\left(1-\frac{t}{t_{0}} \right)^2 \\ & & -\frac{54}{1574640} \sqrt{15 
\left(37321489+30114648 \sqrt[3]{2}+30114648 \sqrt[3]{4}\right)}  
\left( 1-\frac{t}{t_{0}}\right)^{5/2}
\\ & &
+O\left(\left(1-\frac{t}{t_{0}}
\right)^{4}\right).
\end{eqnarray*}
Furthermore, the simplest way to proceed from here is to notice that 
$\calF_{0}(t)/t^{2}$ differentiated three times is $\calR'(t)/\calR(t)$.  
This in particular means that the singularities of $\calF_{0}$ are the 
same as the ones of $R$ and eventually the zeros of $\calR$.  Since the 
zeros of $\calR$ are only $\pm1$, it follows that the singularity of 
$\calF_{0}$ is also $t_{0}$.  The expansion of the third derivative of 
$\calF_{0}(t)/t^{2}$ near $t_{0}$ is thus 
\begin{align*}
&-\frac{1}{3}\sqrt{\frac{1}{5} \left(832+664 \sqrt[3]{2}+176 
\sqrt[3]{4}\right)}/\sqrt{1-\frac{t}{t_{0}}}-\frac{424+308 \sqrt[3]{2}
+256 \sqrt[3]{4}}{135} \\ &+\frac{\sqrt{19984+16062 \sqrt[3]{2}
+12740 \sqrt[3]{4}}}{27}\sqrt{1-\frac{t}{t_{0}}}-\frac{12008+7336 
\sqrt[3]{2}+7472 \sqrt[3]{4} }{3645}\left(1-\frac{t}{t_{0}}\right) 
\\ &-\frac{\sqrt{113890440+92086727 \sqrt[3]{2}+76015706 
\sqrt[3]{4}}}{1458\sqrt{2}}\left(1-\frac{t}{t_{0}}\right)^{3/2}
+O\left(\left(1-\frac{t}{t_{0}} \right)\right).  
\end{align*}
Using \eqref{e:2010}, we can deduce the behavior of the coefficients 
of the third derivative of $\calF_{0}(t)/t^{2}$ and then a simple exercise 
leads to the asymptotics of the coefficients of $\calF_{0}$ from 
\eqref{e:2006}.  

For part (2) we proceed similarly.  This time, 
\[
\calH(b,c)=\log(c)-c^{2}-\frac{b^{2}}{2}-\frac{b}{2\sqrt{t}}
-\frac{1}{2t}\log\left(\frac{1-b\sqrt{t} +\sqrt{(1-b \sqrt{t} )^{2}
-4t c^{2}}}{2}\right).
\]
From the critical system satisfied by $b$ and $c$ eliminate $b$ and then 
consider $\calR=c^{2}$ to arrive at the equation satisfied by $\calR$: 
\begin{equation}\lbl{e:1012}
144 \calR^{4} t^2+\calR^{3} t (60-192 t)+\calR^{2}\left(-2-52 t+88 t^2\right)
+\calR \left(1+15 t-16 t^2\right)-\left(-1+2 t-t^2\right)=0.
\end{equation}
We are interested here in the branch which at $0$ is $1$.  The singularity 
points of $\calR$ are at the zeros of the discriminant.  These are  in our 
case the roots of 
$$
-11 - 128 t + 41088 t^2 - 20480 t^3 + 4096 t^4.
$$   
The solution we are interested in is the only solution $t_{0}$ in $(0,1)$.   
Approximately, $t_{0}\approx 0.0180827901\dots$.   The value of 
$\calR_{0}=\calR(t_{0})$ can be found in terms of $t_{0}$ as 
\[
\calR_{0}=\frac{1}{11}+\frac{856 t_{0}}{11}-\frac{1280 t_{0}^2}{33}
+\frac{256 t_{0}^3}{33}.
\] 
Using Newton's method described in \cite[VII 7]{FS} one can see that 
the singularity of $\calR$ is of square root near $t_{0}$.  To find the 
expansion, write 
\[
\calR(t)=\calR_{0}+\sum_{k=1}^{M}a_{k}\left(1-\frac{t}{t_{0}}\right)^{k/2}
\]
where $M$ is the desired level of approximation.  Plug this into 
Theorem \ref{thm.Ff}, 
expand everything near $t_{0}$, match the coefficients  and then  solve 
the system thus obtained for $a_{k}$.  In our case we can take for simplicity 
$M=3$ and solve for $a_{1},a_{2},a_{3}$.  The system in this case is of the 
form
\[
\begin{cases}
a_{1}^{2}v_{13}+v_{21}=0 & \\
2a_{2}v_{13}+a_{1}^{2}v_{14}+v_{22}=0 & \\
2a_{3} a_{1} v_{13}+a_{2}^{2}v_{13} +3 a_{1}^2 a_{2} v_{14}+a_{1}^4 v_{15}+a_{2} 
v_{22}+a_{1}^2 v_{23}+v_{31}=0\\
2 a_{4}a_{1} v_{13} + 2 a_{2} a_{3}v_{13} +  
  3 a_{1} (a_{2}^2 + a_{1} a_{3}) v_{14} + 4 a_{1}^3 a_{2} v_{15} + 
  a_{3} v_{22} + 2 a_{1} a_{2} v_{23} + a_{1}^3 v_{24} + 
  a_{1} v_{32} =0, &\\ 
2  a_{5}a_{1}v_{13} + (a_{3}^2 + 2 a_{2} a_{4}) v_{13} + (a_{2}^3 + 6 a_{1} a_{2} a_{3} 
+ 3 a_{1}^2 a_{4}) v_{14} + 
 6 a_{1}^2 a_{2}^2 v_{15} + 4 a_{1}^3 a_{3} v_{15} +  
 a_{4} v_{22} + &\\  
 \qquad\qquad+ a_{2}^2 v_{23} + 2 a_{1} a_{3} v_{23} + 
 3 a_{1}^2 a_{2} v_{24} + a_{1}^4 v_{25} + a_{2} v_{32} + 
 a_{1}^2 v_{33}=0
\end{cases}
\]
where the matrix $(v_{ij})_{i=1,3;j=1,5}$ with coefficients in  $\field{Q}(t_{0})$ 
is given in reduced form by
\[
\met{0 &0
& \frac{17+64 t_{0}-64 t_{0}^2}{8} & 72 t_{0} & 144 t_{0}^2 \\
 \frac{-3595+23184 t_{0}-10944 t_{0}^2+2048 t_{0}^3}{1584} & \frac{23-10903 
t_{0}+5440 t_{0}^2-1088 t_{0}^3}{33}  & -9-16 t_{0}+32 t_{0}^2 & -84 t_{0} & 
-288 t_{0}^2 \\
 \frac{1769-21424 t_{0}+12352 t_{0}^2-2048 t_{0}^3}{17424} & \frac{-2+631 
t_{0}-320 t_{0}^2+64 t_{0}^3}{33} & \frac{3-64 t_{0}^2}{8}  & 12 t_{0} & 144 t_{0}^2}.
\]

There are two different solutions for $a_{1}$, a positive and a negative one.  
The appropriate one is the negative one in our situation because $R$ has only 
non-negative coefficients (see \cite{BDG} for a proof of this).  Once $a_{1}$ 
is solved,  the other coefficients are determined automatically in a unique 
way.  Also notice here that $a_{1}$ is a square root of a number in 
$\field{Q}(t_{0})$, and that all $a_{k}$ for even $k$ are in $\field{Q}(t_{0})$, 
while $a_{k}$ for $k$ odd are in $\field{Q}(t_{0})/a_{1}$.

Now given the expansion of $\calR$ near $t_{0}$, the rest follows as in the 
previous case.  Namely, we can find the expansion of $\calR'(t)/\calR(t)$ 
near $t_{0}$ and thus the asymptotics of the coefficients for 
$\calR'(t)/\calR(t)$.  In turn, since $\calF_{0}(t)/t^{2}$ differentiated 
three times is exactly $\calR'(t)/\calR(t)$, the proof of \eqref{e:2007} 
follows straightforwardly.

Worth mentioning is the fact that the constant $C$ from \eqref{e:2007} is 
$C=-\frac{a_{1}}{2\sqrt{\pi } R_{0}}$, which explains the square root 
expression of $C$, while the other constants $d_{1}$ and $d_{2}$ are in 
$\field{Q}(t_{0})$.  
\qedhere
\end{proof}

\begin{remark} 
\lbl{rem.ascount}
\begin{enumerate}
\item The expansion in \eqref{e:2007} can be improved to 
\[
f_{n}\sim C \frac{(1/t_{0})^{n}}{n^{7/2}}\left( 1+\sum_{l=1}^{M}\frac{d_{l}}{n^{l}}
+O\left(\frac{1}{n^{M+1}}\right) \right),
\]
where $C$ is the one from \eqref{e:2007} and the constants $d_{n}$ are 
actually in $\field{Q}(t_{0})$.   
\item $\calF_{0}'''$ is an algebraic function and this determines the 
Stokes constants in the previous result.  However the algebraic equation 
is very long and this is the reason for not including it here.   In 
addition, $\calF_{0}$ also is the solution to some algebraic equation, 
though this is very long either.    The differential equation satisfied by 
$\calF_{0}$ implies a recurrence relation for the coefficients $(f_{n})$ which 
is again very long, thus not included.  
\end{enumerate}
\end{remark}

%%%%%%%%%%%%%%%%%%%%%%%%%%%%%%%%%%%%%%%%%%%%%%%%%%%%%%%%%%%%%%%%%%%%%%%%%%%%%
%%%%%%%%%%%%%%%%%%%%%%%%%%%%%%%%%%%%%%%%%%%%%%%%%%%%%%%%%%%%%%%%%%%%%%%%%%%%%

\section{Other examples of planar limits}
\lbl{sec:6}

Among other computations we mention here the case of counting planar 
diagrams with vertices of valences $3$ or $4$.   This corresponds to 
the case of potentials given by 
\[
\calV_{1}(x)=\frac{x^{2}}{2}-t^{3/2}\frac{x^{3}}{3}-t^{2}\frac{x^{4}}{4} 
\]
for the counting of diagrams with a fixed number of edges.  The problem of 
counting planar diagrams with a fixed number of faces corresponds to the 
potential 
\[
\calV_{2}(x)=\frac{x^{2}}{2}-t^{1/2}\frac{x^{3}}{3}-t \frac{x^{4}}{4}.
\]
The calculations are very similar to the ones for the extreme potentials in 
Sections \ref{s:14} and \ref{s:15}.  The results are as follows.  For 
$\calV_{1}$, the asymptotics of the coefficients $f_{n}$ of $F_{0}$ are given by 
\[
f_{n}=C \frac{(1/t_{0})^{n}}{n^{7/2}}\left(1+\frac{d_{1}}{n}+\frac{d_{2}}{n^{2}}
+O\left(\frac{1}{n^{3}}\right)\right)
\]
where $t=t_{0}$ is the closest root to $1/5$ of the polynomial equation
\begin{eqnarray*}
0&=& 6912-13824 t-146592 t^2-239488 t^3-2602569 t^4-4300752 t^5+79091888 t^6
\\ & & 
+304167552 t^7+410284704 t^8-1349207040 t^9 -7615156224 t^{10} \\ & & 
-4603041792 t^{11}+31506516736 t^{12},
\end{eqnarray*}
and $C,d_{1},d_{2}\in\field{Q}(t_{0})$ are given numerically as 
\begin{align*}
t_{0}&=0.2094195368\dots \\
1/t_{0}&=4.7751036758\dots\\
C&=1.4826787729\dots \\
d_{1}&=-7.2166440681\dots \\
d_{2}&=37.5616277128\dots.
\end{align*}
Similarly for the potential $\calV_{2}$ we have 
\[
f_{n}= C \frac{(1/t_{0})^{n}}{n^{7/2}}\left(1+\frac{d_{1}}{n}
+O\left(\frac{1}{n^{2}}\right)\right)
\]
where $t_{0}$ is the closest root to $0.023$ of the polynomial equation
\begin{eqnarray*}
0&=& -43625 - 614400 t + 89812992 t^2 + 895478272 t^3 - 3041722368 t^4 - 
 11466178560 t^5 \\ & & + 32248627200 t^6.
\end{eqnarray*}
In addition,  $C$, and $d_{1}\in\field{Q}(t_{0})$ are numerically approximated 
as
\begin{align*}
t_{0}&=0.02305646139\dots \\
1/t_{0}&=43.3717899396\dots\\
C&=0.2023938212\dots \\
d_{1}&=-3.2617202693\dots \\
\end{align*}
In both cases one can compute the asymptotics of the planar limit 
in the form
\[
f_{n}= C \frac{(1/t_{0})^{n}}{n^{7/2}}\left(1+\sum_{p=1}^{M}\frac{d_{p}}{n^{p}}
+O\left(\frac{1}{n^{M+1}}\right)\right)
\]
for any $M\ge1$.

%%%%%%%%%%%%%%%%%%%%%%%%%%%%%%%%%%%%%%%%%%%%%%%%%%%%%%%%%%%%%%%%%%%%%%%%%%%%%
%%%%%%%%%%%%%%%%%%%%%%%%%%%%%%%%%%%%%%%%%%%%%%%%%%%%%%%%%%%%%%%%%%%%%%%%%%%%%

\section{Analyticity of the Planar Limit}
\lbl{s:16}

In this section we prove Theorem~\ref{thm.1} and some 
consequences. Let us introduce some notation. For a given sequence 
$\mathbf{a}=\{ a_{n}\}_{n\ge1}$ in one of the spaces $\ell_{r}^{1}(\N)$, define 
\[
\alpha(\mathbf{a})=\sup_{n\ge1}|a_{n}|^{1/n}.  
\]

\begin{theorem}
\lbl{thm.s16}
\begin{enumerate} 
\item For even potentials
\[
\calV(x)=\frac{x^{2}}{2}-\sum_{n\ge 1}\frac{a_{2n}x^{2n}}{2n},
\]
if $\alpha(\mathbf{a})<\sqrt{8}$, then the planar limit 
$\calF^{\ev}_{0}(\mathbf{a})$ is absolutely convergent as a power series 
in infinitely many variables.  In particular $\calF_{0}$ is analytic on 
$B^{\ev}_{1/\sqrt{8}}$.  
\item 
For the potential
\[
\calV(x)=\frac{x^{2}}{2}-\sum_{n\ge 1}\frac{a_{n}x^{n}}{n}
\]
if $\alpha(\mathbf{a})<\sqrt{12}$, then $\calF_{0}(\mathbf{a})$ is an 
absolutely convergent series in infinitely many variables.   In particular  
$\calF_{0}$ is analytic on $B_{1/\sqrt{12}}$.  
\end{enumerate}
\end{theorem}

\begin{proof} 
We can write
\begin{eqnarray*}
\calF_{0}^{\ev}(\mathbf{a})&=&
\sum_{n=1}^\infty\left( \sum_{\substack{\l \vdash 2n\\ \l\text{ has only even blocks}}} c_{\l} 
a_{\l}\right)
\\
\calF_{0}(\mathbf{a})&=&\sum_{n=1}^\infty\left( \sum_{\l \vdash 2n} c_{\l} a_{\l}
\right)
\end{eqnarray*}
where the inner  sum is over partitions of size $2n$.   Note  that 
$c_{\l}\ge0$.   Now if $|a_{n}|\le r^{n/2}$, then 
$$
\sum_{\substack{\l \vdash 2n\\ \l\text{ has only even blocks}}} c_{\l} |a_{\l}| \le r^{n}
\sum_{\substack{\l \vdash 2n\\ \l\text{ has only even blocks}}} c_{\l} $$
and
$$ 
\sum_{\l \vdash 2n} c_{\l} |a_{\l}|\le r^{n} \sum_{\l \vdash 2n} c_{\l}.
$$
Hence, in order to compute the radius of convergence we need to 
compute the radius of convergence of the planar limit $\calF_{0}$ for the case of $a_{n}=r^{n/2}$.  
 Similarly, for the radius of convergence of  $\calF_{0}^{ev}$ it suffices to look at the case $a_{2n}=r^{n}$  and $a_{2n+1}=0$.  

Now,  in these particular cases, according to Proposition~\ref{c:22}, the radius of convergence of 
$\calF^{\ev}_{0}$  is $1/8$  while the one of $\calF_{0}$ is $1/12$.  In fact, it is easy to see that   
the coefficient $f_{n}$ of $t^{n}$ in $\calF_{0}^{\ev}(t)$ satisfies $f_{n}\le 8^{n}$ 
and the coefficient $f_{n}$ of $t^{n}$ in $\calF_{0}(t)$ satisfies 
$f_{n}\le 12^{n}$.   Consequently, we have that 
\[
\sum_{\substack{\l \vdash 2n
\\ 
\l\text{ has only even blocks}}}c_{\l}\le 8^{n} 
\]
 and 
\[
\sum_{\l \vdash 2n} c_{\l}\le 12^{n}.
\]
Therefore,
\[
c_{\l}\le 8^{n} \text{ for any partition }\l\text{ of size }2n\text{ 
with only even blocks}
\]   
and in general 
\[
c_{\l}\le  12^{n} \text{ for any partition }\l\text{ of size }2n.
\]
Now, a celebrated Hardy and Ramanujan (1918) result shows that the number of 
partitions of size $k$ is asymptotically $\frac{1}{4k\sqrt{3}}e^{\pi 
\sqrt{2k/3}}(1+o(1))$ (\cite{An}).  From this it follows easily that the series 
$\calF^{\ev}_{0}$ converges for any $r<1/8$ and $\calF_{0}^{ev}$ converges for any 
$r<1/12$ which concludes the proof.  
\qedhere
\end{proof}

Given a power series in the form
\[
\calV(x)=\frac{x^{2}}{2}-\sum_{n\ge 1}\frac{a_{n}x^{n}}{n}
\]
we set 
\[
\alpha(\calV)=\sup_{n\ge 1}|a_{n}|^{1/n}.  
\]
It is clear that $\calV$ is analytic near $0$ if and only if 
$\alpha(\calV)<\infty$.    With this notation we have the following 
corollary which confirms t'Hooft's conjecture.  For the following 
statement we denote the planar limit $\calF^{\ev}_{0}(t)$ and $\calF_{0}(t)$ to be 
the planar limits obtained by replacing $a_{n}$ by $t^{n/2}a_{n}$.

\begin{corollary}\lbl{c:100}
If $\calV$ is even then $\calF_{0}^{\ev}(t)$ has radius of convergence at 
least $\frac{1}{8\sqrt{\alpha(\calV)}}$.   

For arbitrary potentials $\calV$, $\calF_{0}(t)$ has radius of convergence 
$\frac{1}{12\sqrt{\alpha(\calV)}}$.

Both of these bounds are sharp.  
\end{corollary}

The fact that these bounds are sharp, follow from Proposition~\ref{p:14} or Proposition~\ref{c:22}.  
As made clear from the examples in Proposition~\ref{p:14b}, for the same 
$\alpha(\calV)$, the radius of convergence can be larger than the one given 
in this Corollary.

\begin{remark}
\lbl{rem.infinitev}
The analyticity of $\calF_{0}$ and $\calF_{0}^{\ev}$ in infinitely many 
variables can be deduced also from the perturbation result in 
Theorem~\ref{t:4}, though without any estimate on the radius of convergence.  
\end{remark}

\begin{remark} 
The reader might wonder what happens with the planar limit if instead of 
considering the potentials $\calV(x)=\frac{x^{2}}{2}
-\sum_{n\ge1}\frac{a_{n}x^{n}}{n}$ we consider the potentials 
$\calV(x)=\frac{x^{2}}{2}-\sum_{n\ge1}a_{n}x^{n}$.  In this case the 
extreme potentials are given by the case where  $a_{n}=t^{n/2}$ and this is 
$\calV(x)=x^{2}/2-1/(1-x\sqrt{t})$.  It turns out after some analysis that 
the radius of convergence of $\calF_{0}(t)$ in this case is given by $t_{0}$, 
which is the only solution in $(0,1)$ of the polynomial equation 
\begin{eqnarray*}
0&=& -226492416+962592768 t+34574598144 t^2+334387408896 t^3+7450906184352 t^4
\\ & & +21095006644064 t^5 
+130097822364531 t^6+55792303752096 t^7+67902575063040 t^8
\\ & & +19100742451200 t^9+6115295232000 t^{10}.
\end{eqnarray*}
Numerically this is approximately $t_{0}=0.04955391\dots$.   In this case, 
the planar limit as a function of the coefficients $a_{n}$ is 
an analytic function on $B_{\sqrt{t_{0}}}$.  
\end{remark}

\begin{remark}
It would be interesting to know what happens with the case of sequences $a_{n}$ which are not in $\ell^{1}_{r}$.   Apparently the $\ell^{1}_{r}$ is important for the well definition of convergent geometric series in infinitely many variables.
\end{remark}

%%%%%%%%%%%%%%%%%%%%%%%%%%%%%%%%%%%%%%%%%%%%%%%%%%%%%%%%%%%%%%%%%%%%%%%%%%%%%
%%%%%%%%%%%%%%%%%%%%%%%%%%%%%%%%%%%%%%%%%%%%%%%%%%%%%%%%%%%%%%%%%%%%%%%%%%%%%

\section{Perturbation Theory}
\lbl{s:pert}

The main result of this section is a stability result.  It says that given 
a potential whose equilibrium measure is one interval, then, under some 
non-degeneracy assumptions, any small perturbation preserves the one 
interval support of the equilibrium measure and in addition, the planar 
limit depends nicely on the perturbation.  

Before we state the result, we want to define a class of 
perturbations of a given potential $V$.  This definition is long and 
depends on many parameters, however the idea is quite simple.  We want 
to take perturbations of $V$ so that the maximizer of the function $F$ 
can be parametrized in a nice way.  The reasonable way of doing this is 
to have perturbations close to $V$ on some open interval containing the 
support of $\mu_{V}$ and large outside this open interval.  

For this purpose, assume that $\mathbf{X}$ is a Banach space over the 
reals which will be the ambient space of the parametrization.  Now, given 
an open subset $\mathbf{D}$ of $\mathbf{X}$ such that $\mathbf{0}\in 
\mathbf{D}$, and $I,J$ open sets of $\R$, an integer $k\ge1$ and 
$R,\delta>0$, we define $\mathcal{U}(k,V,\mathbf{D},I,J,R,\delta)$ 
the class of functions $\mathbf{V}:\mathbf{D}\times \R\to\R$ in two 
variables, with the properties,
\begin{equation}\lbl{e:52}
\begin{split}
& 1)\quad\mathbf{V}(\mathbf{0},x)=V(x)\\
& 2)\quad\text{for each}\:\:t\in\mathbf{D},\:\:x\to
\mathbf{V}(\mathbf{t},x)\:\:\text{satisfies}\:\: \eqref{e:V} \\
& 3)\quad (\mathbf{t},x)\to\mathbf{V}(\mathbf{t},x)\:\:\text{is}\:\:C^{k,3}(
\mathbf{D}\times \R) \\
& 4)\quad \sup_{\mathbf{t}\in\mathbf{D},x\in I} |\mathbf{V}(\mathbf{t},x)
-V(x)|<\delta,\quad  \& \quad \sup_{\mathbf{t}\in\mathbf{D}, x\in J}\|
\mathrm{Hess}_{x}(\mathbf{V}(\mathbf{t},\cdot)-V(\cdot))\|_{HS}<\delta,
 \\
& 5)\quad \inf_{\mathbf{t}\in\mathbf{D},x\notin I} 
(\mathbf{V}(t,x)-2\log |x|)\ge R
\end{split}
\end{equation}
where, $C^{k,3}$ stands for the set of jointly differentiable functions 
in $(\mathbf{t},x)$ with $k$ continuous (Fr\'echet) differentials in 
$\mathbf{t}$ and three continuous derivatives in $x$.  Also, 
$\mathrm{Hess}_{x}$ stands for the Hessian with respect to the variable 
$x$ and $\|\cdot \|_{HS}$ is the Hilbert-Schmidt norm.

In words, 1), 2) and   3) of \eqref{e:52} define the perturbation which 
is assumed of class $C^{k,3}$, while 4) means that the perturbation is 
uniformly close to  $V$ on $\mathbf{D}\times I$ while the Hessians are 
uniformly closed on $\mathbf{D}\times J$ and 5) encodes the fact that 
outside the interval $I$, the perturbation (minus the logarithmic term) 
is larger than a constant $R$ uniformly in the parameter 
$\mathbf{t}\in\mathbf{D}$.   We introduce here the interval $J$ because 
as we will see below in the proof of Theorem~\ref{t:4}, we only need the 
Hessians close for $x$ on a neighborhood of the support of $\mu_{V}$.  

The reason of introducing condition 5) in \eqref{e:52} instead of 
condition 4) with $I=\R$ is because for large values of $x$, we do not need 
the perturbation to be close to $V$.  We only need the perturbation to be 
large for large $x$. 
Actually,   4) and 5) constitute a weakening of the condition that the 
perturbation stays close to $V$ uniformly on the whole $\R$.  

Recall that we set 
\[
\psi_{c,b}(x)=\int_{-2}^{2}\frac{(V'(cx+b)-V'(cy+b))dy}{(x-y)\pi \sqrt{4-y^{2}}} \quad \forall x\in[-2,2].
\]

\begin{theorem}
\lbl{t:4}  
Assume that $V$ is a $C^{3}$ potential 
satisfying \eqref{e:V} with $H(c,b)$ and $\psi_{c,b}$ defined by 
\eqref{e:31} and \eqref{e:50} respectively.  Suppose that the following 
conditions hold true:
\begin{equation}\lbl{e:51}
\begin{split}
&1. \quad(c,b) \text{ is the unique maximizer of } H; \\
& 2.\quad\psi_{c,b}(x)>0, \:\:\text{for all}\:\: x\in[-2,2].
\end{split}
\end{equation}
Under these assumptions,  there exist
\begin{itemize} 
\item an interval $I\subset \R$, 
\item positive numbers $R_{0}$ and  $\delta_{0}$
\end{itemize} with the property that for any choice of 
\begin{itemize}
\item $R>R_{0}$,  $0<\delta<\delta_{0}$
\item an open neighborhood $J$ of $[-2c+b,2c+b]$,  
\item a  Banach space $\mathbf{X}$ 
\item and $\mathbf{V}\in\mathcal{U}(k,V,\mathbf{D},I,J, R,\delta)$, 
\end{itemize}
the following hold
\begin{equation}
\begin{split}
& 1)\quad \text{there exists an open }\mathbf{D}_{0}
\subset\mathbf{D} \:\:\text{with}\:\: \mathbf{0}\in
\mathbf{D}_{0}\:\:\text{and}\\
& 2) \quad (c,b):\mathbf{D}_{0}\to(0,\infty)\times\R\text{ 
which is }C^{k}\text{ such that }c(\mathbf{0})=c,\:\:b(\mathbf{0})=b\\
& 3) \quad (c(\mathbf{t}),b(\mathbf{t}))\:\:\text{is the unique 
maximizer of}\:\:H(\mathbf{t},\cdot)\:\:(\text{ defined 
by \eqref{e:31} for}\:\:\mathbf{V}(\mathbf{t},\cdot))\\ 
&4) \quad \mathbf{D}_{\mathbf{0}}\times[-2,2]\ni(\mathbf{t},x)\to
\psi_{c(\mathbf{t}),b(\mathbf{t})}(x)\in\R\:\:\text{is positive}.
\end{split}
\end{equation}
Furthermore, the equilibrium measure for $V(\mathbf{t},\cdot)$ has a 
single interval support for $\mathbf{t}\in\mathbf{D}_{0}$ and the 
planar limit $F_{0,\mathbf{t}}=F_{0,\mathbf{V}(\mathbf{t})}$ is a $C^{k}$ 
function on $\mathbf{D}_{0}$. 

In addition, if $\mathbf{X}$ is either a finite dimensional space or of the 
form $\mathbf{X}=\{ (a_{n})_{n\ge1}\subset \R: 
\sum_{n\ge1}|a_{n}|r^{n}<\infty\}$ for some $r>0$ and $V$ is real 
analytic on a neighborhood of the support of $\mu_{V}$  such that  
$(\mathbf{t},x)\to\mathbf{V}(\mathbf{t},x)$ is real analytic  on a  
neighborhood of $\mathbf{0}\times [-2c+b,2c+b]$, then, we can take  
$\mathbf{D}_{0}$ so that  $\mathbf{t}\to c(\mathbf{t})$,  $\mathbf{t}\to 
b(\mathbf{t})$ and $\mathbf{t}\to F_{0,\mathbf{t}}$ are real analytic functions.
\end{theorem}

\begin{proof}  
The key point of the proof is the fact that the maximizer 
$(c,b)$ of $H$ is unique and isolated and then by perturbing a little bit
 the potential $V$, the maximizer of $H(\mathbf{t},\cdot)$ is to be found 
near $(c,b)$.  Finding the maximizer $(c(\mathbf{t}),b(\mathbf{t}))$ of 
$H(\mathbf{t},\cdot)$ boils down to finding the critical point of this 
function near $(c,b)$.  This can be achieved by the implicit function 
theorem and the fact that the Hessian of $H$ is non-degenerate near $(c,b)$.  

Now technicalities.   The first thing we want to do is to prove that for the 
unperturbed function $H$, $(c,b)$ is a non-degenerate critical point.  To do 
this we want to check that the Hessian of $H$ at $(c,b)$ is positive definite.  
For simplicity of the discussion, we will assume without any loss of generality 
that $c=1$ and $b=0$.  Now the non-degeneracy is equivalent to the fact that
\begin{equation}\lbl{e:1000}
\mat{2+\int_{-2}^{2}\frac{x^{2}V''(x)dx}{\pi \sqrt{4-x^{2}}} & \int_{-2}^{2}\frac{xV''(x)dx}{\pi \sqrt{4-x^{2}}} \\ \int_{-2}^{2}\frac{xV''(x)dx}{\pi \sqrt{4-x^{2}}} &\int_{-2}^{2}\frac{V''(x)dx}{\pi \sqrt{4-x^{2}}}}
\end{equation}
is positive definite.  

Recall that the critical point equations give
\[
\int_{-2}^{2}\frac{xV'(x)dx}{\pi \sqrt{4-x^{2}}}=2 \text{ and } \int_{-2}^{2}\frac{V'(x)dx}{\pi \sqrt{4-x^{2}}}=0.  
\]
Integrating by parts the first of these one deduces that 
\[
2+\int_{-2}^{2}\frac{x^{2}V''(x)dx}{\pi \sqrt{4-x^{2}}}=4\int_{-2}^{2}\frac{V''(x)dx}{\pi \sqrt{4-x^{2}}}.
\]
Armed with this, the non-degeneracy of the Hessian \eqref{e:1000} follows once we prove the following 
\begin{equation}\lbl{e:1001}
\int_{-2}^{2}\frac{(2\pm x)V''(x)dx}{\pi \sqrt{4-x^{2}}}>0.
\end{equation}
This follows from  
\begin{eqnarray*}
\int_{-2}^{2}\frac{(2\pm x)V''(x)dx}{\pi \sqrt{4-x^{2}}}
&=&
\int_{-2}^{2}\frac{d}{dx}(V'(x)-V'(\pm2))\frac{(2\pm x)dx}{\pi \sqrt{4-x^{2}}}
\\ &=&
\int_{-2}^{2}\frac{(V'(\pm2)-V(x))dx}{(\pm2-x)\pi \sqrt{4-x^{2}}}=\psi(\pm 2)>0.
\end{eqnarray*}

Let $M=H(c,b)$ be the maximum of $H$.  For any  
choice of $\epsilon,r>0$ with 
$r>\epsilon>0$, obviously one has
\begin{equation}\lbl{e:53}
\sup\{ H(u,v)\: : \: u>0,v\in\R,r^{2}>(u-c)^{2}+(v-b)^{2}\ge\epsilon^{2} \}<M.
\end{equation}
Indeed if this is not the case, then there is a sequence $(c_{n},b_{n})$ 
such that $r^{2}>(u_{n}-c)^{2}+(v_{n}-b)^{2}>\epsilon^{2}$ so that 
$\lim_{n\to\infty}H(u_{n},v_{n})=M$.  Passing eventually on subsequences, 
we may assume that $u_{n}$ and $v_{n}$ converge to  $u$ and $v$.   Clearly 
$u\ne0$, otherwise $H(u,v)=-\infty$.   This implies that $H(u,v)=M$ and 
at the 
same time, $(u,v)$ is within positive distance from $(c,b)$, hence 
contradicting the uniqueness of the maximizer.    

Next, consider
\[
U(x):=V(x)-2\log|x|,  
\]
and notice that from \eqref{e:V} we have $U(x)\ge C>-\infty$ and 
$\lim_{|x|\to\infty}U(x)=+\infty$.   Now we use \eqref{e:43} to justify that
\[
H(u,v)=-\int_{-2}^{2}\frac{U(ux+v)dx}{2\pi\sqrt{4-x^{2}}}-\int_{-2}^{2}
\frac{\log|x+v/u|}{2\pi\sqrt{4-x^{2}}}\le -\int_{-2}^{2}
\frac{U(ux+v)dx}{2\pi\sqrt{4-x^{2}}}.
\]

Assuming that  $(u-c)^{2}+(v-b)^{2}\ge r^{2}$, it is easy to deduce that, 
$u+|v|\ge \sqrt{r^{2}/2-c^{2}-b^{2}} $,
 and thus, 
\begin{eqnarray}
\lbl{e:55}
H(u,v) &\leq &  
-\int_{-1}^{1}\frac{U(ux+v)dx}{2\pi\sqrt{4-x^{2}}}-\int_{1}^{2}
\frac{U(ux+v)dx}{2\pi\sqrt{4-x^{2}}}-\int_{-2}^{-1}
\frac{U(ux+v)dx}{2\pi\sqrt{4-x^{2}}} \\ \notag
& \leq & -C/3 - h(\sqrt{r^{2}/2-c^{2}-b^{2}} )
\end{eqnarray}
where $h(x)=\inf_{|y|\ge x} U(y)/6$.  In particular, for large $r$ we learn 
that $H(u,v)<M$.

Equations \eqref{e:53} and \eqref{e:55} guarantee that for any
 $\epsilon>0$, there exists $\delta_{0}\in(0,1)$ such that 
\begin{equation}\lbl{e:56}
H(u,v)<M-3\delta_{0}
\end{equation}
for all $(u,v)$ outside a ball of radius $\epsilon$ around $(c,b)$.   
We take $R_{0}>0$ such that $|x|\ge R_{0}$ implies $h(x)>-C/3-M+3$ and 
define $I=[-R_{0},R_{0}] \cup[-2c-1+b,2c+1+b]$.  The purpose of this 
choice of $I$  is to make it a neighborhood of the support of $\mu_{V}$.     

With these choices, for any $R>R_{0}$, $0<\delta<\delta_{0}$ and 
$\mathbf{V}\in\mathcal{U}(k,V,\mathbf{D},I,J, R,\delta)$, from the 
conditions 4) and 5) of \eqref{e:52}, and the reasoning which led to 
\eqref{e:55},  one gets for $r=\sqrt{2(R_{0}^{2}+b^{2}+c^{2})}$ that
\[
\begin{split}
&1.\quad|H(\mathbf{t},u,v)-H(u,v)|<\delta,\quad\text{for}  
\quad r^{2}>(u-c)^{2}+(v-b)^{2}>\epsilon^{2} \\
&2.\quad H(\mathbf{t},u,v)<M-3\quad \text{for}\quad (u-c)^{2}+(v-b)^{2}>r^{2}.
\end{split}
\]

We are led to the conclusion that for all $\mathbf{t}\in\mathbf{D}$,  
$\max_{u>0,v\in\R}\{H(\mathbf{t},u,v) \}$ is attained for $(u,v)$ in the 
ball of radius $\epsilon$ around $(c,b)$.   Indeed, otherwise, assume that 
there is a maximizer $(u,v)$ outside the ball $B_{\epsilon}(c,b)$.   Since,  
$|H(\mathbf{t},u,v)-H(u,v)|<\delta$,  combined with 
\eqref{e:56}, implies that $H(\mathbf{t},u,v)<M-2\delta$.  This 
contradicts 4) of equation \eqref{e:52} from which we gather that 
$H(\mathbf{t},c,b)-H(c,b)>-\delta$, or $H(\mathbf{t},c,b)>M-\delta>
H(\mathbf{t},u,v)+\delta$, thus $(u,v)$ can not be a maximizer of 
$H(\mathbf{t},\cdot)$.
   
 The maximizer is a critical point of $H(\mathbf{t},u,v)$, therefore 
$\nabla_{u,v}H(\mathbf{t},\cdot)=0$. To solve for $(u,v)$,  we interpret 
it as the definition of an implicit function $\mathbf{t}\to(c(\mathbf{t}),
b(\mathbf{t}))$.    This can be done thanks to the combination of the 
last  part of 4) of \eqref{e:52},  1) of \eqref{e:51} and the implicit 
function theorem.  These yield for a set $\mathbf{D}_{0}\subset \mathbf{D}$, 
which contains $\mathbf{0}$ that there exists a $C^{k}$ function 
$\mathbf{t}\to(c(\mathbf{t}),b(\mathbf{t}))$ which is the maximizer of 
$H(\mathbf{t},\cdot)$. Taking a smaller subset of $\mathbf{D}_{0}$, it 
is easy to show that $\psi_{c(\mathbf{t}),b(\mathbf{t})}>0$ on $[-2,2]$ 
and the $C^{k}$ dependence of $F_{0,\mathbf{t}}$ on $\mathbf{t}$ is  a 
simple consequence of \eqref{e:28}. 
 
 In the case of analytic perturbations with $\mathbf{X}$ a finite 
dimensional space,  the only thing we need to point out is that (cf. 
\cite{Kr}) the implicit function theorem produces analytic versions 
$c(\mathbf{t})$ and $b(\mathbf{t})$ for $\mathbf{t}$ in an eventually 
smaller $\mathbf{D}_{0}$.  The analyticity of $F_{0,\mathbf{t}}$ follows 
from \eqref{e:28}.   
 
 On the other hand in the case $\mathbf{X}=\{ (a_{n})_{n\ge1}\subset \R: 
\sum_{n\ge1}|a_{n}|r^{n}<\infty\}$, one needs a bit more work.   
The analyticity of functions in infinitely many 
variables is trickier than the case of analytic functions in finitely many 
variables.   However, our space here is essentially $\ell^{1}(\N)$ over 
the real numbers and for this case many things are like in the finite 
dimensional cases.   
 
 What we mean here is that for the case of $\ell^{1}(\N)$ over the 
complex numbers, the theory of analytic functions is treated in \cite{Le} 
and \cite{Ry}.  The main results are that every holomorphic function on 
$\ell^{1}(\N)$ has a power series expansion and every absolutely convergent 
power series expansion defines a holomorphic function.   
 
 In our situation, the functions are real analytic (meaning they have a 
power series expansion), thus by complexification they become complex 
analytic and therefore they are holomorphic functions.   Then, for the 
complexification, we know that the implicit function theorem yields that the 
resulting functions $c(\mathbf{t})$, $b(\mathbf{t})$ and $F_{0,\mathbf{t}}$ 
all are smooth functions of $\mathbf{t}$ on a small neighborhood of 
$\mathbf{X}$.   Furthermore, since $F$ is actually a holomorphic function 
it is not hard to prove that the choices of $c(\mathbf{t})$, $b(\mathbf{t})$ 
and $F_{0,\mathbf{t}}$ can be made holomorphic.   Using this we can conclude 
that the real parts of $c(\mathbf{t})$, $b(\mathbf{t})$ and $F_{0,\mathbf{t}}$ 
are real analytic.    \qedhere
\end{proof}

%%%%%%%%%%%%%%%%%%%%%%%%%%%%%%%%%%%%%%%%%%%%%%%%%%%%%%%%%%%%%%%%%%%%%%%%%%%%%
%%%%%%%%%%%%%%%%%%%%%%%%%%%%%%%%%%%%%%%%%%%%%%%%%%%%%%%%%%%%%%%%%%%%%%%%%%%%%

\appendix

%%%%%%%%%%%%%%%%%%%%%%%%%%%%%%%%%%%%%%%%%%%%%%%%%%%%%%%%%%%%%%%%%%%%%%%%%%%%%
%%%%%%%%%%%%%%%%%%%%%%%%%%%%%%%%%%%%%%%%%%%%%%%%%%%%%%%%%%%%%%%%%%%%%%%%%%%%%

\section{The first few terms of $\calR$, $\calS$ and $\calF_0$}
\lbl{sec.fewterms}

In this appendix we give the first few terms of the unique solution
$(R,S) \in \calA$ of Equations \eqref{eq.RS} and also of the
formal planar limit $\calF_0$.
{\small
\begin{eqnarray*}
\calS
& =& a_1+a_1 a_2+2 a_3+a_1 a_2^2+a_1^2 a_3+4 a_2 a_3+6 a_1 a_4+6 a_5+a_1 a_2^3
+3 a_1^2 a_2 a_3+6 a_2^2 a_3+8 a_1 a_3^2 \\ & & +a_1^3 a_4 
+18 a_1 a_2 a_4+18 a_3 a_4+12 a_1^2 a_5+18 a_2 a_5+30 a_1 a_6+20 a_7
+a_1 a_2^4+6 a_1^2 a_2^2 a_3+8 a_2^3 a_3 \\ & & 
+2 a_1^3 a_3^2
+32 a_1 a_2 a_3^2+12 a_3^3+4 a_1^3 a_2 a_4+36 a_1 a_2^2 a_4+42 a_1^2 a_3 a_4
+72 a_2 a_3 a_4+54 a_1 a_4^2+a_1^4 a_5 \\ & & +48 a_1^2 a_2 a_5 
+36 a_2^2 a_5+108 a_1 a_3 a_5+72 a_4 a_5+20 a_1^3 a_6+120 a_1 a_2 a_6
+80 a_3 a_6+90 a_1^2 a_7 \\ & & 
+80 a_2 a_7+140 a_1 a_8+70 a_9+O(a^{11})
\end{eqnarray*}
\begin{eqnarray*}
\calR
&=&1+a_2+a_2^2+2 a_1 a_3+3 a_4+a_2^3+6 a_1 a_2 a_3+4 a_3^2+3 a_1^2 a_4+9 a_2 a_4
+12 a_1 a_5+10 a_6+a_2^4 \\ & & +12 a_1 a_2^2 a_3
+6 a_1^2 a_3^2+16 a_2 a_3^2+12 a_1^2 a_2 a_4+18 a_2^2 a_4+42 a_1 a_3 a_4
+18 a_4^2+4 a_1^3 a_5+48 a_1 a_2 a_5 \\ & &
+36 a_3 a_5+30 a_1^2 a_6
+40 a_2 a_6+60 a_1 a_7+35 a_8+a_2^5+20 a_1 a_2^3 a_3+30 a_1^2 a_2 a_3^2
+40 a_2^2 a_3^2+32 a_1 a_3^3 \\ & &
+30 a_1^2 a_2^2 a_4+30 a_2^3 a_4
+20 a_1^3 a_3 a_4+210 a_1 a_2 a_3 a_4+84 a_3^2 a_4+63 a_1^2 a_4^2
+90 a_2 a_4^2+20 a_1^3 a_2 a_5 \\ & &
+120 a_1 a_2^2 a_5+132 a_1^2 a_3 a_5 
+180 a_2 a_3 a_5+252 a_1 a_4 a_5+72 a_5^2+5 a_1^4 a_6+150 a_1^2 a_2 a_6
+100 a_2^2 a_6 \\ & &
+260 a_1 a_3 a_6+150 a_4 a_6+60 a_1^3 a_7
+300 a_1 a_2 a_7+160 a_3 a_7+210 a_1^2 a_8+175 a_2 a_8+280 a_1 a_9
\\ & &
+126 a_{10} +O(a^{11})
\end{eqnarray*}
\begin{eqnarray*}
\calF_{0}
&=& \frac{a_1^2}{2}+\frac{a_2}{2}+\frac{1}{2} a_1^2 a_2+\frac{a_2^2}{4}
+\frac{1}{2} a_1^2 a_2^2+\frac{a_2^3}{6}+\frac{1}{2} a_1^2 a_2^3
+\frac{a_2^4}{8}+\frac{1}{2} a_1^2 a_2^4+\frac{a_2^5}{10}+a_1 a_3
+\frac{1}{3} a_1^3 a_3 \\ & &
+2 a_1 a_2 a_3+a_1^3 a_2 a_3 
+3 a_1 a_2^2 a_3+2 a_1^3 a_2^2 a_3+4 a_1 a_2^3 a_3+\frac{2 a_3^2}{3}
+2 a_1^2 a_3^2+\frac{1}{2} a_1^4 a_3^2+2 a_2 a_3^2 \\ & & +8 a_1^2 a_2 a_3^2
+4 a_2^2 a_3^2+4 a_1 a_3^3+\frac{a_4}{2}
+\frac{3}{2} a_1^2 a_4+\frac{1}{4} a_1^4 a_4+a_2 a_4
+\frac{9}{2} a_1^2 a_2 a_4+a_1^4 a_2 a_4+\frac{3}{2} a_2^2 a_4
\\ & & +9 a_1^2 a_2^2 a_4+2 a_2^3 a_4+6 a_1 a_3 a_4+7 a_1^3 a_3 a_4
+24 a_1 a_2 a_3 a_4+6 a_3^2 a_4+\frac{9 a_4^2}{8}+9 a_1^2 a_4^2
+\frac{9}{2} a_2 a_4^2 \\ & & +2 a_1 a_5 +2 a_1^3 a_5+\frac{1}{5} a_1^5 a_5
+6 a_1 a_2 a_5+8 a_1^3 a_2 a_5 
+12 a_1 a_2^2 a_5+3 a_3 a_5+18 a_1^2 a_3 a_5+12 a_2 a_3 a_5
\\ & & +18 a_1 a_4 a_5+\frac{18 a_5^2}{5}+\frac{5 a_6}{6}+5 a_1^2 a_6
+\frac{5}{2} a_1^4 a_6+\frac{5 a_2 a_6}{2}
+20 a_1^2 a_2 a_6+5 a_2^2 a_6+20 a_1 a_3 a_6 \\ & & +6 a_4 a_6+5 a_1 a_7
+10 a_1^3 a_7+20 a_1 a_2 a_7+8 a_3 a_7+\frac{7 a_8}{4}+\frac{35}{2} a_1^2 a_8
+7 a_2 a_8+14 a_1 a_9\\ & & +\frac{21 a_{10}}{5}+O(a^{11})
\end{eqnarray*}
}
where each $a_{k}$ is given the degree $k$.  For example the monomial 
$a_{1}^{2}a_{2}^{2}a_{4}$ has degree $2\times 1+2\times 2+4=10$.   
The meaning of $O(a^{11})$ is that the degree of the remaining terms is 
at least $11$.

\bibliographystyle{hamsalpha}\bibliography{biblio.bib}
\end{document}